%% file: slither_final.tex
\documentclass[reqno]{amsart}

\usepackage[letterpaper,margin=1in]{geometry}

\usepackage[T1]{fontenc}
\usepackage[sc]{mathpazo}       
\usepackage[euler-digits]{eulervm} 
\usepackage{microtype}

\usepackage{
  amsmath,
  amsfonts,
  amssymb,
  amsthm,
  mathtools,
  mathrsfs,
  stmaryrd,
  wasysym
}

\usepackage{mathabx}

\usepackage{enumerate}

\usepackage{tikz}
\usepackage{tikz-cd}
\usepackage{standalone}

\usetikzlibrary{
  arrows.meta,
  decorations.markings,
  decorations.pathmorphing,
  patterns,
  patterns.meta
}
\usepackage[alphabetic]{amsrefs}

\usepackage[hidelinks]{hyperref}
\usepackage{cleveref}

\newcommand{\ella}{\rotatebox[origin=c]{20}{$\!\ell$}}
\newcommand{\ellb}{%
  \setbox0=\hbox{$1$}%
  \resizebox{\wd0}{\ht0}{$\ella$}%
}
\newcommand{\elll}{\hspace{0.04em}\ellb\hspace{-0.04em}}
\newcommand{\llle}{\hspace{-0.04em}\reflectbox{\rotatebox[origin=c]{180}{\reflectbox{$\elll$}}}\hspace{0.04em}}

\newcommand{\bl}{\bullet}
\newcommand{\Pb}{\mathbb{P}}
\newcommand{\A}{\mathbb{A}}
\newcommand{\ove}{\overline}
\newcommand{\Lamma}{\ensuremath{\mathrm{L}}}
\newcommand{\side}{\ensuremath{\vdash}}
\newcommand{\rel}{\mathbin{\mathrm{rel}}}
\newcommand{\zalg}{\ensuremath{\Z\mathrm{Alg}}}
\newcommand{\opp}{\mathrm{op}}
\newcommand{\ebs}{\ensuremath{\mathrm{EBS}}}
\newcommand{\U}[2]{\ensuremath{{}^{#2}#1}}
\newcommand{\D}[2]{\ensuremath{{}_{#2}#1}}

\makeatletter
\renewcommand{\l@subsection}{\@tocline{2}{0pt}{2.5em}{}{}}
\makeatother

\newtheorem{Thm}{Theorem}[subsection]

\newtheorem{Prop}[Thm]{Proposition}

\newtheorem{Cj}[Thm]{Conjecture}
\newtheorem*{cj}{Conjecture}
\newtheorem{Obs}[Thm]{Observation}
\newtheorem{Lm}[Thm]{Lemma}
\newtheorem{Cor}[Thm]{Corollary}

\newtheorem{IntroThm}{Theorem}

\newtheorem{IntroCor}[IntroThm]{Corollary}

\theoremstyle{definition}
\newtheorem{Df}[Thm]{Definition}

\theoremstyle{remark}
\newtheorem{Rmk}[Thm]{Remark}
\newtheorem{Ex}[Thm]{Example}

\input{bourbaki-warning}
\renewcommand{\warningsymbolheight}{2em}

\renewcommand{\theThm}{%
  \ifnum\value{subsection}=0
    \thesection.\arabic{Thm}%
  \else
    \thesubsection.\arabic{Thm}%
  \fi
}%

\numberwithin{figure}{subsection}
\renewcommand{\thefigure}{%
  \ifnum\value{subsection}=0
    \thesection.\arabic{figure}%
  \else
    \thesubsection.\arabic{figure}%
  \fi
}
\numberwithin{equation}{subsection}
\renewcommand*{\theequation}{%
  \ifnum\value{subsection}=0
    \thesection
  \else
    \thesubsection
  \fi
  .\arabic{equation}%
}

\tikzset{
  middle arrow/.style={
    decoration={markings,mark=at position 0.7 with {\arrow{>}}},
    postaction={decorate}
  },
  middle arrow chain/.style={
    decoration={
      markings,
      mark=at position 0.5 with {\arrow{>}},
      mark=at position 0.75 with {\arrow{>}}
    },
    postaction={decorate}
  }
}

\newcommand{\midarrow}{
  \;
  \begin{tikzpicture}
    \draw[middle arrow] (0,0) -- (0.4,0);
  \end{tikzpicture}
  \;
}

\newcommand{\middarrow}{
  \;
  \begin{tikzpicture}
    \draw[middle arrow chain] (0,0) -- (0.4,0);
  \end{tikzpicture}
  \;
}

\newcommand{\ltl}{\ensuremath{<\hspace{-0.45em}<}}

\newcommand{\inner}[1]{\left\langle #1 \right\rangle}

\newcommand{\pr}{\ensuremath{\langle}}
\newcommand{\rp}{\ensuremath{\rangle}}

\newcommand{\eps}{\epsilon}
\newcommand{\Gm}{\ensuremath{\mathbb{G}_m}}
\newcommand{\BB}{Bia\l ynicki-Birula }
\newcommand{\nsym}{\mathrm{NSym}}
\newcommand{\qsym}{\mathrm{QSym}}
\newcommand{\sym}{\mathrm{Sym}}
\newcommand{\abs}[1]{\ensuremath{\left\lvert #1 \right\rvert}}
\newcommand{\QGr}{QGr}

\newcommand{\R}{\mathbb{R}}
\newcommand{\Z}{\mathbb{Z}}
\newcommand{\C}{\mathbb{C}}
\newcommand{\GL}{GL}

\DeclareMathOperator{\wt}{wt}
\DeclareMathOperator{\Des}{Des}
\DeclareMathOperator{\Lin}{Lin}
\DeclareMathOperator{\SSYT}{SSYT}
\DeclareMathOperator{\deq}{\vcentcolon=}
\DeclareMathOperator{\Spec}{Spec}
\DeclareMathOperator{\Hom}{Hom}
\DeclareMathOperator{\Span}{span}

\DeclareMathOperator{\Cone}{Cone}

\DeclareMathOperator{\Div}{div}
\DeclareMathOperator{\Aut}{Aut}

\DeclareMathOperator{\Gr}{Gr}
\DeclareMathOperator{\rk}{\ensuremath{\mathrm{rk}}}
\DeclareMathOperator{\Id}{id}
\DeclareMathOperator{\Chim}{Chim}
\DeclareMathOperator{\sgn}{sign}
\DeclareMathOperator{\cl}{cl}

\DeclarePairedDelimiter\ide{(}{)}

\title{A quasisymmetric analog of Grassmannian Schubert varieties}
\author{Teddy Gonzales}
\address{Rutgers University}
\email{tg570@math.rutgers.edu}

\author{Tuong Le}
\address{Princeton University}
\email{tuongle@princeton.edu}

\author{Chayim Lowen}
\address{Princeton University}
\email{chayiml@princeton.edu}

\begin{document}
\begin{abstract}
We show that the cohomology rings of toric Richardson varieties in the Grassmannian are finite truncations of the ring of quasisymmetric functions.
We exhibit an affine paving of each such variety 
whose cell closures give rise to the basis of fundamental quasisymmetric functions. 
We similarly interpret the homology of these varieties in terms of the ring of noncommutative symmetric functions and show that the expansion of the homological class of any torus-invariant subvariety into the affine paving basis agrees with the expansion of a corresponding generalized noncommutative ribbon function into the ribbon basis. 
By taking the direct limit of all toric Richardson varieties, we obtain an ind-variety 
equipped with a weak $H$-group structure whose cohomology is the Hopf algebra of quasisymmetric functions.
We conjecture that it is isomorphic
to a similar $H$-group constructed by Baker--Richter. 

As a byproduct, we deduce that the $f$-vectors of shard polytopes are log-concave,
making the first progress on a question of Ferroni--Schr\"oter for matroid base polytopes.

\end{abstract}

\maketitle

\tableofcontents
\section{Introduction}

\subsection{Motivation} 
The theory of symmetric functions and Schur polynomials lies at the heart of algebraic combinatorics. It has important connections to the representation theory of the symmetric group $S_n$ and the general linear group $\GL_n$, as well as to Schubert calculus in the Grassmannian. Schur polynomials were first studied back in 1815 by Cauchy \cite{Cau15}, but it took until 1901 for the connection to the representation theory of $S_n$ and $\GL_n$ to be made explicit in Schur's thesis \cite{Sch01}. The connection to the geometry of the Grassmannian was pointed out even later, in 1947 by Lesieur \cite{Les47}, who realized that the geometric formula discovered much earlier by Pieri \cite{Pie93} and Giambelli \cite{Gia02} matches up with identities in symmetric functions.

In modern language, for a partition\footnote{In this paper, we consistently identify partitions with their Young diagram, drawn using the English convention.} $\lambda$, the Schubert variety $\Omega_\lambda$ in the Grassmannian (which is the Grassmannian itself in the special case of a rectangular partition) satisfies:
\begin{enumerate}
    \setcounter{enumi}{-1}
    \item $\Omega_\lambda$ is a normal projective variety with an affine stratification by cells indexed by partitions $\mu\subseteq \lambda.$
    \item Writing $A_{\bl}$ for Chow homology, there are isomorphisms of graded $\Z$-modules 
    \[
    A_{\bl}(\Omega_{\lambda}) \cong H_\bl(\Omega_\lambda, \Z) \cong \Span_{\Z} \{s_{\mu} \mid \mu \subseteq \lambda\} \subset \sym \] 
    under which the class of a Richardson subvariety maps to the corresponding skew Schur function. 
    Furthermore, the \emph{effective} classes in $A_{\bl}(\Omega_\lambda)$  correspond exactly to the Schur-positive elements.
    \item Writing $A^{\bl}$ for Chow cohomology, there are isomorphisms of graded rings
\[
A^{\bl}(\Omega_{\lambda}) \cong 
H^{\bl}(\Omega_{\lambda}, \Z) \cong \frac{\sym}{\ide{s_\nu\mid \nu\nsubseteq \lambda}}
\]
under which $s_{\mu}$ corresponds to the dual class of the closure of the affine cell indexed by $\mu \subseteq \lambda$.
\item A suitable direct limit of all Schubert varieties produces an ind-variety $\Gr = \Gr(\infty, \infty)$ (the infinite Grassmannian)
naturally equipped with an $H$-group\footnote{An $H$-group is by definition a group object in the homotopy category of topological spaces. In good cases, the integral (co)homology of such a space is naturally equipped with the structure of a Hopf algebra.} structure so that $H^\bl(\Gr, \Z)$ and $H_\bl(\Gr, \Z)$ are each isomorphic to $\sym$ as graded Hopf algebras.
\end{enumerate}
Part (1) gives a geometric interpretation of
the \emph{$\Z$-module} $\sym$: the combinatorially defined skew Schur functions expand into the Schur function basis in the same way that the homology classes of Richardson subvarieties expand into the Schubert 
basis.
Part (2) gives a geometric interpretation of 
the \emph{ring} $\sym$. Part (3) upgrades this to a geometric interpretation of 
the \emph{Hopf algebra} $\sym$.

In his seminal paper on multipartite $P$-partitions, Gessel \cite{Gessel1984} 
introduced the ring of quasisymmetric functions $\qsym$, an extension of the ring of symmetric functions. As an analogue of the Schur basis of $\sym$, Gessel introduced the basis of \emph{fundamental quasisymmetric functions} for $\qsym$. 
Gelfand, Krob, Lascoux, Leclerc, Retakh, and Thibon \cite{GKLLRT95} later introduced
the ring of 
noncommutative symmetric functions $\nsym$, along with its basis of 
\emph{noncommutative ribbon functions}.
Both $\qsym$ and $\nsym$ are graded Hopf algebras and the two are dual to one another \cite{MR95};
the fundamental quasisymmetric basis is dual to the noncommutative ribbon basis. Duchamp, Krob, Leclerc, and Thibon \cite{DKLT96} later gave a natural interpretation of the fundamental quasisymmetric functions as characters of irreducible representations of the $0$-Hecke algebra in much the same way that
Schur functions correspond to characters of irreducible $S_n$-representations.
A natural next step is to find a geometric realization of $\qsym$ and $\nsym$ with properties analogous to (0)--(3).

There has been some previous work 
addressing this question. 
Baker--Richter \cite{BR08} considered the loop space of the suspension of $\C \mathbb{P}^{\infty}$, denoted $\Omega \Sigma \C \mathbb{P}^\infty$, 
and proved that its cohomology is 
the Hopf algebra $\qsym$ (and so its homology is the Hopf algebra $\nsym$).
Pechenik--Satriano \cite{PS26} later showed that the homotopy equivalent space $J \C \mathbb{P}^\infty$---the James space of $\C \mathbb{P}^{\infty}$---admits a natural cell decomposition 
whose dual basis in cohomology is given by the monomial quasisymmetric functions, and also developed a  torus-equivariant version of the theory \cite{PS24}. Oesinghaus \cite{O19} showed that the integral Chow ring of the algebraic stack of expanded pairs described by Abramovich--Cadman--Fantechi--Wise \cite{ACFW13} (relying on prior work of Li \cites{L01, L02}) is the Hopf algebra $\qsym$. Bergeron--Gagnon--Spink--Tewari \cite{BGST26} constructed a reduced torus-invariant subscheme $\QGr(r, n)$ of the Grassmannian 
$\Gr(r, n)$ whose cohomology ring is a finite truncation of $\qsym$.

If we wish to develop a theory which parallels the case of $\sym$, none of these constructions are entirely satisfactory.
The space $\Omega \Sigma \C \mathbb{P}^\infty$,
while elegant in its simplicity, lacks an intrinsic algebro-geometric structure. 
The homotopy equivalent space $J \C \mathbb{P}^\infty$, which comes much closer, is not a \emph{normal} ind-variety, as shown by
Pechenik and Satriano themselves \cite{PS26}.
Also, the natural basis arising from this framework is the simpler monomial basis rather than the fundamental basis. 
Oesinghaus's model is a stack with a rather beautiful functor of points, but does not appear to be a scheme;
it similarly gives rise to the monomial basis. Bergeron--Gagnon--Spink--Tewari's model 
\emph{does} naturally witness the fundamental quasisymmetric basis but fails to be a normal variety because it is not irreducible.
Their work also does not give geometric meaning to $\nsym$ or the full \emph{Hopf algebra} structure of $\qsym$.

We propose the study of toric Richardson varieties in the Grassmannian as a quasisymmetric analog of Schubert varieties. Fundamental quasisymmetric and noncommutative ribbon functions
of degree $d$ 
are traditionally indexed by connected ribbons\footnote{In this paper ribbon simply means a diagram containing no $2\times 2$ box and need not be connected.} of size $d$ or subsets of $[d-1]$; we will index them also by binary strings\footnote{For $d=0$ we conjure up a string $\eps$ of length $-1$ corresponding to the empty ribbon. We set $F_\eps =  R_\eps = 1.$} of length $d-1$. Write $F_s$ and $R_s$ for the fundamental quasisymmetric and noncommutative ribbon functions indexed by a binary string $s$.
We will likewise index (indecomposable) toric Richardson varieties by binary strings; given a binary string $s$, we write $X_s$ for the associated toric Richardson variety. These will be shown to satisfy:
\begin{enumerate}
    \setcounter{enumi}{-1}
    \item $X_s$ is a normal projective toric variety with an affine paving by cells indexed by subwords\footnote{Throughout this paper, 
    subwords and substrings correspond to arbitrary subsequences; they need not be contiguous.
    Also $\eps$ is considered a subword of every word.} $t$ of $s$.
    \item There are isomorphisms of graded $\Z$-modules 
    \[A_{\bl}(X_s) \cong H_\bl(X_s, \Z) \cong \Span_{\Z}\{R_t \mid \text{$t$ is a subword of $s$}\}
    \subset \nsym
    \] 
    under which the class of a torus-invariant subvariety maps to a corresponding generalized noncommutative ribbon function.
    Furthermore, the \emph{effective} classes in $A_{\bl}(X_s)$  correspond exactly to the ribbon-positive elements.
    \item There are isomorphisms of graded rings
\[
A^{\bl}(X_s) \cong 
H^{\bl}(X_{s}, \Z) \cong \frac{\qsym}{(F_w\mid \text{$w$ not a subword of $s$})}
\]
under which $F_t$ corresponds to the dual class of the closure of the affine cell indexed by the subword $t$ of $s$.
\item A suitable direct limit of all toric Richardson varieties produces an ind-variety $Q_\infty$ (the infinite quasisymmetric Grassmannian) naturally equipped with a weak $H$-group\footnote{A weak $H$-group is defined similarly to an $H$-group, except that the homotopy relation is relaxed by evaluating it only on finite CW complexes. See \Cref{sec:infinite} for details.} structure so that 
$H^\bl(Q_{\infty}, \Z)$ and $H_\bl(Q_{\infty}, \Z)$ are isomorphic to the graded Hopf algebras
$\qsym$ and $\nsym$, respectively.
\end{enumerate}
A more detailed account of the relationships between our toric Richardson varieties and previous geometric models for $\qsym$
may be found in \Cref{ssec:compare}. 

The discovery of a geometric model for $\sym$ was striking not only because a variety with cohomology realizing $\sym$ was found, but also because the Grassmannian 
had long been of independent interest and its fundamental importance may be recognized apart from any connection to symmetric functions.
In our case, toric Richardson varieties have similarly been of considerable interest in the literature. Indeed, Richardson varieties play a central role in both algebraic geometry and algebraic combinatorics, appearing in contexts as varied as Schubert calculus, the study of cluster algebras, and the theory of total positivity. Despite being among the most fundamental and classical varieties in algebraic combinatorics, dating back at least to the work of Stanley in 1977 \cite{Sta77}, the (co)homology of Richardson varieties is still surprisingly poorly understood.

Toric varieties
are a rich family of varieties notable for their amenability to combinatorial methods. There has recently been a flurry of interest in 
those Richardson varieties which lie in this family: the toric Richardson varieties
\cites{TW15,LeeMasudaPark2021,LeeMasudaPark2023,CanSaha2025,GH25,ELPSW26,GKSB26, BGS26}, with the focus primarily on those in the full flag variety of type A. 

The moment polytopes of (indecomposable) toric Richardson varieties in the Grassmannian are also remarkably well-studied. They are the matroid base polytopes of snake matroids, which by \cite{FMP26}*{Proposition 5.10}
are exactly the {shard polytopes} in the sense of \cite{PPVJ23}. By \cite{KMSRA18}*{Theorem 4.7}, they are exactly order polytopes of fence posets. By \cite{ABDMPT26}*{Theorem 5.11}, these polytopes are exactly the Harder–Narasimhan polytopes of oriented path quivers in the sense of \cite{BKT14}. By \cite{PPVJ23}*{Theorem 65}, they are 
representatives of the rays of the type cones of Cambrian fans,
and thus are exactly the Newton polytopes of $F$-polynomials of type $A$ cluster algebras \cite{BCDMTY24}*{Theorem 3}, and are brick polytope summands of certain sorting networks \cites{PS12, PS15, BS18, JLS21}.
Our study of toric Richardson varieties will allow us to prove purely combinatorial statements about these polytopes (see \Cref{ssec:bettinumerology}).

\subsection{Summary of main results}
Recall that a Richardson variety in the Grassmannian is given by a
skew shape $\lambda/\mu$ (where $\mu \subseteq \lambda$ are partitions). We prove in \Cref{cor:toric_richardsons_in_GR} that the Richardson variety is toric if and only if $\lambda/\mu$ is a ribbon, i.e.\ a skew shape containing no $2\times 2$ square. In this case, it has a dense open orbit for the action of 
the standard torus of $\GL_n$.
It was a remarkable discovery that the analogous statement for the full flag variety is false \cites{GKSB26}.

A fundamental tool in the study of varieties with the action of an algebraic torus $T$ is the \BB decomposition, which was first described in \cites{BB, BB_Example}. This decomposition provides an algebraic analog of Morse theory for smooth projective $\Gm$-varieties. In particular, for a smooth projective $T$-variety $X$ whose fixed locus is finite, the \BB decomposition induced by a general cocharacter gives a paving of $X$ into affine spaces $\A^{n_i}$.
The cell closures then give a basis for the homology of $X$.
For \emph{non-smooth} projective varieties, two things can go wrong. First, the \BB cells may fail to be irreducible. Second, even when they are irreducible they may contain singularities. Surprisingly, though toric Richardson varieties in the Grassmannian are quite singular in general, both these problems can be avoided, as shown by our first result.
\begin{IntroThm}\label[Thm]{thm:affine_paving}
    Any toric Richardson variety in the Grassmannian has a torus-invariant paving by affine spaces given by a suitably chosen \BB decomposition.
\end{IntroThm}

\Cref{thm:affine_paving} produces a basis for the homology of $X_s$, which is naturally indexed by the distinct substrings of $s$. 
In \Cref{sec:homology}, we will see how to associate a ribbon to each torus-invariant subvariety of $X_s$.
For any ribbon $\beta$, there is a corresponding generalized noncommutative ribbon function $R_\beta$ in $\nsym.$ We recall precise definitions in \Cref{ssec:hopf_stuff}. 
\begin{IntroThm}\label[Thm]{thm:homology_relation_intro}
    The cycle class map $\cl\colon A_{\bl}(X_s) \to H_{\bl}(X_s, \Z)$ from Chow homology to singular homology is an isomorphism, and we have the further isomorphism of graded $\Z$-modules
    \[
        A_{\bl}(X_s) \cong H_\bl(X_s, \Z)
        \cong \Span_{\Z}\{R_t \mid \text{$t$ is a substring of $s$}\} \subset \nsym
    \]
    sending the class of each torus-invariant subvariety to the corresponding generalized noncommutative ribbon function. Furthermore, the \emph{effective} classes in $A_{\bl}(X_s)$ correspond exactly to the ribbon-positive elements.
\end{IntroThm}

We remark here that the product structure on $\nsym$ 
shows up naturally when we consider the homology of these toric Richardsons functorially, i.e.\ when we consider the pushforward maps induced by the natural inclusions between them; see \Cref{thm:homology_morphisms}.

The generalized noncommutative ribbon function which \Cref{thm:homology_relation_intro} associates to a torus-invariant subvariety is a particularly simple one in two special cases. 
When the torus-invariant subvariety is the closure of a cell in the \BB decomposition of \Cref{thm:affine_paving}, it is simply the noncommutative ribbon function of the associated binary string.
When the torus-invariant subvariety is a Richardson subvariety, it is simply the one indexed by the corresponding skew shape.

The cohomological counterpart of \Cref{thm:homology_relation_intro} reads as follows.
\begin{IntroThm}\label{thm:qsym_cohomology}
    The duals of the isomorphisms in homology are isomorphisms of graded rings
    \[
    A^{\bl}(X_s) \cong H^\bl(X_s, \Z)\cong \frac{\qsym}{\ide{F_w\mid \text{$w$ is not a substring of $s$}}}
    \]
    under which,  for any substring $t$ of $s$, the function $F_t$ corresponds to
    the dual class of the affine cell indexed by $t$.
\end{IntroThm} 

Just as in the case of $\sym$, the (co)homology of our toric Richardson varieties embodies only a \emph{shadow} of the full Hopf algebra structure, but is not actually a Hopf algebra. It is only upon passing to direct limits that the full structure comes into view. We first briefly recall the situation for $\sym$. We refer the reader to \cites{MT91, L11} for details.

The infinite Grassmannian $\Gr = \Gr(\infty, \infty)$ is an ind-variety obtained as the direct limit of the finite Grassmannians along natural inclusions between them.
It comes naturally equipped with the structure of an $H$-group whose multiplication map is induced by the direct sum maps 
$\boxplus\colon \Gr(r, n) \times \Gr(s, m) \to \Gr(r + s, n + m)$.
The $H$-group structure naturally leads to a graded Hopf algebra structure on its homology and cohomology groups, which both turn out to be isomorphic to $\sym$. The natural inclusions of the finite Grassmannians into $\Gr$ give rise to the usual presentations of the homology and cohomology of $\Gr(r, n)$ in terms of symmetric functions.

We similarly construct an ind-variety $Q_{\infty}$, which we call the \emph{infinite quasisymmetric Grassmannian}, as a direct limit of toric Richardson varieties. 
This space likewise comes naturally equipped with the structure of a \emph{weak} $H$-group whose multiplication map is induced by
the aforementioned direct sum map $\boxplus$, suitably restricted to toric Richardson varieties.
The weak $H$-group structure naturally leads to a graded Hopf algebra structure on its homology and cohomology groups, which turn out to be $\nsym$ and $\qsym$, respectively.
The natural inclusions of toric Richardson varieties into $Q_{\infty}$ give rise to the presentations in Theorems \ref{thm:homology_relation_intro} and \ref{thm:qsym_cohomology} of the homology and cohomology of $X_s$ in terms of noncommutative symmetric and quasisymmetric functions.
We summarize this in the following theorem.
\begin{IntroThm}
    The ind-variety $Q_{\infty}$
    admits a weak $H$-group structure giving rise to graded Hopf algebra structures on its homology and cohomology so that 
    $H^\bl(Q_\infty, \Z)\cong \qsym$ and $H_\bl(Q_\infty, \Z)\cong \nsym$.
\end{IntroThm}
\noindent
One can further show that the natural map $Q_{\infty} \to \Gr$ induced by the inclusion of toric Richardsons into Grassmannians induces the usual homomorphisms of Hopf algebras $\nsym \to \sym$ and $\sym \to \qsym$ by considering the corresponding maps in homology and cohomology, respectively; see \Cref{rmk:Grass}.

\subsection{Betti numerology}\label{ssec:bettinumerology}

It follows from \Cref{thm:homology_relation_intro} that for any word $s$, the $i$-th Betti number of $X_s$ is the number of distinct substrings of $s$ of length $i-1$. Recall that the $h$-vector of a \emph{simple} polytope is defined by \[\sum_{i=0}^d h_ix^i = \sum_{i=0}^d f_i(x-1)^i,\]
where $f_i$ counts the number of faces of dimension $i$. Somewhat nonstandardly, we will use the same terminology for \emph{non-simple} polytopes.\footnote{This $h$-vector is distinct from the so-called \emph{toric} $h$-vector of the polytope. The two notions coincide for smooth polytopes.}
For any binary string $s$, we write $P(s)$ for the moment polytope of $X_s$. The following is a byproduct of our study of toric Richardson varieties, combined with an old result of Chase \cite{C76}.
\begin{IntroThm}\label[Thm]{thm:bettinumbersprop}
    Let $s$ be a binary string of length $d-1$. There exists a sequence of integers $b_0, \dots, b_d$ such that
    \begin{enumerate}
        \item $b_0 = 1$ and $b_i$ for $i \geq 1$ is the number of distinct substrings of $s$  of length $i-1$.
        \item $b_i$ is the number of lattice paths in the ribbon diagram associated to $s$ in 
        which there are exactly $i$ loose steps (weakly) above the lowest vertical such step (see \Cref{df:loose}).
        \item $(b_0, \dots, b_d)$ is the $h$-vector of $P(s).$
        \item $b_i$ is the $i$-th Betti number of $X_s.$
        \item $\sum_{i=0}^d b_iq^i = \#X_s(\mathbb{F}_q)$ is the point-count polynomial of the closed Richardson variety $X_s$.
    \end{enumerate}
    Furthermore, this sequence enjoys the following properties.
    \begin{itemize}
        \item Log-concavity:  $b_i^{2}\ge b_{i-1}b_{i+1}$ for $1 < i < d.$
        \item Top-heaviness: 
        we have $b_i \leq b_j$ whenever $i \leq j \leq d+1 - i$.
    \end{itemize}
\end{IntroThm}
\noindent
The last properties are proved by Chase for any finite alphabet, but we only need the binary case.
The log-concavity portion of Chase's result was recently reproved by 
Vatter \cite{V26} in a surprisingly simple way.
We note that $b_0 = b_1= 1$ and $b_2 = 2$ whenever $s$
contains both $0$'s and $1$'s; 
so log-concavity fails at $i = 1$.
\begin{Rmk}
    The log-concavity (except at $i= 1$) of the Betti numbers of $X_s$ should be contrasted with the examples in \cite{Sta90} which show, for example,  that Betti numbers of the Grassmannian Schubert variety $\Omega_{\lambda}$ with $\lambda = (8,8,4,4)$ are not even unimodal.
\end{Rmk}
\begin{Rmk}
    \Cref{thm:bettinumbersprop} shows that the $h$-vector of any shard polytope (which is generally not simple) is nonnegative, top-heavy, and log-concave in positive degrees.
    It also provides a manifestly positive combinatorial formula for it.
\end{Rmk}
\noindent
From \Cref{thm:bettinumbersprop},
we will deduce the following corollary.
\begin{IntroCor}\label[Cor]{cor:flogconcave}
    The $f$-vector of any snake matroid base polytope is log-concave.
\end{IntroCor}
A question of Ferroni--Schr\"oter 
\cite{FS25}*{Question 3.4}
asks if the $f$-vector of any matroid base polytope is unimodal, or perhaps even log-concave.
\Cref{cor:flogconcave} gives a positive answer in the case of snake matroids.
To the best of our knowledge, this is the first nontrivial infinite family of matroids for which a positive answer to this question is known.

By comparing item (2) in \Cref{thm:bettinumbersprop} with \cite{S10}*{Corollary 3.4}, we obtain the following curious relationship between the $h$-vectors of $P(s)$ and of the independence complex of the corresponding lattice-path matroid $M(s)$. 
\begin{IntroCor}\label[Cor]{cor:hvectordominate}
    Let $(h_0, \dots, h_d)$ be the $h$-vector of $P(s)$ and let $(h'_0, \dots, h'_d)$ be the $h$-vector of the independence complex of $M(s).$ Then $h'$ dominates $h$. That is, 
    for any $0\le j\le d$,
    \[\sum_{i=0}^jh_i\le \sum_{i=0}^jh_i'.\]
\end{IntroCor}

\subsection{Relation to the literature}\label{ssec:compare}
We now relate our work with previously known geometric models for $\qsym$ and $\nsym$.

Baker--Richter \cite{BR08} considered the space $\Omega \Sigma \C \mathbb{P}\infty$, i.e. the loop space of the suspension of the infinite-dimensional projective space over $\C$ and proved that its cohomology is $\qsym$ as a Hopf algebra. Here, the $H$-group structure is intrinsic to the loop space.
Pechenik--Satriano \cite{PS26} 
later showed that the homotopy equivalent space $J\C \mathbb{P}^{\infty}$ can be approximated by non-normal projective varieties.
Given the strong resemblance between this space and our own, we propose the following conjecture (cf.\ \Cref{cj:homotopic2james} and the discussion thereafter).
\begin{cj}
    $Q_{\infty}$ is isomorphic to $\Omega \Sigma \C \mathbb{P}^\infty$ as a weak $H$-group.
\end{cj}
\noindent
This conjecture, if true, would parallel the situation for the infinite Grassmannian. In that case, there is an isomorphism of $H$-groups between $\Gr$ and the loop space $\Omega SU(\infty)$ \cite{MT91}.
Since there is also a natural map $\Omega\Sigma\C \mathbb{P}^\infty \to \Omega SU(\infty)$ 
inducing on cohomology the inclusion $\sym\to \qsym$ \cite{BR08}*{\textsection~1}, one can hope for a commutative square linking 
$Q_{\infty}$, $\Gr$, $\Omega \Sigma \C \mathbb{P}^{\infty}$, and $\Omega SU(\infty)$, which would be especially pleasant.

The algebraic stack of expanded pairs described in \cite{ACFW13} also admits the Hopf algebra $\qsym$ as its integral Chow ring by work of Oesinghaus \cite{O19}.
In this setting, an explicit description of the monomial basis in terms of certain tautological classes was provided \cite{BS22}. 
Pechenik and Satriano \cite{PS26}*{Remark 1.3} ask for a connection of this stack to $J\C \mathbb{P}^{\infty}$. We are not aware of any progress on this front.

Bergeron--Gagnon--Spink--Tewari \cite{BGST26} constructed a subscheme of the Grassmannian whose irreducible components are translated toric Richardson varieties and showed that its cohomology is given by a finite truncation of $\qsym$. They also constructed an affine paving for this space giving rise to a basis in cohomology corresponding to fundamental quasisymmetric functions.

Despite the strong similarity with our own construction, one cannot formally deduce our results from theirs. 
One basic obstruction is that the manner in which the various components in $\QGr(r,n)$ are glued together is highly nontrivial: 
$\QGr(r, n)$ is not a polytopal complex. For instance, $\QGr(2,4)$ contains two irreducible components whose intersection is \emph{not} irreducible \cite{BGST26}*{Figure 1}. 
In general, an affine paving of a reducible scheme need not restrict to an affine paving on each component. This \emph{does} hold for a torus-invariant affine paving of a polytopal complex of toric varieties, which is not the case for $\QGr(r, n)$.
We contrast this situation with the analogous one for Schubert varieties in the Grassmannian.
In that setting, 
the Grassmannian is irreducible and its
affine paving is a \emph{stratification},\footnote{By this we mean that the closure of each open cell is a union of open cells.} which allows one to formally deduce that the Schubert varieties are likewise stratified and to compute their (co)homology directly.

We remark that the cohomologies of $\QGr(r, n)$ and $X_s$ are \emph{different} finite truncations of $\qsym$. Recall from \Cref{thm:qsym_cohomology} that
\begin{equation*}
H^\bl(X_s, \Z)\cong \frac{\qsym}{\ide{F_w\mid\text{$w$ is not a subword of $s$}}},
\end{equation*}
whereas for $\QGr(r,n)$, one has 
\begin{equation*}H^\bl (\QGr(r,n), \Z)  = \frac{\qsym}{\ide{F_w\mid w\text{ contains at least $r$ many $1$'s or at least $n-r$ many $0$'s}}}.\tag{$\diamond$}\label{eq:qqr_ring}
\end{equation*}
Notably, the dimension of the top degree part 
of $H^\bl (\QGr(r,n), \Z)$
is greater than $1$ since there are multiple binary strings with exactly $r-1$ ones and $n-r-1$ zeros.
This explains why $\QGr(r,n)$ \emph{must} be reducible if it is to have the cohomology given by \eqref{eq:qqr_ring}.

The work of Bergeron, Gagnon, Spink, and Tewari in \cite{BGST26} is part of a series of works
with their collaborator Nadeau
on quasisymmetric functions, forest polynomials, and equivariant generalizations thereof \cites{NT24, NST24, NST25, BGNST25}. Notably, in \cite{BGNST26},
they construct a reduced torus-invariant subscheme of the flag variety whose cohomology is a ring of quasisymmetric coinvariants. The present authors are not aware of a direct connection between our current work and the latter.

\subsection{Summary of proof strategy}
In \Cref{sec:slither_cat}, we present a number of combinatorial models for toric Richardson varieties. This diversity of perspectives will help make the proofs of our key results more transparent. In \Cref{sec:paving}, we describe an affine paving of toric Richardson varieties by exploiting their recursive description as the toric varieties associated to series-parallel matroid polytopes.
In \Cref{sec:cohomology}, we will see how a degeneration formula due to Fulton--Sturmfels \cite{fulton94} gives an expression for the diagonal pushforward in homology which precisely matches the coproduct rule for noncommutative ribbon functions.
\Cref{sec:infinite} assembles these results to construct the infinite quasisymmetric Grassmannian $Q_{\infty}$.

\subsection*{Acknowledgements}
The authors are grateful to Anders Buch, whose insightful questions have inspired a large part of the present work.
We are indebted to Matt Larson for a number of helpful conversations throughout the development of the project. 
We express our gratitude to Sergey Fomin for indispensable practical advice.
Thomas Lam directed us to a number of important references in the literature.
We thank Cameron Chang and Josephine Hlavinka for invaluable discussions relating to the content of \Cref{sec:toric_positroids}.
We additionally thank 
Paolo Aluffi, Dave Anderson, 
Michael Barz, 
June Huh, Soyeon Kim and Allen Knutson for illuminating conversations.
Many of these useful discussions took place at the Affine Combinatorics and Quantum $K$-theory workshop, which took place at the Ohio State University, and at the CMND 2026 Thematic Program in Algebraic Combinatorics and Applications, which took place at the University of Notre Dame. We thank the organizers for their efforts. We likewise wish to thank the organizers of the Schubert Seminar; it was following one of its weekly talks that the authors first realized that toric Richardson varieties ought to have affine pavings.

AI tools were used for proofreading, generating figures, reference searches, and assistance with coding used to test initial conjectures.
The mathematical content---both the ideas and the words---is due entirely to the authors.

\section{Preliminaries}
\subsection{Toric Varieties}\label{sec:toric}
Throughout the paper, except where explicitly stated otherwise, we will work over a fixed field $\Bbbk$ of arbitrary characteristic. Notably, in \Cref{sec:infinite}, we restrict ourselves to working over $\C$.
In any case, we will usually assume that $\Bbbk$ is algebraically closed.
A \textbf{variety} over $\Bbbk$ is a reduced, irreducible, separated, finite-type scheme over $\Bbbk$.

Given an algebraic torus $T$, let $M = \Hom(T, \Gm)$ be its character lattice and $N = \Hom(\Gm, T)$ its cocharacter lattice. We write $M_{\R} \deq {M \otimes_{\Z} \R}$ and $N_{\R} \deq {N \otimes_{\Z} \R}$.
Given $m \in M_{\R}$ and $v \in N_{\R}$, we write $\langle m, v \rangle$ for the value of the canonical pairing $M_{\R} \times N_{\R} \to \R$. 

Let $P$ be a full-dimensional lattice polytope in $M_{\R}$.
The \emph{outer} normal fan $\Sigma$ to $P$ is the complete fan in $N_{\R}$ whose cones are
\[
\sigma_F \deq \{u \in N_{\R} \,\mid\, \langle p , u \rangle \leq \langle f , u \rangle \text{ for all } p \in P, f \in F\},
\]
for each face $F$ of $P$. 
A fan constructed in this way defines a projective toric variety $X_P$ (see \cite{fult93toric}, \cite{cox2011toric}) with dense torus $T$.
The faces of $P$ are in an inclusion- and dimension-preserving correspondence with the irreducible $T$-stable closed subvarieties of $X_P$. 
Each such subvariety is the closure of a unique $T$-orbit.
We write $O_{F}$ for the $T$-orbit corresponding to a face $F$ of $P$, which can be thought of as corresponding to the relative interior of $F$. Each such orbit comes equipped with a distinguished basepoint, which may be defined as the limit of the trajectory of the basepoint in the dense orbit under the action of any cocharacter in the relative interior of $\sigma_F$. 

For a cone $\sigma$ in $\Sigma$, 
we write $N_{\sigma}$ for the sublattice of $N$ generated by $\sigma \cap N$.
The orbit $O_P$ is the unique dense open orbit of $X_P$, which we identify with $T$ itself.
For a vertex $p \in P$, the orbit $O_{p}$ consists of a single fixed point, which we identify with the point $p$. 
When dealing with toric varieties, we will always implicitly assume that they are normal and projective; equivalently, they may be obtained from a polytope $P$ by the procedure above.
The choice of such a polytope $P$ defines an ample Cartier divisor $\mathcal{O}(P)$ on $X_P$ whose global sections are $\bigoplus_{m \in P\cap M} \Bbbk \chi^{-m}$. 
If $P$ is \emph{very ample},
which can always be made true by rescaling it by a suitably large integer $k$ (see \cite{cox2011toric}*{Chapter 6}), 
thereby replacing 
$\mathcal{O}(P)$ with
$\mathcal{O}(kP) = \mathcal{O}(P)^{\otimes k}$,
its complete linear series defines an embedding of $X_P$ into the projective space 
$\mathbb{P}^{P \cap M - 1}$. On the torus, this takes the point $t$ to $(\chi^{-m}(t))_{m \in P}\in \mathbb{P}^{P \cap M - 1}$. 
 
For each face $F$ of $P$, the orbit closure 
$V(F) \deq \ove{O_F}$ is again a toric variety.
The standard affine charts for $V(F)$ are of the form $U_p^F \deq \Spec(\Bbbk[M\cap \sigma_p^\vee \cap \sigma_F^\perp])$ for $p$ a vertex of $F$. Equivalently, if $\Cone(p - F) = \R_{\geq 0} \{p - f | f \in F\}$ then $U_p^F = \Spec(\Bbbk[M \cap \Cone(p -F)])$. 
The polytope $P$ is said to be \textbf{smooth at a vertex $p$} if $\Cone(p - P)$ is generated as a cone by the basis of some saturated sublattice of $M$. We say that $P$ is \textbf{smooth} if it is smooth at each vertex. 
Then $P$ is smooth at $p$ if and only if $U_p^P$ is a smooth variety, in which case it is isomorphic to the affine space $\A^{\dim P}$. In particular, $P$ is smooth if and only if $X_P$ is smooth.

\subsection{Chow homology and cohomology}\label{sec:chow_section}

For a variety $X$, we write $A_\bl(X) = \bigoplus_i A_i(X)$ for its total Chow homology group and $A^\bl(X) = \bigoplus_i A^i(X)$ for its operational Chow cohomology ring. We refer the reader to \cite{fult_int_theory} for definitions. 

Chow cohomology is equipped with cup products $A^p(X) \otimes A^q(X) \to A^{p + q}(X)$ taking $a \otimes b \mapsto a \cup b$ and cap products: $A^p(X) \otimes A_q(X) \to A_{q - p}(X)$ taking $a \otimes c \mapsto a \cap c$. When $X$ is complete, there is a canonical map $\deg\colon A_0(X) \to \Z$. This gives rise to
a canonical pairing $A^\bl(X) \otimes A_\bl(X) \to \Z$ taking $a\otimes c$ to $\inner{a, c} \deq \deg(a \cap c)$, as well as the corresponding Kronecker map $A^{\bl}(X) \to \Hom(A_{\bl}(X), \Z)$. The latter map is not an isomorphism in general.

For a closed subscheme $U$ of $X$, we write $[U]$ for its class in $A_{\bl}(X)$. When $U = X$, this is the \emph{fundamental class} of $X$.
Suppose $X$ is complete and that the K\"unneth map $\kappa\colon A_\bl(X) \otimes A_\bl(X) \to A_\bl(X\times X)$ given by $[U]\otimes[V] \mapsto [U \times V]$ is an isomorphism. Then, the group $A_\bl(X)$ forms a graded coalgebra.\footnote{We always assume that (co)algebras are (co)associative.} By this, we mean the following. Let $\Delta\colon A_\bl(X) \to A_\bl(X) \otimes A_{\bl}(X)$ be the composition $\Delta = \kappa^{-1} \circ \delta_*$ where $\delta\colon X \to X \times X$ is the diagonal
and let $e\colon A_\bl(X) \to A_{\bl}(\text{pt}) \cong \Z$ be the pushforward to a point.
Then coassociativity $(\text{id}_{A_\bl(X)} \otimes \Delta) \circ \Delta = (\Delta \otimes \text{id}_{A_\bl(X)})\circ \Delta$ and counitality $(\text{id}_{A_\bl(X)} \otimes e) \circ \Delta = \text{id}_{A_\bl(X)}= (e \otimes \text{id}_{A_\bl(X)})\circ \Delta$ hold for $\Delta$, $e$. 
This coalgebra is also cocommutative: the map 
$\Delta$ is invariant under the involution of 
$A_{\bl}(X) \otimes A_{\bl}(X)$ that swaps the two factors. This is because $\delta\colon X \to X \times X$ is invariant under the corresponding involution of $X \times X$.
It follows formally that $\Hom(A_\bl(X), \Z)$ is a commutative $\Z$-algebra. 
Sometimes, this agrees with the operational Chow cohomology:
\begin{Lm}[\cite{int_theory_spherical}*{\textsection 4}]\label[Lm]{lm:kunneth_chow}
     Let $X$ be a complete algebraic scheme over an algebraically closed field $\Bbbk$. Assume the K\"unneth map $A_\bl(X)\otimes A_{\bl}(Y) \to A_{\bl}(X\times Y)$ taking $[V]\otimes [W]$ to $[V\times W]$ is an isomorphism for all algebraic schemes $Y$ over $\Bbbk$. Then, the Kronecker map $A^{\bl}(X) \to \Hom(A_{\bl}(X), \Z)$ is an isomorphism of $\Z$-algebras. 
\end{Lm}

\begin{Df}
    We say that $X$ admits an \textbf{affine paving} if we can decompose $X$ as a disjoint union $X = X_0 \sqcup \dots \sqcup X_k$ such that 
    each partial union $X_0 \sqcup \dots \sqcup X_i$ is closed in $X$ and
    each cell $X_i$ is isomorphic to an affine space $\A^{n_i}$.
\end{Df}
Whenever $X$ admits an affine paving, the cell closures $\ove{X}_{i}$ form a basis for $A_{\bl}(X)$ \cite{chow_basis}*{Theorem 1} (see also \cite{fult_int_theory}*{Example 19.1.11}). To relate singular homology and cohomology to Chow groups we use the result below. 
For convenience, we rescale the grading in both singular homology and cohomology by a factor $\frac{1}{2}$,\footnote{This means that singular homology and cohomology are graded by the group $\frac{1}{2}\Z$. In practice, the half-degrees will always sit empty.} so that the cycle class map $A_{\bl}(X) \to H_{\bl}(X)$ is degree-preserving.
\begin{Lm}\label[Lm]{lm:cycle_class}
    Let $X$ be a complete variety over $\C$ with an affine paving. Then the cycle class map, $\cl\colon A_\bl(X) \to H_{\bl}(X, \Z)$ is an isomorphism of graded coalgebras and its dual $H^\bl(X, \Z) \to A^\bl(X, \Z)$ is a graded ring isomorphism.
\end{Lm}
\begin{proof}
    First, $A_\bl(X)$ is freely generated by the cell closures by the discussion above. That $\cl$ is a group isomorphism is \cite{chow_basis}*{Theorem 1}. Thus, $H_\bl(X, \Z)$ is free, finitely generated, and vanishes at half-integer degrees. On the Chow side, $\Hom(A_\bl(X), \Z) \cong A^\bl(X)$ by  \Cref{lm:kunneth_chow} and \cite{TOTARO_chow}*{Proposition 1}. On the singular side, the universal coefficient theorem for cohomology \cite{hatcher}*{Theorem 3.2} implies that the Kronecker map $H^\bl(X, \Z) \to \Hom(H_\bl(X, \Z), \Z)$ is an isomorphism. 
    The K\"unneth theorem for singular homology
    \cite{munkres84}*{Theorem 59.3}
    ensures that $H_\bl(X, \Z) \otimes H_{\bl}(X, \Z) \to H_\bl(X \times X, \Z)$ is an isomorphism. Thus, $H_\bl(X, \Z)$ has a coalgebra structure induced by the diagonal map. Since $\cl$ commutes with pushforward, it is an isomorphism of coalgebras. Since the singular and Chow cohomologies of $X$ are the duals of their respective coalgebras, it follows that the dual map $\cl^\vee$ is an isomorphism of algebras.   
\end{proof}

\subsection{Bia\l ynicki-Birula decompositions and homology}\label{sec_BB}

We recall the \BB decomposition introduced in \cite{BB}. For a complete $\Gm$-variety $X$ with finite fixed locus $X^{\Gm}$ and a point $p \in X^{\Gm}$, write $\lim_{t\to 0} t\cdot x$ and $\lim_{t\to \infty} t\cdot x$ for the values at $0$ and $\infty$ of the unique extension to $\Pb^1$ of the map $\Gm \to X$ defined by $t \mapsto t \cdot x$. Let
\[
X_p^+ = \left\{x \in X \;\middle|\;  \lim_{t\to 0} t \cdot x = p\right\}, \qquad X_p^- = \left\{x \in X \;\middle|\;  \lim_{t\to \infty} t \cdot x = p\right\}.
\]
The collections $\{X_p^+\}_{p \in X^{\Gm}}$ and $\{X_p^-\}_{p \in X^{\Gm}}$ each partition $X$ into locally-closed subvarieties. They are, respectively, the \emph{positive} and \emph{negative} \BB decompositions of $X$. 

Under the assumptions that $X$ is a smooth complete $\Gm$-variety with finite fixed locus, it was shown in \cite{BB} that each cell $X_p^+$ is isomorphic to affine space. If $X$ is additionally projective, these cells form an affine paving of $X$
\cite{BB_Example}*{Theorem 3}. By the discussion in \Cref{sec:chow_section} this implies that the classes $[\ove{X_p^+}]$ form a basis of $A_{\bl}(X)$. However, when $X$ is singular, the cells $X_p^+$ are usually quite singular themselves, so in particular are not affine spaces. Since the varieties we work with will be singular in all but trivial cases, we need some sufficient conditions for the \BB cells to be affine in the toric setting.

Let $X$ be the toric variety associated to a lattice polytope $P$. A cocharacter $v \in N$ induces a $\Gm$-action on $X$ via $(t, x) \mapsto v(t) \cdot x$. The condition that $X^{\Gm}$ has finite fixed locus is equivalent to the condition that $v$ is not perpendicular to any edge of $P$. 
We make this assumption throughout.
In this setting we write $X_p^{v}$ for $X_p^+$ and $X_p^{-v}$ for $X_p^-$.
For a face $F$ of $P$ and a vertex $p \in P$, write $F^{v}$ for the vertex of $F$ maximizing $v$. Then $F^{-v}$ is the vertex of $F$ minimizing $v$. With this notation, we have that $X_p^v = \bigsqcup_{F^v = p} O_F$ by \cite{struct_prop}*{Lemma 3.3}.\footnote{Though smoothness is assumed in that paper throughout, the proof of Lemma 3.3 does not use that assumption.} In general, $X_p^v$ need not be irreducible. However, if the edges $e$ of $P$ satisfying $e^v = p$ are precisely those edges incident to $p$ which are contained in some face $F$ of $P$, then $X_p^v = \bigsqcup_{p \in G \subseteq F} O_G = U_p^F$ and $\ove{X_p^v} = V(F)$. This condition is automatic if $P$ is \textbf{simple};
that is, if $P$ has $\dim P$ many edges at each vertex.
However, the polytopes we consider will almost never be simple. The definition below identifies some desirable properties of $v$ and the associated \BB decomposition.

\begin{Df}
    Let $P$ be a polytope in a real vector space and $v$ a covector.
    \begin{itemize}
        \item We say that $v$ \textbf{paves $P$ at $p$} if $P$ has a unique maximal face $P_p^v$ on which $p$ maximizes the functional $\inner{-,v}$. 
        We say that $v$ \textbf{paves} $P$ if it paves $P$ at each of its vertices; equivalently all \BB cells $X_p^{v}$ of $X$ are irreducible.
        \item We say that $v$ \textbf{simply paves $P$ at $p$} if $P_p^v$ is simple at $p$.
        We say that $v$ \textbf{simply paves} $P$ if it simply paves $P$ at each of its vertices.
        \item We say that $v$ \textbf{smoothly paves $P$ at $p$} if $P_p^v$ is smooth at $p$.
        We say that $v$ \textbf{smoothly paves} $P$
        if it smoothly paves $P$ at each of its vertices; equivalently, the \BB decomposition of $X$ is an affine paving.
        \item We say that $v$ \textbf{stratifies} $P$ if it paves $P$ and if 
        every cell closure of the \BB decomposition ${X_p^v}$ is a union of cells. 
    \end{itemize}
\end{Df}

If $m$ is a lattice vector in a lattice $M \subseteq M_{\R}$, the \textbf{primitive direction} of $m$ is the unique generator of the semigroup $M \cap \Cone(m)$. If $P$ is a lattice polytope in $M$ and $p, q \in P$ are adjacent vertices, the \textbf{primitive edge direction} from $p$ to $q$ is the primitive direction of $q - p$.
A polytope $P$ is \emph{edge-unimodular} if the matrix which records all primitive edge directions in $P$ is unimodular. Here, a matrix $\Xi$ of rank $r$ is unimodular if all $r \times r$ minors of $\Xi$ lie in $\{0, \pm1\}$.
\begin{Obs}\label[Obs]{obs:tot_uni}
    In an edge-unimodular polytope $P$, a covector which simply paves $P$ smoothly paves $P$.
\end{Obs}

\subsection{\texorpdfstring{The Hopf algebras $\sym$, $\nsym$ and $\qsym$}{The Hopf algebras Sym, NSym, QSym}}\label{ssec:hopf_stuff}
We collect some facts about the Hopf algebras $\sym$, $\nsym$ and $\qsym$.
We refer the reader to \cite{GR20} for a comprehensive survey. For $\nsym$, we will use the combinatorial interpretation in terms of tableaux, which is implicit in \cites{GKLLRT95, KLT97} and appears more explicitly in \cites{Thi01, NT09, H16, HM25}.

In this text, we insist that all algebras and coalgebras are associative, unital, graded and connected; a graded $\Z$-algebra is \textbf{connected} if its degree $0$ component is $\Z$. 
This means that our Hopf algebras are uniquely determined by their underlying algebra along with the coproduct map $\Delta$; the antipode is determined by Takeuchi's formula \cite{GR20}*{Proposition 1.4.24}. A Hopf algebra is \textbf{commutative} if the underlying algebra is commutative. 

Consider the ring $\Z[[x_1, x_2, \dots]]$ of formal power series in the infinitely many variables $x_1, x_2, \dots$. The ring of symmetric functions $\sym$ is the subring of $\Z[[x_1, x_2, \dots]]$ consisting of elements of bounded degree which are symmetric, i.e., invariant under swapping any two variables. 

We draw partitions using the English convention. For any skew shape $\lambda/\mu$, a semistandard Young tableau (SSYT) of shape $\lambda/\mu$ is a filling of the boxes in $\lambda/\mu$ with positive integers such that the numbers weakly increase along rows and strictly increase down columns. The weight of such a tableau $T$, denoted $\wt(T)$, is defined to be the product $\prod_{b\in \lambda/\mu} x_{T(b)}$ where $T(b)$ is the number written in the box $b$ and the product is over all boxes in the skew shape. The \textbf{skew Schur function} associated to $\lambda/\mu$ is
\[s_{\lambda/\mu} := \sum_{T\in \SSYT(\lambda/\mu)}\wt(T),\]
where $\SSYT(\lambda/\mu)$ is the set of SSYTs of shape $\lambda/\mu$.  The \textbf{Schur functions} $s_{\lambda} := s_{\lambda/\varnothing}$ form a $\Z$-basis of $\sym$. The ring $\sym$ admits a canonical Hopf algebra structure whose coproduct is given by
\[\Delta(s_{\lambda/\mu}) = \sum_{\nu}s_{\lambda/\nu}\otimes s_{\nu/\mu}.\]

We can similarly define the ring $\qsym$ of \textbf{quasisymmetric functions} as  the subring of $\Z[[x_1, x_2, \dots]]$ consisting of the elements $f$ of bounded degree such that for any $i_1 < \dots < i_k$ and $j_1 < \dots < j_k$ and all $a_1, \dots, a_k > 0$, the coefficients in $f$ of the monomials $x_{i_1}^{a_1}\cdots x_{i_k}^{a_k}$ and $x_{j_1}^{a_1}\cdots x_{j_k}^{a_k}$ are equal. For each binary string $w$ of length $d-1$, we define the \textbf{fundamental quasisymmetric function} $F_w$ by
 \[F_w:= \sum_{\substack{i_1 \le \dots \le i_d\\ i_{a} < i_{a+1}\text{ if }w_a = 1}}x_{i_1 }\cdots x_{i_d}.\]
 These form a $\Z$-basis for $\qsym$.
 We will occasionally have use for an imagined binary string $\eps$ of length $-1$ corresponding to the unique composition of $0$. Here, $F_{\eps} = 1$. We will speak of \textbf{extended binary strings} to mean ones possibly taking the value $\eps$.
 
To multiply fundamental quasisymmetric functions, one can use the following product rule. For a permutation $\pi$, we define $\Des(\pi)$ to be the binary string whose $i$-th character is $0$ if $\pi(i) < \pi(i+1)$ and $1$ otherwise. A permutation $\pi \in S_{m+n}$ is a shuffle of $\sigma\in S_m$ and $\tau\in S_n$ if the relative order of $1, \dots, m$ in the one-line notation for $\pi$ is the same as in $\sigma$ and the relative order of $m+1, \dots, m+n$ is the same as that of $1,\dots,n$ in $\tau$. Then multiplication of quasisymmetric functions satisfies
\[F_{\Des(\sigma)}F_{\Des(\tau)} = \sum_{\substack{\pi\text{ shuffle of }\\\sigma\text{ and }\tau}}F_{\Des(\pi)}.\]
It is easy to convince oneself that $\Des(\sigma)$ is a subsequence of $\Des(\pi)$ for any shuffle $\pi$ of $\sigma$ and $\tau$. Hence,
\begin{Prop}[{cf. \cite{S22}*{p. 214}}]\label[Prop]{prop:quasisubword}
    For binary strings $u$ and $v$, if $F_t$ appears in the expansion of $F_uF_v$ then $u, v$ are substrings of $t$.
\end{Prop}
Just like the ring $\sym$, the ring $\qsym$ likewise has a canonical Hopf algebra structure 
\cite{GR20}*{Proposition 5.2.15}
in which the coproduct is given 
by \[\Delta(F_w) = \sum_{\substack{w = s0t\\\text{ or }\\w = s1t}}F_s\otimes F_t.\]

If $V^{\bl}$ is a graded $\Z$-module which is free and finitely-generated in each degree, its \textbf{graded dual} $\widecheck{V}^{\bl}$ is the $\Z$-module whose $j$-th graded piece is $\Hom(V^j, \Z)$, i.e.\ the dual of $V^j$.
When $V^{\bl}$ is equipped with the structure of a Hopf algebra, there is a canonical Hopf algebra structure on $\widecheck{V}^{\bl}$ \cite{GR20}*{\textsection~1.6} called the \textbf{dual graded Hopf algebra}.
The multiplication in $\widecheck{V}^{\bl}$ is dual to the comultiplication in $V^{\bl}$ and vice versa. Explicitly,
$\inner{xy,  a} = \inner{x\otimes y, \Delta a},$ and $\inner{\Delta x, a\otimes b} = \inner{x, ab},$ for all $x, y \in \widecheck{V}^{\bl}$ and all $a, b \in V^{\bl}$, where we have written $\inner{-,-}$ for the canonical pairing.
The graded dual of $\widecheck{V}^{\bl}$ is again $V^{\bl}$.

The Hopf algebra $\sym$ is isomorphic to its dual. The graded dual of $\qsym$ is a non-commutative Hopf algebra called $\nsym$.
 To define $\nsym$, we work in the ring  $\Z\langle\langle x_1, x_2, \dots\rangle\rangle$ of formal power series in \textbf{non-commuting} variables $x_1, x_2, \dots$. For each (possibly disconnected) ribbon $\alpha$, we define the \textbf{generalized noncommutative ribbon function}
 \[R_{\alpha} = \sum_{T\in \SSYT(\alpha)} \wt(T),\]
 where now $\wt(T)$ is again the product of variables indexed by the entries in $T$ but now the variables lie in $\Z\langle\langle x_1, x_2, \dots\rangle\rangle$.
The product is taken by traversing the ribbon southwest to northeast. It is clear that $R_{\alpha}R_{\beta}$ is $R_{\alpha\oplus \beta}$ where $\alpha\oplus \beta$ is the ribbon obtained by joining the southwest corner of $\beta$ to the northeast corner of $\alpha$. Conditioning on whether the last variable in a SSYT for $\alpha$ is $\leq$ or $>$ than the first variable in a SSYT for $\beta$, one obtains:
 \[R_{\alpha}R_\beta = R_{\alpha\oplus \beta} = R_{\alpha \cdot \beta} + R_{\alpha \odot\beta}.\]
 Here, $\alpha\cdot \beta$, (resp. $\alpha\odot \beta$) is the ribbon obtained by joining $\alpha$ and $\beta$ so that the first box of $\beta$ is placed on top of (resp. to the right of) $\alpha$. The $R_\alpha$ for connected ribbons $\alpha$ are called \textbf{noncommutative ribbon functions} and 
  form a $\Z$-basis of $\nsym.$ From a connected ribbon $\alpha$, one can obtain a binary string by traversing the boxes of $\alpha$ southwest to northeast and recording a $0$ (resp.\ a $1$) each time the next box appears to the right of (resp. on top of) the current one, and $\eps$ corresponds to the empty ribbon. In this way, we index noncommutative ribbon functions by binary strings as well (and $R_\eps=1$). One then has $R_sR_t = R_{s0t}+R_{s1t}$. 
The coproduct in $\nsym$ is given by the following rule.
 \begin{Prop}[{cf. \cite{H16}*{Theorem 4.3}}]\label[Prop]{prop:coproduct_nsym}
    For any (possibly disconnected) ribbon $\lambda/\mu$
     \[\Delta(R_{\lambda/\mu}) = \sum\limits_{\nu}R_{\lambda/\nu}\otimes R_{\nu/\mu}.\]
 \end{Prop}
 
\subsection{Cohomology of the Grassmannian}\label{sec:richardson}
For a totally ordered indexing set $E$, let $\Gr(r, E)$ be the Grassmannian of $r$-dimensional subspaces of $\Bbbk^E$. A rational point $L \in \Gr(r, E)$ corresponds to the row span of an $r\times E$ matrix $A_L$ of full rank. Given $L$, this matrix is well-defined up to the left action of $\GL_r$. The natural (left) action of $\GL_E$ on $\Gr(r, E)$ takes $(g, L)$ to the rowspan of $A_L g^{-1}$. Write $B_E^+$ (resp.\ $B_E^-$) for the upper (resp.\ lower) triangular matrices in $\GL_E$ and write $T_E$ for the diagonal torus. 

The $T_E$-fixed points of $\Gr(r, E)$ are the subspaces $\Bbbk^I = \langle e_i : i \in I \rangle$ for $I \in \binom{E}{r}$. 
Let $E = \{\xi_1 < \dots < \xi_n\}$.
For a partition $\lambda$ that fits inside the $r\times (n-r)$ rectangle, let $I(\lambda)$ denote the set 
\[
    \{\xi_j \mid \text{$j$ is an up step of the lower boundary path of $\lambda$}\}.
\]
Write $\Omega^\lambda$ (resp. $\Omega_\lambda$) for the (closed) Schubert variety given by $\ove{B_E^+\Bbbk^{I(\lambda)}}$ (resp. $\ove{B_E^-\Bbbk^{I(\lambda)}}$). Whenever we have two partitions $\mu \subseteq \lambda$ within an $r \times  (n - r)$ grid, we write $\Omega_\lambda^\mu = \Omega_{\lambda} \cap \Omega^\mu$ for the Richardson variety associated to $\lambda/\mu$, whose dimension is $|\lambda| - |\mu|$. 

It is classical that Schubert classes $[\Omega_\lambda]$ form a basis for $A_\bl(\Gr(r, E))$. The total Chow cohomology group $A^{\bl}(\Gr(r, E))$ correspondingly has the dual basis whose classes we denote by $\sigma^\lambda$. For background and standard facts in Schubert calculus we refer the reader to \cite{Fulton1997}. 
We may identify $A_\bl(\Gr(r, E))$ with the $\Z$-submodule of $\sym$ spanned by $s_\lambda$ for which $\lambda$ fits within an $r \times (n-r)$ sized grid by identifying $[\Omega_\lambda]$ with $s_\lambda$. Under this identification, the class of a Richardson variety corresponding to $\lambda/\mu$ corresponds to $s_{\lambda/\mu}$. Correspondingly, we will view $A^{\bl}(\Gr(r, E))$ as the quotient 
\[
    \sym/( s_{\lambda} \mid \text{$\lambda$ does not fit in the $r \times (n-r)$ grid} )
\]
by identifying 
$\sigma^\lambda$ with the element $s_\lambda$ in the quotient. 

\section{The slither category}\label{sec:slither_cat}
We introduce a category\footnote{In the sense of Eilenberg and Mac Lane.} which we call the \textbf{slither category}.
The objects of this category can be thought of either as combinatorial or geometric objects.
More precisely, the objects can variously be described as:
\begin{itemize}
    \item Loopless, coloopless, regular lattice-path matroids, which we call \textbf{snake den matroids}.
    \item The matroid base polytopes of these matroids, a special class of \emph{Nakajima} polytopes.\footnote{In the sense of \cite{haase21}*{\textsection~2.2.1}.}
    \item Sequences of binary strings.
    \item ``Primitive'' toric Richardson varieties in the Grassmannian.
    \item ``Reduced'' ribbon diagrams in a rectangular grid, which we call \textbf{snake den diagrams}.
\end{itemize}
The morphisms in this category, which we call \textbf{slithers}, can correspondingly be described as: 
\begin{itemize}
    \item Regular lattice-path matroids with loops and coloops added, which we call \textbf{slither matroids}.
    \item The matroid base polytopes of these matroids, the \emph{faces} of the matroid polytopes considered earlier.
    \item Certain words in the alphabet 
    $\pr, \rp, 0, 1, \elll, \llle$, which we call \textbf{slither words}.
    \item Certain toric positroids in the Grassmannian; namely those contained in toric Richardson varieties.
    \item Certain subgraphs of a rectangular grid, which we call \textbf{slither diagrams}. These are ``topological deformations'' of ribbon diagrams.
\end{itemize}
The various subsections that follow are dedicated to explaining these different viewpoints on this category and to proving that they are indeed equivalent.
\subsection{Snake matroids}\label{subsec:matroids}
We recall that a \emph{matroid} $M$ of rank $r$ is a finite set $E = E(M)$, called the \textbf{ground set} of $M$, along with a nonempty collection $\mathcal{B}(M) \subseteq \binom{E}{r}$ of \textbf{bases} satisfying the \textbf{exchange axiom}: 
\begin{center}
For any $B, B' \in \mathcal{B}(M)$ and $a \in B - B'$ there is $b\in B'-B$ such that $B - a + b$ is in $\mathcal{B}(M)$.\footnote{We use $+$ and $-$ as a convenient substitute for set union and set difference. As usual, we often write $x$ in lieu of $\{x\}$.}
\end{center}
We refer the reader to \cite{Oxley} for a proper treatment of the subject.
The \textbf{rank} of a subset $S \subseteq E(M)$ is the maximum cardinality of $B \cap S$ as $B$ ranges over the bases in $\mathcal{B}(M)$; we write $\rk_M S$ for this number.
The \textbf{loops} (resp.\ \textbf{coloops}) of a matroid $M$ are the elements of $E$ contained in none of (resp.\ each of) the bases. 
We write $\mathrm{L}(M)$ (resp.\ $\Gamma(M)$) for the set of loops (resp.\ coloops) of $M$. 
We call the elements of $\Lamma(M) \cup \Gamma(M)$ \textbf{bound} and the remaining elements of $E(M)$ \textbf{unbound}.
A matroid is \textbf{loopless} (resp.\ \textbf{coloopless}) if it has no loops (resp.\ coloops).

In this text, we make the unusual but crucial assumption that, unless otherwise specified, the ground set $E(M)$ of any matroid $M$ is \emph{totally ordered}. 
On the rare occasion in which we will need them, we speak of \textbf{unordered matroids} when this ubiquitous assumption is momentarily dropped.
We say that two matroids are \textbf{isomorphic} if there is an \textbf{isomorphism} between their ground sets: an \emph{order-preserving} bijection under which the bases of one are sent bijectively onto the bases of the other. Isomorphisms are \emph{unique} when they exist.
When the ground set $E$ is the set $[n] = \{1,\dots, n\}$ for some nonnegative integer $n$, we will assume---unless otherwise specified---that the total order is the usual one. Matroids on such a ground set are \textbf{standard}.
Each matroid is isomorphic to a unique standard matroid, which we call its \textbf{standardization}.
Likewise, if $f\colon E \to E'$ is an increasing map of totally ordered sets of cardinalities $n$ and $n'$ respectively, its \textbf{standardization} is the unique map $\hat f\colon [n] \to [n']$ such that $f$ sends the $i$-th element of $E$ to the $\hat{f}(i)$-th element of $E'$.

Many of the usual operations on unordered matroids have natural analogs for matroids with totally ordered ground sets. We now list those that we will need later, with a particular emphasis on the effect they have on the totally ordered ground set.
\begin{itemize}
\item 
The \textbf{dual} of a matroid
$M$ is the matroid $M^*$ whose bases are 
\(
    \mathcal{B}(M^*) = \{E(M) - B \mid B \in \mathcal{B}(M)\}
\)
The ground set $E(M^*)$ is $E(M)$ but with the order reversed. 
Correspondingly, we write $E^*$ for the order-reversal of a totally ordered set $E$. Then $E(M^*) = E(M)^*$. Also $(M^*)^* = M$.

\item If $S \subseteq E(M)$ is any subset, the \textbf{$S$-initial matroid} of $M$ is the matroid $\D{M}{S}$ with the same ground set and with bases $\mathcal{B}(\D{M}{S}) = \{B \in \mathcal{B}(M) \mid \abs{B \cap S} = \rk_M S\}$. The \textbf{$S$-terminal matroid} of $M$, denoted $\U{M}{S}$, is the matroid $\D{M}{E(M)-S}$.
For $s \in E(M)$, we write $\D{M}{s}$ and $\U{M}{s}$ for $\D{M}{\{s\}}$ and $\U{M}{\{s\}}$ respectively.\footnote{As an \emph{unordered matroid}, $\D{M}{S}$ is $M/S \oplus M|_S$.}

\item The \textbf{parallel extension} of a matroid $M$ by an element $e \in E(M)$ is the matroid $\Pi_eM$ whose ground set $E(\Pi_eM) = E(M) \sqcup \{e'\}$ is obtained by adding a new element $e'$ to $M$ immediately following $e$ in the total order, and whose bases are
\(
    \mathcal{B}(\Pi_e M) = \mathcal{B}(M) \sqcup \{B - e + e' \mid e \in B \in \mathcal{B}(M)\}
\).
We write simply $\Pi M$ for $\Pi_e M$ when $e = \max E(M)$. 

\item The \textbf{series extension} of a matroid $M$ by an element $e \in E$ is the matroid $\Sigma_e M$ whose ground set $E(\Sigma_e M) = E(M) \sqcup \{e'\}$ is obtained by adding a new element $e'$ to $M$ immediately following $e$ in the total order, and whose bases are
\(
    \mathcal{B}(\Sigma_e M) = \{B + e' \mid B \in \mathcal{B}(M)\} \sqcup \{B + e \mid e \notin B \in \mathcal{B}(M)\}
\).
We write simply $\Sigma M$ for $\Sigma_e M$ when $e = \max E(M)$. 

\item 
The \textbf{pruning}\footnote{This nomenclature is non-standard.} of a matroid $M$ is the matroid $\ove M$ whose ground set is $E(\ove M) = E(M) - (\Lamma(M) \cup \Gamma(M))$ with the induced total order and whose bases are $\mathcal{B}(\ove{M}) = \{B - \Gamma(M) \mid B \in \mathcal{B}(M)\}$.

\item The \textbf{direct sum} $M_1 \oplus M_2$ of matroids $M_1$ and 
$M_2$ is the usual direct sum on the disjoint union $E(M_1 \oplus M_2) \deq E(M_1) \sqcup E(M_2)$, whose bases are 
\(
    \mathcal{B}(M_1 \oplus M_2) = \{B_1 \sqcup B_2 \mid B_1 \in \mathcal{B}(M_1), B_2 \in \mathcal{B}(M_2)\}.
\) 
The total order on $E(M_1 \oplus M_2)$ is given by
the \textbf{sum} $E(M_1) + E(M_2)$:
the total order which restricts to the given total orders on $E(M_1)$ and $E(M_2)$ and in which every element of $E(M_1)$ is less than every element of $E(M_2)$.
\end{itemize}
We note that the direct sum operation as we defined it is \emph{not} commutative. It is however associative, which allows us to define the direct sum of any (finite) sequence of matroids.
The empty direct sum is by definition the \textbf{empty matroid} (i.e.\ the unique matroid on the empty set, with unique basis $\varnothing$).

Up to isomorphism, there is a unique matroid of rank $1$ on a two-element ground set with no loops (equivalently, no coloops). 
It is the matroid with ground set $[2]$ and bases $\{1\}$ and $\{2\}$. We call such a matroid a \textbf{parallel pair}.
This matroid is usually denoted $U_{1,2}$.
\begin{Df}
    With the conventions above, a \textbf{snake} is a matroid obtained from a parallel pair by repeatedly taking parallel and series extensions by the \emph{maximum element}.
    A \textbf{snake den} is any (possibly empty) direct sum of snakes. The \textbf{components} of a snake den are the summands of this direct sum (which are uniquely determined).
    A \textbf{slither} is any matroid whose pruning is a snake den. 
\end{Df}
\begin{Ex}
    There is a unique snake den of size 0, namely the empty matroid with set of bases $\mathcal{B} = \{\varnothing\}$.
    There are no snake dens on a ground set of size 1. The parallel pair is the unique snake den of size 2. 
\end{Ex}
Each snake is a snake den and each snake den is a slither. 
It will sometimes be convenient to include the empty matroid among the snakes. For this purpose, we will use the term \textbf{extended snake} to refer to a snake den which is either a snake or the empty matroid.

\begin{Df}
    Let $M$ and $N$ be matroids of the same rank on the same (totally ordered) ground set.
    We say that $N$ is a \textbf{weak image} of $M$ if every basis of $N$ is a basis of $M$.\footnote{This nomenclature is a more restrictive variant of the usual notion for matroids. One might more precisely call it a \emph{rank-preserving weak image induced by the identity map}.}
    If $M$ is a snake den and $N$ is a slither, we also say that $N$ is \textbf{subordinate} to $M$.
\end{Df}
Snake dens represent the objects of our slither category. Slithers correspond to morphisms. Let us see how to compose such morphisms.
\begin{Df}\label[Df]{df:slither_morphism}
    A \textbf{morphism} from a snake den $N$ to a snake den $M$ is the data of a slither $S$ on $E(M)$ subordinate to $M$ such that $N$ is isomorphic to 
    $\ove{S}$. Each snake den is its own identity morphism.

    Composition of slithers $S\colon N \to M$ and 
    $T\colon P \to N$ is constructed as follows.
    Let $f\colon E(N) \to E(\ove{S})$ be the isomorphism.
    We define a weak image slither $S \circ T$ of $M$ by
    \[
        \mathcal{B}(S \circ T) = 
        \{f(B) \sqcup \Gamma(S) \mid B \in \mathcal{B}(T)\}
    \]
    Then $\ove{S \circ T}$ is isomorphic to 
    $P$ (which shows that $S \circ T$ is a slither).
    It is straightforward to verify the identity and associativity axioms for categories hold for morphisms so defined.
\end{Df}
We note that though a slither determines its domain up to isomorphism, the same slither can represent  morphisms with truly different codomains.
However, as the notation suggests, the slither representing the composition $S \circ T$ depends only on $S$ and $T$ themselves. Slithers $S, T$ are composable if and only if $T$ is isomorphic to a weak image of $\ove{S}$.

\subsection{Snake polytopes}\label{subsec:snake_polytopes}

As discovered by Edmonds \cite{Edmonds} and later Gel'fand--Goresky--MacPherson--Serganova \cites{GS87, GGMS},
matroids on $E$ are naturally thought of as lattice polytopes in $\R^E$. (In particular, standard matroids with $n$ elements correspond to lattice polytopes in $\R^n$.) 
Consistent with our convention for matroids, we will consider lattice polytopes in $\R^E$ where $E$ is a finite totally ordered set.
\begin{Df}
    Let $M$ be a matroid and $E = E(M)$.
    The \textbf{base polytope} of $M$ is the convex hull in $\R^E$ of the lattice points $e_B$ for $B \in \mathcal{B}(M)$ where $e_B \deq \sum_{b \in B} e_b$ and $e_b \in \R^{E}$ is the standard basis vector in $\R^E$ indexed by $b$.
    We write $P(M)$ for this polytope.
    If $M$ is a snake (resp.\ snake den) then $P(M)$ is a \textbf{snake polytope} (resp.\ \textbf{snake den polytope}).
    If $s$ is a slither word, we write 
    $P(s)$ for the \textbf{slither polytope} $P(M(s))$.
\end{Df}

The matroid polytopes of snakes have a particularly simple inductive structure, which will be explored in \Cref{sec:homology}. 
An important property of all matroid polytopes is their edge directions.
Recall that a $0/1$-polytope in $\R^E$ is one all of whose vertices are vertices of the standard cube $\{0,1\}^E$.
\begin{Thm}[{\cite{GGMS}*{Theorem 4.1}}]\label{thm:edge_dirs}
    A $0/1$-polytope $P$ in $\R^E$ is the base polytope of a matroid if and only if each edge of $P$ is parallel to $e_x - e_y$ for some 
    $x,y \in E$.
\end{Thm}
\noindent
\Cref{thm:edge_dirs} implies that each face of a matroid base polytope is again a matroid base polytope.
It is clear that a matroid base polytope determines the corresponding matroid.
We now describe in polytopal terms some of the operations on matroids in \Cref{subsec:matroids}. Below, 
we write $e_b^*$ for the dual basis vector to $e_b$.
\begin{itemize}
\item 
Given a $0/1$-polytope $P$ in $\R^E$, we write $P^*$ for the image of $P$ under the affine-linear map 
$x \mapsto e_{E} - x$, thought of as a polytope in 
$\R^{E^*}$. 
We then have $P(M^*) = P(M)^*$ for any matroid $M$.

\item For any polytope $P$ in $\R^E$
and any subset $S \subseteq E$, the \textbf{$S$-initial face} $\D{P}{S}$ of $P$ is the face which maximizes the linear functional 
$\sum_{s \in S} e_s^{\ast}$. The \textbf{$S$-terminal face} $\U{P}{S}$ of $P$ is the face which minimizes it. If $P$ is a matroid polytope, then $\U{P}{S} = \D{P}{E-S}$. For a matroid $M$ on $E$, 
we have $P(\D{M}{S}) = \D{P(M)}{S}$
and $P(\U{M}{S}) = \U{P(M)}{S}$.
\item 
For any polytope $P$ in $\R^E$, the \textbf{bound coordinates} of $P$ are those $b \in E$ such that the $b$-th coordinate function is constant on $P$;
the remaining ones are the \textbf{unbound coordinates}.
The \textbf{pruning} of $P$ is the projection of $P$ to $\R^S$ where $S \subseteq E$ is the subset of \emph{unbound} coordinates. (Here $S$ has the induced total order.)
Then $P(\ove{M}) = \ove{P(M)}$ for any matroid $M$.

\item Given polytopes $P \subseteq \R^{E_1}$, $Q \subseteq \R^{E_2}$, their \textbf{product} is the polytope $P \times Q \subseteq \R^{E_1 + E_2}$,
where $E_1 + E_2$ is the sum of totally ordered sets.
Then $P(M_1 \oplus M_2) = P(M_1) \times P(M_2)$ for any matroids $M_1, M_2$.
\end{itemize} 
When $P = P(M)$ is a matroid base polytope
and $e \in E = E(M)$,
we write $\Pi_e P$ (resp.\ $\Sigma_e P$) for 
$P(\Pi_e M)$ (resp.\ $P(\Sigma_e M)$). We similarly write $\Pi P$ and $\Sigma P$ when $e = \max E$.
We defer to \Cref{sec:homology} the explanation of the effects of these operations on matroid base polytopes.
We now define morphisms for (matroid base) polytopes. Only in \Cref{subsec:equiv} will we see that this agrees with our definition of morphisms for snake dens.
\begin{Df}\label[Df]{polytope_morphism}
Let $S, T$ be totally ordered sets.
An affine-linear map $F\colon \R^{S} \to \R^{T}$ is a \textbf{coordinate map} if there exists an increasing injection $f\colon S \to T$ such that $F$ is given coordinate-wise by 
$F(x)_t = x_s$ whenever $t = f(s)$ for some $s \in S$
and $F(x)_t$ is constant for $t \notin f(S)$. A \textbf{coordinate morphism} from a polytope $P$ in $\R^S$ to a polytope $Q$ in $\R^T$ is a coordinate map $F\colon \R^S \to \R^T$ which takes $P$ bijectively onto a face of $Q$.
\end{Df}
\noindent
Polytopes with coordinate morphisms form a category, in which composition of morphisms is simply composition of affine-linear maps and
the identity map $F\colon \R^E \to \R^E$ gives the identity of each polytope.

If $P$ and $Q$ are lattice polytopes defining toric varieties $X_P$ and $X_Q$ and $F\colon P \to Q$ is a coordinate map taking $P$ to a face $R$ of $Q$, then $F$ defines an embedding of $X_P$ into $X_Q$ with image $V(R)$ taking the distinguished point of $X_P$ to that of $O_{R}$. In particular, this is true for coordinate morphisms between snake den polytopes, which we will identify with slithers in \Cref{subsec:equiv}. 

\subsection{Slither words}
In this subsection, we describe a labeling of isomorphism types of snake dens and slithers by words in the alphabet 
$\pr, \rp, 0,1,\elll, \llle$. 
We will write $\abs{w}$ for the length of a word $w$.

We start by encoding snakes. A snake $M$ is given by a sequence of parallel and series extensions. At each step, we extend the maximum element of the ground set. Thus the matroid is specified (up to order-preserving isomorphism) by the sequence of choices made at each step. If we represent parallel extension by $0$ and series extension by $1$ and concatenate these from left to right as we go, a snake is thereby encoded by a binary string $\mathring{w}(M)$. 
This defines a bijection between binary strings and (standard) snakes.
For example, the string $01011$ represents the snake 
$\Sigma \Sigma \Pi \Sigma \Pi U_{1,2}$.
We then define $w(M) \deq \pr \mathring{w}(M) \rp$, in which we have prepended and appended the symbols $\pr, \rp$.
Then $\abs{w(M)} = \abs{E(M)}$ and the number of occurrences of the character $1$ in $w(M)$ is one fewer than the rank of $M$.

Next, we encode a snake den as follows. Suppose 
$M = N_1 \oplus \dots \oplus N_k$ where each $N_j$ is a snake.
Then we define $w(M)$ to simply be the concatenation 
$w(N_1)w(N_2)\dots w(N_k)$. Finally, suppose $M$ is a slither. 
Suppose $\abs{E(M)} = n$ and $\abs{E(\ove{M})} = m$, and let 
$f\colon [m] \to [n]$ be the standardization of the inclusion map $E(\ove{M}) \hookrightarrow E(M)$.
Let $e_1 < \dots < e_n$ be the elements of $E(M)$.
Then
$w(M)$ is the word of length $n$ defined character-by-character by
\[
    w(M)_i = \begin{cases}
        w(\ove{M})_{j} & \text{if $i = f(j)$}\\
        \elll & \text{if $e_i \in \Lamma(M)$}\\
        \llle & \text{if $e_i \in \Gamma(M)$}
    \end{cases}
\]
This indeed encodes precisely the data of the isomorphism type of $M$. 
We have $\abs{w(M)} = \abs{E(M)}$ and the rank of $M$ is the total number of occurrences of the characters $\pr, 1, \llle$ in $w(M)$. By analogy with matroids, given a slither word $s$, we write $\Lamma(s)$
(resp.\ $\Gamma(s)$) for the set of indices at which $s$ has the character $\elll$ (resp.\ $\llle$). The indices in 
$\Lamma(s) \cup \Gamma(s)$ are \textbf{bound}; the other indices are \textbf{unbound}.
\begin{Ex}
    The \textbf{empty word} $\eps$ is associated to the empty matroid. The word $\pr \rp$ represents the matroid $U_{1,2}$. The word 
    $\pr 0\elll \rp \llle\pr01\rp$ represents a slither whose pruning is the snake den 
    $\Pi U_{1,2} \oplus \Sigma \Pi U_{1,2}$.
\end{Ex}
\begin{Df}
    A \textbf{slither word} is a (possibly empty) word in the alphabet 
    $0,1,\pr, \rp,\elll,\llle$ such that 
    (i) the subword consisting of all parentheses has the form $\pr\rp\pr\rp\dots\pr\rp$ (but is possibly empty),
    (ii) each $0$ or $1$ is located between an opening $\pr$ and the subsequent closing $\rp$.
\end{Df}
\noindent
Evidently, the slither words are precisely the words of slithers, so are in bijection with standard slithers. We write $M(s)$ for the standard slither corresponding to a slither word $s$. When $s$ is a binary string, we also write $M(s)$ for $M(\pr s \rp)$. Similarly, given a slither word $s$, we write $P(s)$ for the corresponding matroid base polytope 
$P(M(s))$.

\begin{warning}
We write $\varnothing$ for the empty binary string. It is distinct from $\eps$, which is the empty slither word. We identify $\eps$ with the imagined binary string of length $-1$ discussed in \Cref{ssec:hopf_stuff}.
\end{warning}

\begin{figure}[ht]
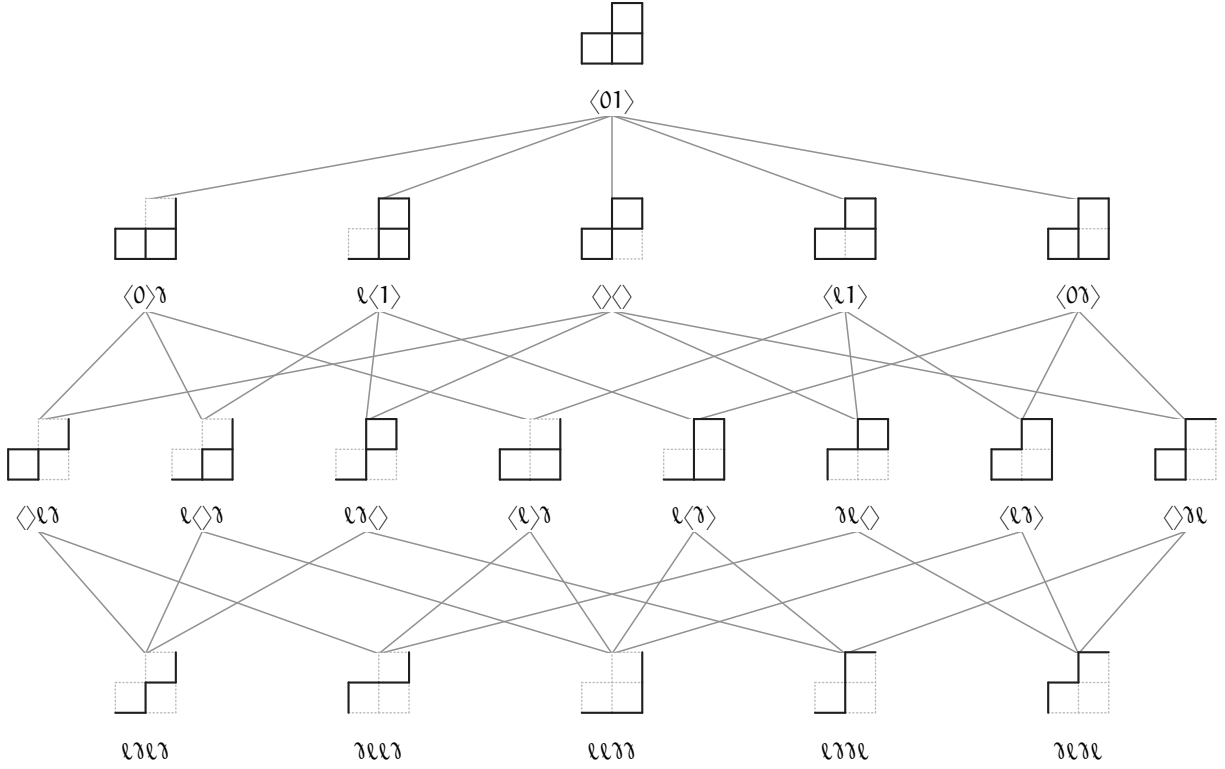

    \centering
    {\includestandalone[width=1\textwidth]{figures/tuong_poset}}
    \qquad
    \caption{All sub-slithers of the snake $\pr 0 1 \rp$ with associated words and slither diagrams. See \Cref{subsec:richardson_and_snake} for the meaning of the diagrams.} 
    \label{fig:slither_poset}
\end{figure}
\noindent
We now express some of the matroid operations in terms of slither words.
\begin{itemize}
    \item
    Let $s$ be a slither word. The \textbf{dual} $s^*$ of $s$ is the slither word obtained from $s$ by (i) interchanging the characters 
    $\elll$ and $\llle$, 
    (ii) interchanging the characters $\pr$ and $\rp$,
    (iii) interchanging the characters $0$ and $1$,
    and (iv) reversing the order of the characters. 
    Equivalently, $s^*$ is obtained from $s$ by flipping the text upside down and reading $\rotatebox[origin=c]{180}{0}$ as $1$ and $\rotatebox[origin=c]{180}{1}$ as $0$. We then have $M(s^*) = M(s)^*$.
    \item For a slither word $s$, we define the \textbf{pruning} $\ove{s}$ of $s$ to be the word obtained from $s$ by removing every occurrence of the characters 
    $\elll$ and $\llle$.
    We then have $M(\ove{s}) = \ove{M(s)}$.
    \item Given slither words $s, t$ we write 
    $st$ for their concatenation, which is again a slither word. 
    We then have
    $M(st) = M(s) \oplus M(t)$.
\end{itemize}
\noindent
Though in general morphisms are not as simple from the point-of-view of slither words, \emph{composition} of morphisms is pleasingly simple in this setting.
\begin{Df}
    Let $w$ and $w'$ be words of length $n$ and $n'$ respectively in the alphabet $\pr, \rp, 0, 1, \elll, \llle$. Assume that the length of $\ove{w}$ is $n'$.
    We define the \emph{composition} $w \circ w'$ as follows. 
    Let $i_1, \dots, i_{n'}$ be the indices at which $w$ has a character in $\{\pr,\rp,0,1\}$.
    Then
    \[
        (w \circ w')_i  \deq \begin{cases}
            \elll & \text{if $w_i = \elll$,}\\
            \llle & \text{if $w_i = \llle$,}\\
            w'_j & \text{if $i = i_j$.}
        \end{cases}
    \]
    This is an associative (partial) operation on words.
\end{Df}
\begin{Prop}
    Let $S$ and $T$ be composable slithers.
    Then $w(S \circ T) = w(S) \circ w(T)$.
\end{Prop}
\begin{proof}
    Using that $\ove{s \circ t} = \ove{t}$ for slither words $s, t$,
    and that $\ove{S \circ T} \cong \ove{T}$ for slithers $S, T$, we get
    \[
        \ove{w(S) \circ w(T)} = \ove{w(T)} = w(\ove{T}) = w(\ove{S \circ T})
        = \ove{w(S \circ T)}.
    \]
    It therefore remains only to show that $w(S \circ T)$ and $w(S) \circ w(T)$ have the characters $\elll$ and $\llle$ at the same positions.
    Suppose $S$ has length $n$ and $T$ has length $m$.
    Let $f\colon E(T) \to E(\ove{S}) \subseteq E(S)$ be the inclusion of  $T$ as a weak image of  $\ove{S}$ and let 
    $\hat{f}\colon [m] \to [n]$ be its standardization.
    Then the desired claim for $\elll$ follows by comparing $\Lamma(S \circ T) = \Lamma(S) \cup f(\Lamma(T))$ with
    $\Lamma(w(S) \circ w(T)) = \Lamma(w(S)) \cup \hat{f}(\Lamma(w(T)))$. The claim for $\llle$ follows by a similar observation for $\Gamma$.
\end{proof}

\subsection{Torus orbit closures in the Grassmannian}\label{sec:t_orb}

We briefly review the relationship between torus orbit closures in the Grassmannian, realizable matroids, and matroid base polytopes. For the Grassmannian, we continue the notation of \Cref{sec:richardson}. 

Let $E$ be a totally ordered finite set.
Recall the Plücker embedding  of $\Gr(r, E)$ into $\Pb^{\binom{E}{r} - 1}$. If $A_L$ is a full-rank $r \times E$ matrix with entries in $\Bbbk$ whose row span is $L \in \Gr(r, E)$, then the Plücker coordinates of $L$ are $(p_S(L))_{S \in \binom{E}{r}}$ where $p_S(L)$ is the determinant of the square submatrix of $A_L$ indexed by the columns in $S$. Every rational point $L \in \Gr(r, E)$ gives a \textbf{linear realization} of the matroid $M_L$ whose bases are 
\[
    \mathcal{B}(M_L) = \left\{B \in \binom{E}{r} \;\bigg|\; p_B(L) \neq 0\right\}. 
\]
Any matroid of the form $M_L$ is said to be \textbf{linearly realizable} over $\Bbbk$. 
A matroid is \textbf{binary} if it is realizable over the field with two elements $\mathbb{F}_2$.

For a point $L \in \Gr(r, E)$, the associated $T_E$-orbit closure $\ove{T_E\cdot L} \subseteq \Gr(r, E)$ is a normal toric variety \cite{MichalekSturmfels2021}*{\textsection~13.2}. The relationship between this toric variety and the matroid base polytope of $M_L$ is as follows. 

\begin{Thm}[\cite{GGMS}]
    The moment polytope of $\ove{T_E \cdot L}$ is the base polytope $P(M_L)$. 
\end{Thm}
\noindent
Let $X_{P(M_L)}$ be the toric variety associated to $P(M_L)$. The map 
\[
    \phi_L \colon X_{P(M_L)} \to \Gr(r, E)
\]
induced by sending $1$ in the dense torus to $L$ and extending $T_E$-equivariantly gives an isomorphism between $X_{P(M_L)}$ and $\ove{T_E \cdot L}$. Moreover, if $x_B$ is the fixed point of $X_{P(M_L)}$ corresponding to a basis $B \in \mathcal{B}(M_L)$, then $\phi_L(x_B) = \Bbbk^B$. In particular, the fixed points of $\ove{T_E \cdot L}$ are $\Bbbk^B$ for $B \in \mathcal{B}(M_L)$. Note that the fixed points of a torus orbit closure depend only on the matroid $M_L$ and not the realization $L$. 

In general, there can be many torus orbit closures in $\Gr(r, E)$ which have the same moment polytope. However, if a matroid $M$ is \textbf{regular}, i.e. linearly realizable over every field, this does not happen. 
\begin{Thm}[{\cite{Oxley}*{Proposition 6.3.12, 6.6.5}}]\label{thm:unique_torus}
    Let $M$ be a regular matroid. Then there is a unique $T_E$-orbit closure in $\Gr(r, E)$ whose moment polytope is $P(M)$. 
\end{Thm}
Matroids obtained via series and parallel extension starting from $U_{1,2}$ are called \textbf{series-parallel}. All series-parallel matroids are regular; in particular, this includes snake matroids.
Direct sum preserves regularity, so snake dens are also regular.
Since slithers are obtained from these by adding loops and coloops, these too are regular.
Thus every slither matroid has a unique $T_E$-orbit closure in $\Gr(r, E)$. In particular, all linear realizations of a slither lie in a common $T_E$-orbit inside $\Gr(r, E)$.

Snake dens and slithers also enjoy another remarkable property: they are all positroids. This follows from 
(i) snake dens are lattice path matroids
(as shown in \Cref{subsec:richardson_and_snake})
(ii) lattice path matroids are positroids \cite{oh}*{Lemma 21}, 
and
(iii) adding loops or coloops to a positroid results in another positroid. 
\begin{Df}[{cf.\ \cite{postnikov06}*{Definitions 3.1, 3.2}}]\label{df:pos}
    A \textbf{positroid} is a matroid realized over $\R$ by a subspace whose nonzero Pl\"ucker coordinates all have the same sign.
\end{Df}

In fact, the most naive attempt at a realization already works for regular positroids.

\begin{Prop}[{cf. \cite{QR25}*{\textsection \textsection~4-5}}]\label[Prop]{prop:easy_realization}
    Let $M$ be a regular positroid of rank $r$ on $E$. Then, the Plücker coordinates given by $p_{I} = 1$ for $I \in \mathcal{B}(M)$ and $p_I = 0$ otherwise define a realization of $M$ over any field.
\end{Prop}
\begin{proof}
    The regularity of $M$ ensures that we can realize it over $\R$ using Pl\"ucker coordinates $(p_I)_{I \in \binom{E}{r}}$ satisfying $p_I \in \{\pm 1, 0\}$ for all $I$ \cite{Oxley}*{Lemma 2.2.21}. 
    Since $M$ is a positroid, we can similarly realize it over $\R$ using Pl\"ucker coordinates $(q_I)_{I \in \binom{E}{r}}$ satisfying $q_I \geq 0$ for all $I$.
    By \Cref{thm:unique_torus}, these two realizations are in the same torus orbit. This implies that
    we can find $a \in \R^\times$ and $c \in (\R^\times)^E$ for which $q_I = ac^I p_I$. Let $\widetilde c = \sgn(c)$ be the vector of signs of $c$. Then, $\widetilde q_I \deq \sgn(a)\widetilde c^I p_I$ satisfies $\widetilde q_I = \sgn(q_I) \geq 0$ and $|\widetilde q_I| = |p_I|$ so that $\widetilde{q}_I \in \{0,1\}$ for all $I$. Since $(\widetilde{q}_I)_{I \in \binom{E}{r}}$ is in the torus orbit of $({p}_I)_{I \in \binom{E}{r}}$, it represents a realization of $M$ over $\R$. Since the $\widetilde{q}_I$ are integers, the realization is equally valid over any field.
\end{proof}

For a slither matroid $M$, we will write $\phi_M\colon X_{P(M)} \to \Gr(r(M), E(M))$ for the morphism induced by the realization with Plücker coordinates given by \Cref{prop:easy_realization}. For a snake den $s$, we write $X_s$ for $X_{P(s)}$.

\subsection{Richardson varieties and slither diagrams}\label{subsec:richardson_and_snake}

Recall from Section~\ref{sec:richardson} that Richardson varieties in $\Gr(r, E)$ are indexed by skew shapes $\lambda/\mu$. Our goal in this section is to associate to each slither a ``generalized'' skew diagram, which we call a \textbf{slither diagram} and relate slithers to toric Richardson varieties. 
Recall the \textbf{Gale Order} on $\binom{[n]}{r}$ which is defined by $B' \leq_G B$ if and only if $|[k] \cap B'|\geq |[k]\cap B|$ for all $1 \leq k \leq n$. Equivalently, if $B' = \{b'_1 < \cdots < b_r'\}$ and $B = \{b_1 < \cdots < b_r\}$ then $B'  \leq_G B$ if and only if $b_i' \leq b_i$ for all $i$. 

Given comparable elements $A\leq_G B$ in $\binom{[n]}{r}$, there is an associated \textbf{lattice path matroid} $M[A, B]$ whose bases are the elements of the Gale interval $[A, B]_{G}$.
There is a bijection between $\binom{[n]}{r}$ and southwest-to-northeast lattice paths within the $r\times (n-r)$ rectangle which start at $(0,0)$ and end at $(n-r, r)$. A subset $A$ is associated with the lattice path whose north steps are at positions $A$ and whose east steps are at positions in $[n]\setminus A$. We represent a Gale interval $[A,B]_G$ by the skew diagram
in the $(n-r) \times r$ grid bounded by the corresponding lattice paths; see \Cref{fig:skew_diagrams}.
We identify this diagram with the set of all edges inside it, including all edges of the lattice paths corresponding to $A$ and $B$.
This convention (of including the edges of the bounding paths) allows one to recover both $A$ and $B$ from the diagram. 
\begin{figure}[ht]
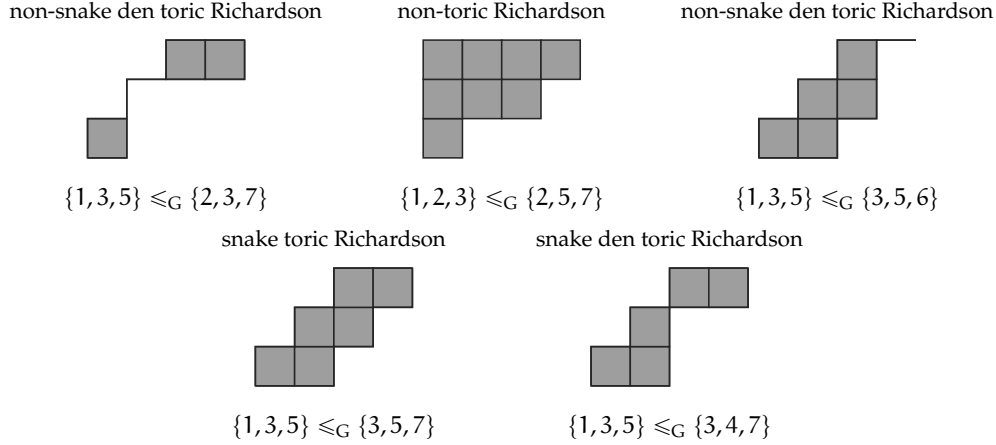

    \centering
    {\includestandalone[width=.8\textwidth]{figures/skew_diagrams}}
    \qquad
    \caption{Some skew diagrams associated to Gale intervals within the box $\binom{[7]}{3}$.}
    \label{fig:skew_diagrams}
\end{figure}
The bases of $M[A,B]$ are then identified with the lattice paths within the skew diagram corresponding to $[A,B]_G$. Write $\lambda_A$ for the partition whose lower right lattice path corresponds to $A$. Then $M[A, B]$ is the matroid realized by a \emph{general} point in the Richardson variety $\Omega_{\lambda_B}^{\lambda_A}$.
 In \Cref{sec:toric_positroids}, we show the following. 

\newtheorem*{CorRestated}{\bf Corollary \ref{cor:toric_richardsons_in_GR}}
\begin{CorRestated}
Let $X$ be a Richardson variety in the Grassmannian. Let $T$ be the
standard torus. The following are equivalent.
\begin{itemize}
    \item $X$ admits the structure of a toric variety.
    \item $X$ has a dense $T$-orbit.
    \item The skew diagram of $X$ contains no $2\times 2$ box.
\end{itemize}
\end{CorRestated}
Given \Cref{cor:toric_richardsons_in_GR}, 
skew diagrams containing no $2 \times 2$ box will be called \textbf{toric Richardson diagrams} and the associated lattice path matroids will be called \textbf{toric Richardson matroids}.
\begin{Ex}
    The parallel pair $U_{1,2}$ is the matroid 
    $M[\{1\}, \{2\}]$ in the $1\times1$ box.
\end{Ex} 
The effect of parallel (resp.\ series) extension by the last element in a connected skew diagram is to add a box immediately to the right of (resp.\ immediately above) the northeasternmost box. \textbf{Thus, every snake matroid is a toric Richardson matroid}.
More precisely, snake matroids are precisely those toric Richardson matroids whose skew diagrams are edge-to-edge connected. 
If $M$ and $N$ are lattice path matroids, their direct sum $M \oplus N$ is as well, and its skew diagram can be obtained by
joining the southwest corner of the skew diagram for $N$ to the northeast corner of the skew diagram for $M$. Since this process introduces no $2\times 2$ subgrid, \textbf{every snake den is a toric Richardson matroid}. See \Cref{fig:snakes_and_dens}.
\begin{figure}[ht]
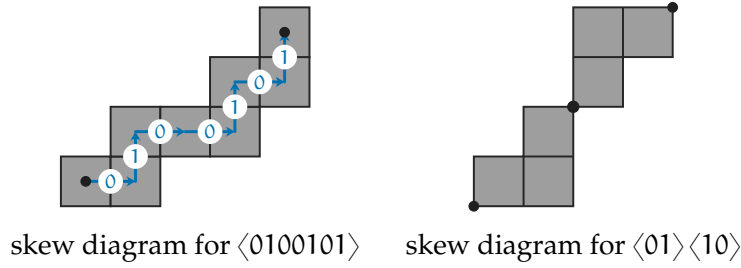

    \centering
    {\includestandalone[width=.6\textwidth]{figures/snakes_and_dens}}
    \qquad
    \caption{Skew diagrams for a snake and a snake den.}
    \label{fig:snakes_and_dens}
\end{figure}

We call an edge of a skew diagram a \textbf{vine} if it does not bound a box of the skew diagram. The vines are simply the edges of intersection of the upper and lower bounding paths of the skew diagram. The horizontal (resp.\ vertical) vines of a skew diagram correspond precisely to the elements of the ground set of the corresponding lattice path matroid which belong to none of (resp.\ all of) the bases; these are the loops (resp.\ coloops) of the matroid.
Thus pruning a toric Richardson matroid is equivalent to contracting all vines in its skew diagram (leaving connected ribbons joined corner to corner). This perspective makes it clear that the pruning of a toric Richardson matroid is a snake den. \textbf{Thus, every toric Richardson is a slither}; see \Cref{fig:snake-inclusions}.
\begin{figure}[htbp]
\centering
\resizebox{\linewidth}{!}{%
\begin{tikzpicture}[
  font=\large,
  line width=.5pt,
  rounded corners=5pt
]
  \draw[fill=black!1] (-1.7,-1.04) rectangle (14.5,1.04);
  \draw[fill=black!2] (-1.5,-.86)  rectangle (10.2,.86);
  \draw[fill=black!4] (-1.3,-.68)  rectangle (7.5,.68);
  \draw[fill=black!6] (-1.1,-.50)  rectangle (3.85,.50);
  \draw[fill=black!8] (-.9,-.32)   rectangle (1.05,.32);

  \node at (0,0)     {snakes};
  \node at (2.45,0)  {snake dens};
  \node at (5.675,0) {toric Richardsons};
  \node at (8.85,0)  {slithers};
  \node at (12.35,0) {regular positroids};
\end{tikzpicture}%
}
\caption{The inclusions among some classes of matroids.}
\label{fig:snake-inclusions}
\end{figure}
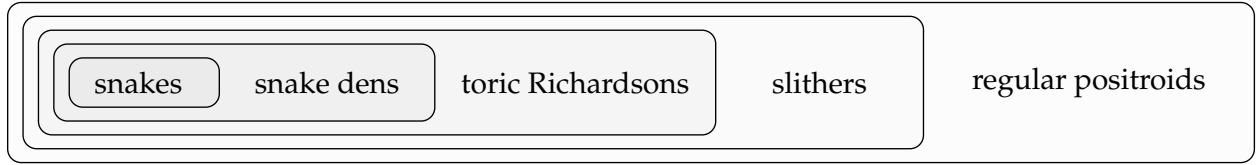

In fact, every slither is isomorphic to a toric Richardson matroid as an \emph{unordered matroid}. Since we are working with ordered matroids, however, the distinction will be maintained.
To any slither, we can associate a \textbf{slither diagram},\footnote{These are called \emph{block-tiled bands} in \cite{Facial_structures}.} which closely resembles a skew diagram.
We define a \textbf{slither diagram} to be the diagram obtained from a toric Richardson diagram by erasing some of the edges from its \emph{interior}. (Here, an edge is in the \emph{interior} of a diagram if both boxes that it bounds are contained in the diagram.)
The bases of the slither diagram are defined to be those lattice paths that can be drawn using the remaining edges, these being identified with their set of north steps. (We will soon see why this defines a matroid.)
When we wish to consider a slither diagram as a subdiagram of a given toric Richardson diagram, we draw those edges of the toric Richardson diagram \emph{not} lying in the slither diagram as dotted lines; cf.\ \Cref{fig:slither_poset}. We will call the bounded regions of the slither diagram its \textbf{scales}. 
\begin{figure}[htbp]
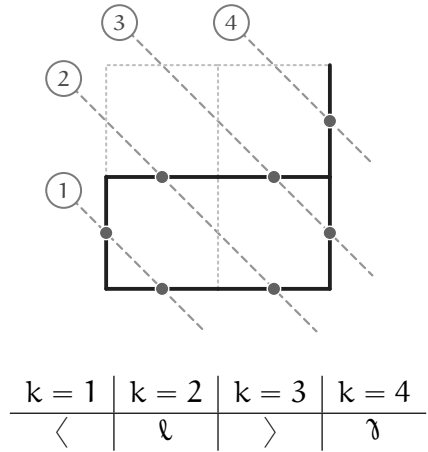

  \centering
  \includestandalone[width=.35\textwidth]{figures/slither_word_diagram}
  \caption{Reading a slither word from its diagram.}
  \label{fig:slither-word-antidiagonals}
\end{figure}
To show that each slither diagram corresponds to a slither, we first explain how to obtain the corresponding slither word. 
For this, we traverse the slither diagram from southwest to northeast, recording which edges of the diagram intersect each antidiagonal. At the
$k$-th antidiagonal, we transcribe:
\begin{itemize}
  \item An $\elll$ if the only edges in the slither diagram are
    horizontal.
  \item A $\llle$ if the only edges in the slither diagram are vertical.
  \item A $1$ if there are two vertical edges and one horizontal edge.
  \item A $0$ if there are two horizontal edges and one vertical edge.
  \item A $\pr$ if there is one vertical edge northwest of one horizontal
    edge.
  \item A $\rp$ if there is one vertical edge southeast of one horizontal
    edge.
\end{itemize}
We claim that the resulting word represents a slither matroid whose bases are precisely the bases of the slither diagram. This may be proven by induction on the number of interior edges that need to be removed from a toric Richardson diagram to obtain the given slither diagram. When this number is $0$, the claim may be verified directly using the previously established identification of toric Richardson matroids as particular slither matroids.
For the induction step, suppose that a horizontal edge has been deleted on the $k$-th antidiagonal. One obtains a slither diagram with one less missing edge by erasing all edges on the $k$-th antidiagonal and rejoining separated halves as in \Cref{fig:hydra}. By the induction hypothesis, the transcription procedure above produces the slither word of a matroid whose bases are precisely the bases of the new smaller diagram. Since the original slither diagram had only vertical edges on the $k$-th antidiagonal, reversing the deletion procedure introduces the element $k$ as a coloop of the slither diagram matroid. This agrees with the effect of adding an $\llle$ in the $k$-th position in the slither word, as prescribed by the transcription procedure. This proves the induction step in the case of a missing horizontal edge. 
A similar argument works for a missing vertical interior edge.
\begin{figure}[ht]
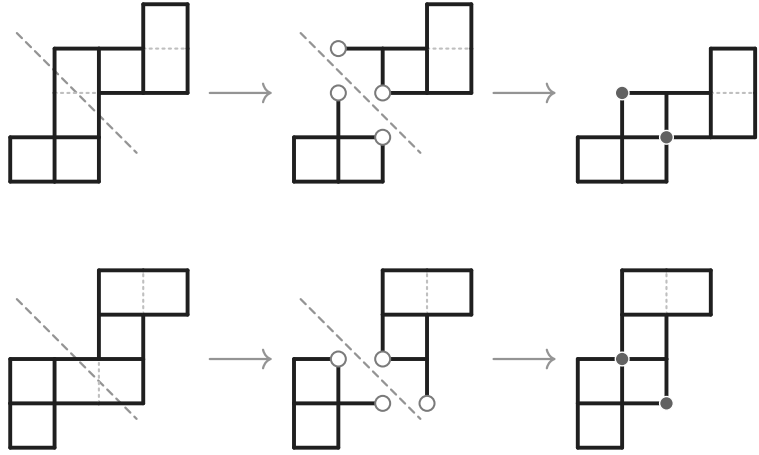

    \centering
    {\includestandalone[width=.6\textwidth]{figures/hydra}}
    \qquad
    \caption{Cutting and gluing a slither at a missing internal edge.}
    \label{fig:hydra}
\end{figure}

Conversely, we can show that each slither word represents a slither diagram by directly constructing one from it. We take the word as an instruction manual, read from left to right.
The drawing proceeds in two \emph{modes}:
\emph{pen mode} and \emph{box mode}.
We are in \emph{pen mode} when the symbol currently being read lies outside a matched pair $\pr\cdots\rp$, and in \emph{box mode} otherwise.
We start off in pen mode at the southwest corner of the grid.

When in pen mode, we draw a horizontal edge when we read an $\elll$ and a vertical edge when we read a $\llle$. A $\pr$ brings us into box mode and tells us to draw a square immediately northeast of our current position.
In box mode, a $0$ tells us to draw a box immediately to the right of the last one we drew and a $1$ tells us to draw a box immediately above it instead.
An $\elll$ (resp.\ a $\llle$) in this mode acts much like a $0$ (resp.\ a $1$) except that we then erase the edge that separates the old box from the new one.
A $\rp$ tells us to exit box mode and place ourselves at the northeast corner of the last drawn box. 

By refusing to erase any edges in the procedure above, we instead construct a toric Richardson diagram of which the slither diagram is obtained by deleting interior edges, showing that it is indeed a slither diagram. Comparing the construction with the earlier transcription rule for slither diagrams shows that the slither diagram we have obtained indeed corresponds to the correct matroid and slither word.

Pruning is a very natural operation on slither diagrams. Namely, it contracts all vines and collapses all scales to a single box. \Cref{fig:pruning_slither} illustrates this for a small slither.
Since the dimension of a Richardson variety is equal to the number of boxes in its skew diagram, and since pruning preserves the dimension of a matroid polytope, the dimension of the toric variety corresponding to a slither is equal to the number of scales in its diagram. From the construction described above of a slither diagram from its slither word, we deduce:
\begin{Obs}\label{obs:dim_count}
For any slither $s$, the dimension of $P(s)$ is equal to the number of characters among $\pr, 0, 1$ in the corresponding slither word. 
In particular, $\dim P(\beta) = \abs{\beta} + 1$ for any binary string $\beta$.
\end{Obs}

\begin{figure}[ht]
    \centering
    {\includestandalone[width=.6\textwidth]{figures/pruning_slither}}
    \qquad
    \caption{The effect of the pruning $\pr \rp \pr \llle \elll \elll 1 \elll \llle \rp \elll \llle \pr 0 \elll \elll 1 \elll \rp \elll \elll \mapsto \pr \rp \pr 1 \rp \pr 0 1 \rp$. }
    \label{fig:pruning_slither}
\end{figure}

\subsection{Equivalence of two notions of morphism}\label{subsec:equiv}

\Cref{subsec:matroids} gave a notion of morphism for snake dens.
\Cref{subsec:snake_polytopes} gives a notion of coordinate morphism for snake den polytopes.
Here we show that these two notions coincide. 
Let $M$ and $N$ be snake den matroids and suppose we are given a coordinate morphism $F\colon P(N) \to P(M)$ with corresponding increasing injection $f\colon E(N) \to E(M)$. 
By design, each $b \in E(M) - f(E(N))$ is a bound coordinate of $F(P(N))$ and the pruning of $F(P(N))$ is isomorphic to $P(N)$; the same is then true of the corresponding matroids.
Since the vertices of $F(P(N))$ form a subset of the vertices of $P(M)$, the former is a weak image of the latter.
This shows that coordinate morphisms of base polytopes give rise to slither morphisms of the corresponding matroids. 
The converse also holds, provided that each weak image of a snake den matroid corresponds to a face of its base polytope.
This is indeed true, by a theorem of Lucas.
\begin{Thm}[{\cite{Lucas75}*{Theorem 6.18}}]
    If $M$ is a binary matroid and $N$ is a weak image of $M$, then there exists a sequence 
    $N = N_0, N_1, \dots, N_k = M$ of weak images of $M$ such that, for each $i$, $N_i$ is the $S_i$-initial matroid of $N_{i+1}$ for some subset $S_i \subseteq E(M)$.
\end{Thm}
Since the base polytope of an $S$-initial matroid is a face of the original, and since snake den matroids are regular (hence binary), this shows that each slither morphism gives rise to a coordinate map of the corresponding base polytopes.

Next, we show that all faces of slither base polytopes are again slithers.
To start, we examine the facets of snake matroid polytopes.
The following lemma follows from the \emph{proof} of {\cite{Facial_structures}*{Lemma 10}}.
We employ the \emph{initial matroid} notation of \Cref{subsec:matroids}.
\begin{Lm}\label[Lm]{lm:face_structure}
    Let $s = s_1 \cdots s_m$ be a binary string and let $M \deq M(s)$. Set $\widetilde s \deq s_1 s s_m$ with the convention that $\widetilde s = 11$ if $s = \varnothing$. The facets of $P(s)$ are in bijection with certain contiguous substrings of $\widetilde{s}$ of length $1$ or $2$, as follows.
    \[
        \begin{array}{c|c|c}
        \text{substring}
        &
        \text{facet matroid}
        &
        \text{facet outward normal}
        \\
        \hline
        \widetilde s_i=0
        &
        {\U{M}{i}}
        &
        -{e_i^*}
        \\
        \widetilde s_i=1
        &
        \D{M}{i}
        &
        {e_i^*}
        \\
        \widetilde s_i\widetilde s_{i+1}=01
        &
        \D{M}{[i]}
        &
        {e_{[i]}^*}
        \\
        \widetilde s_i\widetilde s_{i+1}=10
        &
        \U{M}{[i]}
        &
        -e_{[i]}^*
        \end{array}
    \]
    Here $e_S^* = \sum_{s \in S} e_s^*$ for $S \subseteq [m+2]$. These facet outward normals are merely \emph{representatives} (that restrict to the facet normals along the affine hull of $P$) since the polytope $P(s)$ is not full-dimensional.
\end{Lm}

We use the above lemma to show that each facet of a snake den polytope is the base polytope of a slither matroid and to identify the slither word of this matroid.
Namely, for a snake den $M$ with word $w = w(M)$,
its facets are the slithers whose slither words are obtained from $w$ by replacing either a character or a pair of adjacent characters with a different string of the same length, according to the schema in \Cref{fig:local_moves}.

\begin{figure}[ht]
    \centering
    {\includestandalone[width=.8\textwidth]{figures/local_moves}}
    \qquad
    \caption{The allowable substitutions to pass from a snake den word to the slither word for one of its facets. The first two rows depict a 1- or 2-character contiguous substring of a snake den word and the corresponding local picture of its snake den diagram. The last two rows depict the transformation to obtain the word and diagram of the facet.}
    \label{fig:local_moves}
\end{figure}
Since the polytope of a snake den is the \emph{product} of the polytopes of its constituent snakes, and since each substitution rule operates necessarily within a single component, the validity of these rules for snake dens will follow from their validity for snakes.
In the case of snakes, it may be deduced by examining the effect on the bases of the slither diagram and comparing this to \Cref{lm:face_structure}. We illustrate as an example the transformation $01 \mapsto \rp \pr$. 
If the characters $01$ are at positions $i-1$ and $i$ in $s$, this corresponds to the facet matroid $\D{M}{[i]}$. To see this, we draw the snake den diagram of $s$ along with the antidiagonal with equation $x+y = i$. 
The diagram makes an $\reflectbox{L}$-shaped turn where it intersects this line. Let $v$ be the inner corner of this $\reflectbox{L}$ shape.
We observe that the lattice paths in the diagram maximizing the number of north steps among the first $i$ steps are precisely those which pass through $v$. These are also precisely the lattice paths in the diagram obtained by deleting the edges near the outer corner of the  $\reflectbox{L}$ as depicted in \Cref{fig:01_example}. It is also easy to see that the word for the resulting snake den is obtained by replacing $01$ with $\rp\pr$ in $w$.
\begin{Cor}
    Each face of a slither matroid base polytope is again a slither matroid base polytope.
\end{Cor}
\begin{proof}
    We proceed by induction on the codimension of the face.
    By pruning, we may assume that this slither is a snake den. If the face is a facet, the claim follows from the discussion above. Otherwise, it lies in some facet.
    The claim then follows by applying the inductive hypothesis to this facet.
\end{proof}
\begin{figure}[ht]
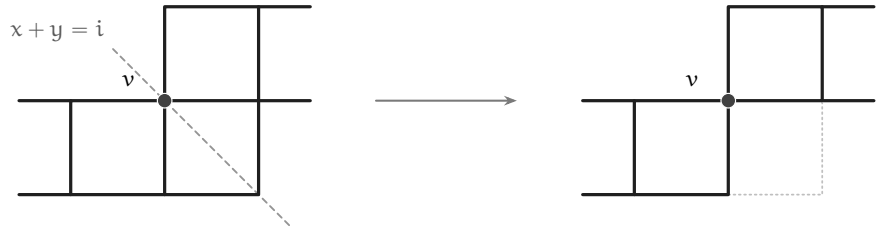

    \centering
    {\includestandalone[width=.7\textwidth]{figures/01_example}}
    \qquad
    \caption{Example of the modification associated to $M_{[i]}$.}
    \label{fig:01_example}
\end{figure}

Since we can identify a slither morphism $S\colon N\to M$ with the coordinate morphism $P(N) \to P(M)$ whose image is the face $P(S)$ of $P(M)$, we will frequently identify this face of $P(M)$ with the slither $S$ under its various guises. \Cref{fig:face_labels} gives an illustration; it depicts the slither words associated to each proper face of the snake polytope $P(01)$.

\begin{figure}[ht]
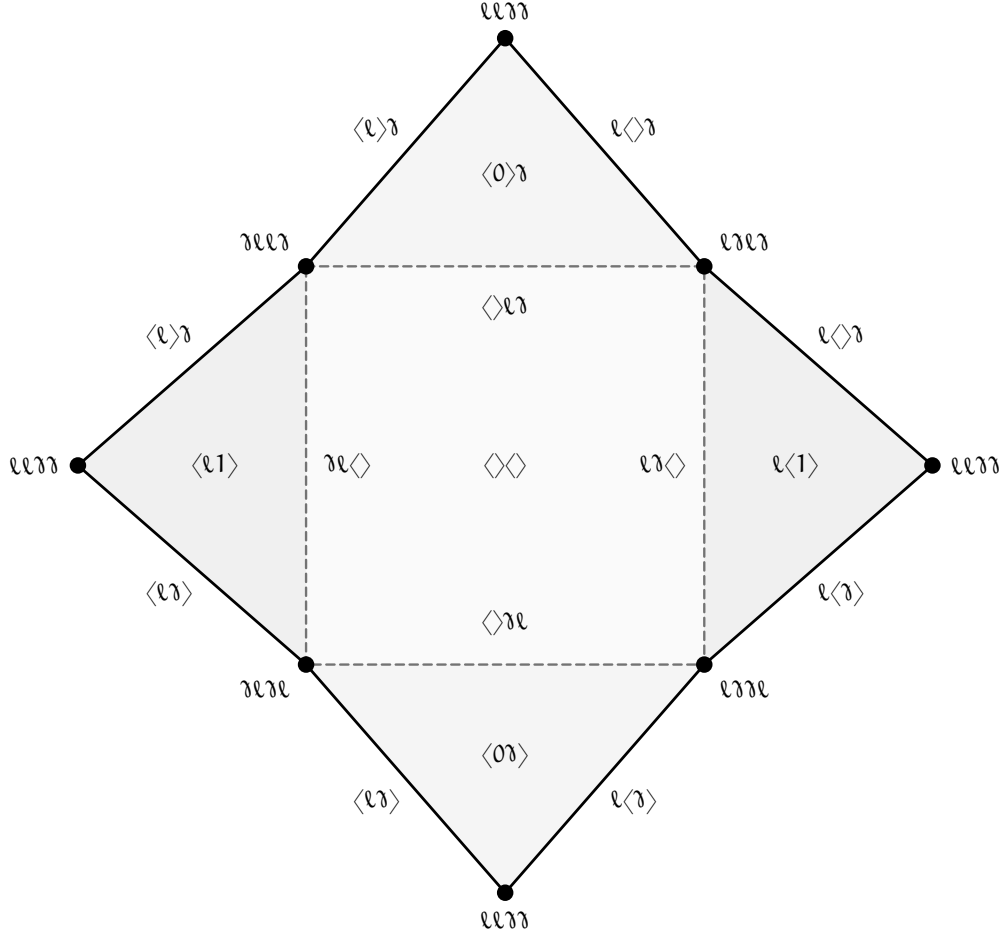

    \centering
    {\includestandalone[width=.8\textwidth]{figures/unfolded}}
    \qquad
    \caption{Net of the square pyramid $P(01)$ along with the slither words of its faces.}
    \label{fig:face_labels}
\end{figure}

\section{Homology}\label{sec:homology}
\subsection{An affine paving}\label{sec:paving}
To understand the behaviour of matroid base polytopes under parallel and series extension, we need the following notion, a special case of the similarly-named construction in \cite{haase21}*{\textsection~2.2.1}.
\begin{Df}
    Let $P$ be an integral lattice polytope in $\R^n$ and let $f\colon \R^n \to \R$ be an \emph{integral} affine-linear function\footnote{Here \emph{integrality} for $f$ means that $f(\Z^n) \subseteq \Z$.}  which is nonnegative on $P$.
    The \textbf{chimney polytope} $\Chim(P, f)$ is defined by
    \[
        \Chim(P, f) = \{ (x, h) \in \R^n \times \R \mid x\in P,\; 0 \leq h \leq f(x)\}.
    \]
    The \textbf{pivot} of the chimney is $H = \{x \in P \mid f(x) = 0\}$, which is a face of $P$
    provided $H \neq \varnothing$.\footnote{We do not consider the empty set a polytope, nor a face of any polytope.}
    For a face $F$ of $P$ we write
    \begin{itemize}
        \item $F^\top$ for the face $\{(x, f(x)) \mid x \in F\}$ of $\Chim(P, f)$, the \emph{top copy} of $F$.
        \item $F^\bot$ for the face $\{(x, 0) \mid  x \in F\}$ of $\Chim(P, f)$, the \emph{bottom copy} of $F$.
        \item $F^{\side}$ for the face $\{(x, h)\in \Chim(P, f) \mid x \in F\}$ of $\Chim(P, f)$, the \emph{chimney over $F$}. The nomenclature reflects the fact that, as a polytope, $F^{\side} = \Chim(F, f)$.
    \end{itemize}
    The faces of $\Chim(P, f)$ are exactly the faces 
    $F^{\top}, F^{\bot}, F^{\side}$ as $F$ ranges over the faces of $P$. These are pairwise distinct, except for the fact that $F^{\top}, F^{\bot}, F^{\side}$ all \emph{coincide} when $F$ lies in $H$.
    In particular, the vertices of $\Chim(P, f)$ are the vertices $p^{\side}$ for vertices $p \in H$ and 
    $p^{\bot}, p^{\top}$ for $p \in P - H$.
    These observations show that the combinatorial type of $P$ along with the knowledge of the pivot together determine the combinatorial type of $\Chim(P, f)$.
\end{Df}

\begin{Lm}\label[Lm]{lm:is_chimney}
    Let $M$ be a matroid and let $E = E(M)$.
    Let $P = P(M) \subseteq \R^E$ and let $j \in E$ be an unbound element. Let $Q = \Pi_j P$ (resp.\ $\Sigma_j P$) with newly added element $k$. Then there is a unimodular change-of-coordinates identifying $Q$ with a chimney over $P$ so
    that for each basis $B \in \mathcal{B}(M)$ 
    the vertex $e_{B}$ (resp.\ $e_{B+k}$) in $Q$ is identified with $(e_B)^{\bot}$
    and if $j \in B$ (resp.\ $j \notin B$) then the vertex $e_{B-j+k}$ (resp.\ 
    $e_{B+j}$) is identified with 
    $(e_B)^{\top}$.
\end{Lm}
\begin{proof}
    We work out the case of parallel extensions. The case of a series extension then follows using 
    $\Sigma_j P = (\Pi_j P^*)^*$.
    Let $M$ be the matroid of $P$ and $k$ the newly added element in $\Pi_j M$.
    For a basis $B \in \mathcal{B}(M)$
    set $B^{\bot} = B$ and let
    $B^{\top} = B - j + k$ if $j \in B$
    and $B^{\top} = B$ otherwise.
    Then the bases of $\Pi_j M$ are 
    $B^{\bot}$ and $B^{\top}$ for 
    $B \in \mathcal{B}(M)$.
    Working in the coordinate system with basis 
    $\{e_b \mid b \in E\} \cup \{e_{k-j} \deq e_k - e_j\}$, the vertices of $\Pi_j P$ are
    $e_{B^{\bot}} = e_B$ along with
    $e_{B^{\top}} = e_B + e_{k-j}$ for $B$ containing $j$. In this coordinate system,
    \[
        \Pi_j P = \{(x, h) \in \R^{E} \times \R \mid x \in P,\; 0 \leq h \leq x_j\} = \Chim(P, x_j).
    \]
    In the new coordinate system, the vertex $e_{B^{\bot}}$ is identified with $(e_B, 0)$ and 
    $e_{B^{\top}}$ is identified 
    with $(e_B, x_j(e_B))$.
    From this, it is clear that 
    $e_{B^{\bot}} = (e_B)^{\bot}$ and 
    $e_{B^{\top}} = (e_B)^{\top}$.
\end{proof}

\begin{Obs}\label{obs:up_and_down}
    Making the identification of $\Pi_j P$ (resp.\ $\Sigma_j P$) with the chimney over $P$ described in \Cref{lm:is_chimney}, 
    the pivot is the $j$-terminal face $\U{P}{j}$ 
    (resp.\ the $j$-initial face $\D{P}{j}$).
    Furthermore, $P^{\bot}$ is identified with $\U{Q}{k}$
    (resp.\ $\D{Q}{k}$)
    and $P^{\top}$ is identified with $\U{Q}{j}$
    (resp.\ $\D{Q}{j}$).
    The vertices of $P$ \emph{not} contained in the pivot are exactly those in the face $\D{P}{j}$ (resp.\ $\U{P}{j}$); its lower copy in $Q$ is 
    $(\D{P}{j})^{\bot} = \D{Q}{j}$
    (resp.\ $(\U{P}{j})^{\bot} = \U{Q}{j}$)
    and its upper copy is
    $(\D{P}{j})^{\top} = \D{Q}{k}$
    (resp.\ $(\U{P}{j})^{\top} = \U{Q}{k}$).
\end{Obs}

\begin{Df}\label[Df]{df:liftable}
    Let $(P, F_0, F_1, v)$ be the data of a polytope $P$, faces $F_0, F_1$ of $P$ and a covector $v$. We call such data \textbf{liftable} if
    \begin{enumerate}
        \item The vertex sets of $F_0$ and $F_1$ partition the vertex set of $P$,
        \item The covector simply paves $P$,
        \item For each vertex $p \neq P^{-v}$ in $F_0$ there is exactly one vertex $q$ in $F_1$ adjacent to $p$ such that $\inner{q, v} < \inner{p, v}$. The same holds with the roles of $F_0$ and $F_1$ reversed. 
    \end{enumerate}
\end{Df}
We note that the tuple $(P, F_0, F_1, v)$ is liftable if and only if the same is true of the tuple $(P, F_1, F_0, v)$.
In the context above, we will write $q \midarrow p$ to mean that $q$ and $p$ share an edge and $\inner{q, v} < \inner{p, v}$; we call this an \textbf{incoming arrow} to $p$.
Recall from \Cref{sec_BB} that for $P$ and $v$ as in the definition above, $v$ paves $P$ at a vertex $p$ if and only if the incoming arrows to $p$ are precisely those edges of a face $F$ of $P$ which are incident to $p$. In this case, we write $F = P_p^v$. It simply paves $P$ at $p$ if and only if the number of such arrows is exactly $\dim P_p^{v}$.

\begin{Prop}\label[Prop]{prop:lifting}
    Suppose $(P, F_0, F_1, v)$ is a liftable tuple.
    Let $Q$ be the chimney over $P$ with pivot $F_0$ and let
    $v' = (v, -\eps)$ for $0 < \eps \ltl 1$.
    Then the tuple $(Q, P^{\bot}, F_1^{\top}, v')$ is liftable. Moreover, for any vertex $p \in P$, we have $Q_{p^\bot}^{v'} = (P_p^v)^\side$ and if $p \in F_1$ then $Q_{p^\top}^{v'} = (P_p^v)^\top$.
\end{Prop}
\begin{proof}
    It is easy to see from the chimney construction that $Q$ is the vertex-wise union of $P^{\bot}$ and $F_1^{\top}$. This is condition (1) in \Cref{df:liftable}. Noting that $Q$ is the vertex-wise union of $F_0$, $F_1^{\top}$ and $F_1^{\bot}$, we will check conditions (2) and (3) simultaneously for the vertices in each of these faces.
    Note that the choice of $v'$ ensures that the edges in $P^{\top}$ and $P^{\bot}$ are directed just as in $P$.

    Consider first a vertex $q \in F_0$.
    If $q = Q^{-v'}$, there is nothing to prove.
    If $q \neq Q^{-v'}$ then $q \neq P^{-v}$. There is then a unique arrow $p \midarrow q$ in $P$ with $p\in F_1$.
    By hypothesis, the incoming arrows to $q$ in $P$ are the edges incident to $q$ in the face $P_q^v$ and are $\dim P_q^v$ in number.
    The arrows incoming to $q$ in $Q$ are the edges incident to $q$ in $(P_q^v)^{\bot}$ along with the additional arrow
    $p^{\top} \midarrow q$. Together, these are exactly the edges incident to $q$ in the face $(P_q^v)^\side$. Thus  $Q_{q}^{v'} = (P_q^v)^\side.$ Moreover, the number of incoming arrows is $1+\dim P_q^v = \dim (P_q^v)^\side = \dim Q_{q}^{v'}$. 

    Next, consider a vertex $q^{\bot} \in F_1^{\bot}$.
    It is adjacent to a unique vertex in $F_1^{\top}$, namely $q^{\top}$. This implies (3) for $q^{\bot}$.
    The incoming arrows at $q$ in $P$ are the edges incident to $q$ in the face $P_q^v$ and are $\dim P_q^v$ in number. The choice of $v'$ is such that the incoming arrows to $q^{\bot}$ are exactly the edges incident to $q^{\bot}$ in the face $(P_q^v)^{\side}$; that is, $Q_{q^{\bot}}^{v'} =  (P_q^v)^{\side}.$ Aside from $q^{\top} \midarrow q^{\bot}$, these all lie in $(P_q^v)^{\bot}$.
    Hence they total exactly $1 + \dim (P_q^v)^{\bot} = \dim (P_q^v)^{\side} = \dim Q_{q^{\bot}}^{v'}$ in number.

Finally, consider a vertex $q^{\top} \in F_1^{\top}$. If $q = P^{-v}$, then $q^{\top} = Q^{-v'}$ and there is nothing to prove. Suppose $q \neq P^{-v}$. The incoming arrows at $q$ in $P$ are the edges incident to $q$ in the face $P_q^v$ and number $\dim P_q^v$ in total.
The choice of $v'$ is such that all incoming arrows to $q^{\top}$ lie in $P^{\top}$.
Hence they are exactly the edges of the face $(P_q^v)^{\top}$ and number $\dim  (P_q^v)^{\top}$ in total, so $Q_{q^\top}^{v'} = (P_q^v)^{\top}$. The arrows $p \midarrow q^{\top}$ with $p \in P^{\bot}$ all have $p \in F_0$ and occur precisely when $p \midarrow q$ in $P$. By hypothesis, there is exactly one of these.
\end{proof}

The matroid base polytope $P = P(U_{1,2})$ trivially admits the liftable tuple $(P, e_1, e_2, e_2^*)$. Since each snake polytope is obtained from it by series and parallel extension,
\Cref{lm:is_chimney} shows that it is obtained from $P(U_{1,2})$ by iteratively taking chimneys.
Using \Cref{prop:lifting} and 
\Cref{obs:up_and_down} inductively one shows that for each $s$ there is a covector $v(s)$ such that $(P(s), \U{P(s)}{j}, \D{P(s)}{j}, v(s))$ is a liftable tuple where $j$ is the maximal element of $E(M(s))$.
Since matroid base polytopes are edge-unimodular, $v(s)$ smoothly paves $P(s)$ by \Cref{obs:tot_uni}.
In particular, the toric varieties of snake polytopes admit \BB affine pavings.
By taking products, one sees that the same is true for snake dens. Since
the pruning of a toric Richardson matroid is a snake den, 
it is the toric variety associated to a snake den polytope.
Consequently:
\begin{Thm}
Each toric Richardson variety in the Grassmannian admits an affine paving induced by a \BB decomposition.
\end{Thm}

\subsection{Bases, \BB cells, and subwords}\label{sec:bijection}

In this subsection we examine more closely the \BB decomposition of the toric varieties of snake matroids established in \Cref{sec:paving}.
We show that the pruning of each slither corresponding to a \BB cell is a snake and that pruning induces a bijection between these cells and the distinct subwords of the snake word.
The precise identification of the slithers corresponding to \BB cells will require the following definitions.
\begin{Df}\label{df:loose}
     We say that a lattice point in a snake diagram is \textbf{loose} if it sits at the top right corner of some box in the diagram. It is called \textbf{tight} otherwise. A horizontal (resp.\ vertical) \emph{edge} in the diagram is \textbf{loose} if its right (resp.\ top) endpoint is loose.
\end{Df}

\begin{Df}\label{df:base2word}
    Let $s$ be a snake and let $B$ be a lattice path in its skew diagram. 
    We define a binary string $l_s(B)$ and an extended binary string $\phi_s(B)$ as follows.
    To define $l_s(B)$, we traverse $B$ from southwest to northeast recording a $0$ for each loose horizontal edge we encounter and a $1$ for each loose vertical edge.
    We then define $\phi_s(B)$ to be the substring of $l_s(B)$ consisting of all characters that (strictly) follow the first occurrence of $1$. If no $1$ appears in $l_s(B)$, we set $\phi_s(B) = \eps$.
\end{Df}
\noindent
 We write $\ebs(s)$ for the set of extended binary substrings of $s$ (i.e.\ including $\eps$). 
\begin{Df}\label{df:fst}
    For $t\in \ebs(s)$, we will define a slither $r_s(t)$. (Here, ``$r$'' stands for \emph{rightmost}.) We start by identifying the binary string $1t$ with its rightmost embedding in $1s$.
    If $t = \eps$, we convene that $1t = \varnothing$.
    The string $r_s(t)$ is then obtained from $1s$ as follows; cf. \Cref{fig:BB_face_word}.
    \begin{itemize}
        \item If $t \neq \eps$, replace the leftmost $1$ in $1t$ with the symbol $\pr$.
        \item Replace every character which is in $1s$ but not in $1t$ using the rule $0 \mapsto \elll$ and $1 \mapsto \llle$.
        \item If $t \neq \eps$, append the symbol $\rp$. If $t = \eps$, append $\elll$ instead.
    \end{itemize}
\end{Df}
\noindent
    \begin{figure}[ht]
    \centering
    {\includestandalone[width=.7\textwidth]{figures/BB_face_word}}
    \qquad
    \caption{Obtaining $r_s(t)$ from the binary strings $s=010011010$ and $t=0100$.}
    \label{fig:BB_face_word}
\end{figure}

Starting in the following lemma and continuing until the end of this section, we fix for each snake $s$ a \BB paving vector $v(s)$ constructed inductively as described at the end of \Cref{sec:paving}.

\begin{Lm}\label[Lm]{lm:bijection1}
    Let $s$ be a binary string. Then $\phi_s$ maps $\mathcal{B}(M(s))$ bijectively onto $\ebs(s)$.
\end{Lm}
\begin{proof}
    We prove the statement by induction on the length of $s$. The base case $s = \varnothing$ is clear.
    For $s \neq \varnothing$, write $s = \widetilde s b$ for $b \in \{0, 1\}$. 
    We split the bases of $M(s)$ according to the last step in the lattice paths representing them.
    A basis $B$ of $M(s)$ obtained from a basis $\widetilde B$ of $M(\widetilde s)$ as in the first column of Figure~\ref{fig:subseqbij_2} satisfies $l_s(B) = l_{\widetilde s}(\widetilde B)b$. A basis $B$ of $M(s)$ obtained from a basis $\widetilde B$ of $M(\widetilde s)$ as in the second column of  Figure~\ref{fig:subseqbij_2} satisfies $l_s(B) = l_{\widetilde s}(\widetilde B)$ where $l_{\widetilde s}(\widetilde B)$ ends in $1-b$. Let $\mathcal{B}_1 \subseteq \mathcal{B}(M(s))$ be the set
    of bases of the first kind and write $\mathcal{B}_2 \subseteq \mathcal{B}(M(s))$ for those
    of the second kind.
    By the inductive hypothesis, $\phi_s$ maps $\mathcal{B}_1$ bijectively onto $\{ab | a \in \ebs(\widetilde s)\}$ and maps $\mathcal{B}_2$ bijectively onto $\{a | a \in \ebs(\widetilde s) \text{ ending in } 1-b\}$. Since $\mathcal{B}(M(s)) = \mathcal{B}_1 \sqcup \mathcal{B}_2$ and $\ebs(s)= \{ab | a \in \ebs(\widetilde s)\} \sqcup \{a | a \in \ebs(\widetilde s) \text{ ending in } 1-b\}$,
    this shows that $\phi_s$ gives the claimed bijection $\mathcal{B}(M(s)) \to \ebs(s)$.
    Note that to account for the exceptional behaviour of $\eps$ and $\varnothing$ we must convene that $\eps 0 = \eps$ and $\eps 1= \emptyset$; in particular, $\eps$ is deemed to end in $0$ and $\varnothing$ ends in $1$.     
\end{proof}

\begin{Lm}\label[Lm]{lm:bijection2}
    Let $s$ be a binary string. 
    For any basis $B \in \mathcal{B}(M(s))$ we have 
    $P(s)_{e_B}^{v(s)} = P(r_s(\phi_s(B)))$. In other words,
    $r_s(\phi_s(B))$ is the slither word of the face $P(s)_{e_B}^{v(s)}$ of $P(s)$. 
\end{Lm}
\begin{figure}[ht]
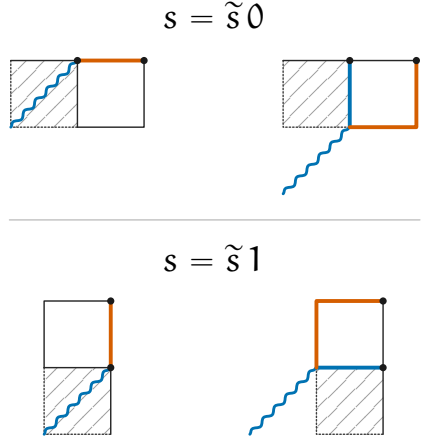

    \centering
    {\includestandalone[width=.33\textwidth]{figures/subseqbij_2}}
    \qquad
    \caption{Going from a basis of $s$ (in red) to a basis of $\widetilde s$ (in blue). 
    The blue lattice path represents the basis $\widetilde B$ whereas the red parts indicate where $B$ differs from $\widetilde B$.
    A squiggle represents the continuation of a lattice path whose precise behavior is unknown.}
    \label{fig:subseqbij_2}
\end{figure}
\begin{proof}
    We proceed by induction on the length of $s$.
     The statement for $s = \varnothing$
    is an easy consequence of our choice of paving vector for $U_{1,2}$. 
    For $s \neq \varnothing$, write $s = \widetilde{s} b$ with $b \in \{0, 1\}$.
    Set $F_0 \deq \U{P({\widetilde s})}{j}$ and $F_1 \deq \D{P(\widetilde s)}{j}$, where $j$ is the last element of $E(M(\widetilde s))$. \Cref{lm:is_chimney} shows that $P(s)$ is the chimney over $P(\widetilde s)$ with pivot $F_b$.
    The discussion at the end of \Cref{sec:paving} shows that  
    $(P({\widetilde s}), F_0, F_1, v(\widetilde s))$ is a liftable tuple.
    
    Let $B \in \mathcal{B}(M(s))$ be a basis.
    We show that $P(s)_{e_B}^{v(s)} = P(r_s(\phi_s(B)))$. 
    Suppose first that $B \in \mathcal{B}_1$ is constructed from a basis $\widetilde B$ as above. Then $e_B= p^\bot$, where $p = e_{\widetilde B}$. 
    \Cref{prop:lifting} now shows that 
    $P(s)_{p^\bot}^{v(s)} = (P({\widetilde s})_p^{v(\widetilde s)})^\side$.
    The inductive hypothesis gives $(P({\widetilde s})_p^{v(\widetilde s)})^\side = P(r_{\widetilde s}(a))^\side$, where $a \deq \phi_{\widetilde s}(\widetilde B)$. 
    Now $(P(r_{\widetilde s}(a)))^\side$ is $\Pi P(r_{\widetilde s}(a))$ if $b = 0$ and 
    $\Sigma P(r_{\widetilde s}(a))$ if $b = 1$.
    Since $r_{\widetilde s}(a)$ ends in the symbol $\rp$, the slither word for this series or parallel extension is constructed by inserting $b$ right before this final $\rp$. Since $\phi_s(B) =ab$ (with the convention $\eps0 = \eps$  and $\eps1=\varnothing$), it remains only to show that 
    $r_s(ab)$ is obtained from $r_{\widetilde s}(a)$ by inserting $b$ in this manner.
    This is indeed the case because under the rightmost embedding of $1ab$ into $1s$, the final $b$ in $1ab$ is sent to the final $b$ in $s = \widetilde s b$.

    The argument in the above paragraph
    works fine as long as $a \neq \eps$. We explain what happens in this exceptional case. 
    As above, we have  
    $P(s)_{e_B}^{v({s})} 
    = P(s)_{p^{\bot}}^{v(s)} = (P(\widetilde{s})_p^{v(\widetilde{s})})^{\side} 
    = P(r_{\widetilde{s}}(a))^{\side}$.
    Note that 
    $P(r_{\widetilde{s}}(a))$ consists of a single point and 
    $r_{\widetilde{s}}(a)$ ends in $\elll$ by construction. We now split into cases depending on the value of $b$.
    If $b = 0$,  then
    $P(r_{\widetilde{s}}(a))^{\side} = \Pi P(r_{\widetilde{s}}(a))
    = P(r_{\widetilde{s}}(a) \elll)$.
    Now $r_{\widetilde s}(\phi_{\widetilde{s}}(\widetilde B))$ (resp.\ $r_s(\phi_s(B))$) is obtained from $1\widetilde{s} \elll$ 
    (resp.\ $1s\elll = 1\widetilde{s}0\elll$)
    by making the substitutions $0 \mapsto \elll$ and $1 \mapsto \llle$.
    So $r_{\widetilde{s}}(a)\elll = r_s(\phi_s(B))$ and 
    $P(r_{\widetilde{s}}(a) \elll) = P(r_s(\phi_s(B)))$.
    If instead $b = 1$, then $P(r_{\widetilde{s}}(a))^{\side} = \Sigma P(r_{\widetilde{s}}(a))$.
    Since the last element in
    $M(r_{\widetilde{s}}(a))$ is a loop,
    the slither word for
    $\Sigma P(r_{\widetilde{s}}(a))$ 
    is obtained from $r_{\widetilde{s}}(a)$ by replacing the final $\elll$ with $\pr \rp$. We compare this with $r_s(\phi_s(B)) = r_s(\varnothing).$
    The latter is obtained by replacing the final $1$ in $1s = 1\widetilde{s}1$ by $\pr$, appending a $\rp$ and making the substitutions $0 \mapsto \elll$ and $1 \mapsto \llle$, which agrees with the slither word for $\Sigma P(r_{\widetilde{s}}(a))$.
    
    Now suppose that $B \in \mathcal{B}_2$, again constructed from a basis $\widetilde B \in \mathcal{B}(M(\widetilde s))$. Set $a = \phi_{\widetilde s}(\widetilde B)$. Then $e_B  = p^\top$ where 
    $p = e_{\widetilde B}$. Note that 
    $p \in F_{1-b}$. By \Cref{prop:lifting}, $P(s)_{p^\top}^{v(s)} = (P(\widetilde s)_p^{v(\widetilde s)})^\top$. By the inductive hypothesis, $(P(\widetilde s)_p^{v(\widetilde s)})^{\top} = P(r_{\widetilde s}(a))^\top$. 
    From the matroid description of the chimney construction in \Cref{lm:is_chimney} and \Cref{obs:up_and_down}, one sees that, 
    for any slither $\sigma$ subordinate to $\widetilde s$, 
    $P(\sigma)^{\bot}$ is obtained
    from $P(\sigma)$
    by adjoining a loop or a coloop as a new last element, depending on whether $b = 0$ or $b = 1$.
    In particular, the slither word of $P(r_{\widetilde s}(a))^\bot$ is obtained by appending an $\elll$ (resp.\ a $\llle$) to $r_{\widetilde{s}}(a)$, provided $b = 0$ (resp.\ $b = 1$).
    $P(r_{\widetilde s}(a))^\top$ is obtained
    from $P(r_{\widetilde s}(a))^\bot$ 
    by switching the roles of last two elements, so the corresponding slither is obtained by adding the $\elll$ or $\llle$ immediately \emph{before} the final character of $r_{\widetilde{s}}(a)$.
    We need only check that this agrees with $r_s(a)$, since $\phi_s(B) = \phi_{\widetilde{s}}(\tilde B) = a$.
    Since $B \in \mathcal{B}_2$, $a$ ends in $1-b$ (with the convention that $\eps$ ends in $0$ and $\varnothing$ ends in $1$). Hence under the rightmost embedding of $1a$ into $1s$, no character is sent to the last character of $s = \widetilde{s}b$. This character therefore becomes an 
    $\elll$ or a $\llle$ according to whether $b = 0$ or $b = 1$, which agrees with the slither of $P(r_{\widetilde s}(a))^\top$.
\end{proof}
The combination of Lemmas \ref{lm:bijection1} and \ref{lm:bijection2} shows that 
the faces $P(s)^{v(s)}_p$ as $p$ ranges over the vertices in $P(s)$ are exactly the faces 
$P(r_s(t))$ as $t$ ranges over $\ebs(s)$. 
Since the corresponding \BB cells in $X_s$ give an affine paving, the classes of the corresponding torus-invariant subvarieties form a basis for $A_{\bl}(X_s)$. 
We state this as a corollary, using the convention of writing simply $[\sigma]$ for the (Chow) homology class corresponding to the face $P(\sigma)$ 
given by a slither $\sigma$ subordinate to $s$. Note that if $\sigma = \pr t \rp$ for a binary string $t$, then 
$[\sigma]$ lives in degree $\abs{t} + 1$ by \Cref{obs:dim_count}.
\begin{Cor}\label{cor:basis_of_A}
    Let $s$ be a snake. The classes $[r_s(t)]$ for $t \in \ebs(s)$ form a basis for $A_\bl(X_s)$.
    In particular, a basis for $A_k(X_s)$
    is given by $[r_s(t)]$ for extended binary substrings $t$ of $s$ of length $k-1$.
\end{Cor}
\subsection{Homological relations}\label{sec:homology_relations}
In this section we compute the Chow homology groups of snake den varieties, as well as the pushforward maps induced by slither morphisms, and explain their relationship to $\nsym$.
By \Cref{lm:cycle_class}, the same story holds when we instead consider singular homology (with integral coefficients) for these varieties over $\C$.
Whereas computing the Chow homology only requires us to work at the level of rational equivalence relations, 
we will prove that some of these relations are witnessed by algebraic homotopies. This will be useful in \Cref{sec:infinite}.

\begin{Df}
    Let $X$ and $Y$ be varieties.
Let $f, g\colon X \to Y$ be two morphisms.
A \textbf{naive $\A^1$-homotopy} from $f$ to $g$ is a morphism $H\colon \A^1 \times X \to Y$ whose restrictions to $\{0\} \times X$
and $\{1\} \times X$ coincide with $f$ and $g$, respectively. Two morphisms $f,g\colon X \to Y$ are 
\textbf{naively $\A^1$-homotopic} if there exists a sequence of maps $f = h_0, \dots, h_n = g$ 
and for each $i$ a naive $\A^1$-homotopy from $h_i$ to $h_{i+1}$.
\end{Df}

\begin{Lm}\label[Lm]{lm:a1_homologous}
    Let $f, g \colon X\to Y$ be a pair of morphisms between complete algebraic varieties. If $f$ and $g$ are naively $\A^1$-homotopic then the pushforward maps
    $f_*, g_*\colon A_\bl(X) \to A_{\bl}(Y)$ coincide.
\end{Lm}
\begin{proof}
We may assume that there exists a naive $\A^1$-homotopy 
$H\colon \A^1 \times X \to Y$
from $f$ to $g$. Write $\widetilde{H}\colon \A^1 \times X \to \A^1 \times Y$ for the morphism $(t, x) \mapsto (t, H(t, x))$.
Let $i_0\colon X \to \A^1 \times X$ and 
$i_1\colon X \to \A^1 \times X$ be the inclusions of $\{0\} \times X$
and $\{1\} \times X$, respectively.
Let $\pi_1^X\colon \A^1 \times X \to \A^1$ and 
$\pi_2^X\colon \A^1 \times X \to X$ be the projections onto the first and second factors, respectively. Similarly define 
$j_0, j_1\colon Y \to \A^1 \times Y$, 
$\pi_1^Y\colon \A^1 \times Y \to \A^1$,
and $\pi_2^Y\colon \A^1 \times Y \to Y$.
\[
\begin{tikzcd}
    \A^1 \times X\ar[rr, "\widetilde{H}"]\ar[rd, "\pi^X_1" below] 
    && \A^1 \times Y\ar[ld, "\pi_1^Y"]
    &&&&
    X\ar[rr, "f"]
    \ar[d, "i_0" right]
    && Y
    \ar[d, "j_0" left]
    \\
    &\A^1
    &&&&&
    \A^1 \times X \ar[u, "\pi_2^X" left, bend left]
    \ar[rr, "\widetilde{H}" below] && \A^1 \times Y
    \ar[u, "\pi_2^Y" right, bend right]
\end{tikzcd}
\]
Since $X$ is proper, so is $\pi_1^X$. Since $\pi_1^Y \circ \widetilde{H} = \pi_1^X$, the morphism $\widetilde{H}$ is proper by Vakil's cancellation lemma \cite{vakil}*{Theorem 11.1.1}.
Note also that $i_0$, $j_0$ are regular embeddings and that 
$\pi_2^X$ is flat.
The pullback maps $(\pi_2^X)^*$ and 
$i_0^*$ are inverse isomorphisms \cite{fult_int_theory}*{Definition 3.3};
the same is true of $(\pi_2^Y)^*$ and $j_0^*$.
Since the commutative square formed by $i_0, j_0, f, \widetilde{H}$ is Cartesian, we have
$f_* \circ i_0^* = j_0^* \circ \widetilde{H}_*$ by the push-pull formula
\cite{fult_int_theory}*{Theorem 6.2(a)}. Pre-composing with 
$(\pi_2^X)^*$ gives $f_* = j_0^* \circ \widetilde{H}_* \circ (\pi_2^X)^*$.
We similarly get $g_* = j_1^* \circ \widetilde{H}_* \circ (\pi_2^X)^*$.
Since $j_0^*$ and $j_1^*$ are both inverses to $(\pi_2^Y)^*$, we have 
$j_0^* = j_1^*$ and thus $f_* = g_*$.
\end{proof}

\begin{Lm}\label[Lm]{lm:chimneyhomotopy}
    Let $P$ be a polytope and $Q = \Chim(P, f)$ be a chimney over $P.$ The natural embeddings of $X_P$ into $X_Q$ given by identifying $P$ with $P^\bot$ or $P^\top$ are naively $\A^1$-homotopic. 
\end{Lm}
\begin{proof}
    After rescaling $P$ and $Q$ by a large integer, we may assume that both are very ample.
    Recall from \Cref{sec:toric}, that the projective embedding of $X_Q$ determined by the complete linear system $\abs{\mathcal{O}(Q)}$ is into the projective space $\Pb^{M_Q \cap Q - 1}$ with projective coordinates $(y_m)_{m \in M_Q \cap Q}$.
    For each lattice point $m \in M_Q \cap Q$, the character $\chi^{-m}$ is a regular section of the line bundle $\mathcal{O}(Q)$.
    The map to projective space is given by 
    $x \mapsto ({\chi}^{-m}(x))_{m \in M_Q \cap Q}$. The image of this morphism is the closed subvariety of $\Pb^{M_Q \cap Q - 1}$ determined by the equations
    \[
        y_{m_1} \cdots y_{m_n} = y_{m_1'} \cdots y_{m_n'}
    \]
    for all pairs 
    $(m_1, \dots, m_n), (m_1', \dots, m_n')$ of equally-sized tuples of points in $M_Q \cap Q$ sharing the same sum 
    $m_1 + \dots + m_n = m_1' + \dots + m_n'$ in $M_Q$
    \cite{cox2011toric}*{Proposition 2.1.4}.

    The embeddings of $P$ as the faces $P^{\bot}$ and $P^{\top}$ of $Q$ are given in these coordinates as:
    \[
        \iota^{\bot}: y_m \mapsto \begin{cases}
                {\chi}^{-m} & \text{if $m \in M \cap P^{\bot}$}\\
                0 & \text{if $m \in (M_Q \cap Q) \setminus (M \cap P^{\bot})$}
            \end{cases}
            \qquad\qquad
            \iota^{\top}: y_m \mapsto \begin{cases}
                {\chi}^{-m} & \text{if $m \in M \cap P^{\top}$}\\
                0 & \text{if $m \in (M_Q \cap Q)\setminus(M \cap P^{\top})$}
            \end{cases}
    \]
    We define a naive $\A^1$-homotopy 
    $\iota^{\side}\colon \A^1 \times X_P \to X_Q \subseteq \Pb^{M_Q \cap Q - 1}$ by
    $y_{(b,h)} \mapsto t^h(1-t)^{f(b) - h}{\chi}^{-b}$, where $t$ is the coordinate on $\A^1$ and $(b, h) \in M_P \times \Z$ is a point (written using the decomposition of the lattice $M_Q$ inherent in the chimney construction).
    It is easy to see that $\iota^{\side}$ restricts to $\iota^{\bot}$ and $\iota^{\top}$ when we set $t = 0$ or $t = 1$ respectively. We need only show that the image of $\iota$ lands in $X_Q$.
    Suppose $(b_1, h_1), \dots, (b_n, h_n)$
    and $(b_1', h_1'), \dots, (b_n', h_n')$ are lattice points in $M_Q$ such that
    \[
        (b_1, h_1) + \dots + (b_n, h_n) = (b_1', h_1') + \dots + (b_n', h_n')
    \]
    Since both $(b, h) \mapsto h$ and 
    $(b, h) \mapsto f(b) - h$ are affine-linear maps, it follows that 
    \begin{gather*}
        h_1 + \dots + h_n = h_1' + \dots + h_n' \\ \text{and} \\
        (f(b_1) - h_1) + \dots + (f(b_n) - h_n)
        = (f(b_1') - h_1') + \dots + (f(b_n') - h_n').
    \end{gather*}
    We also have 
    \[
        {\chi}^{-b_1} \cdots {\chi}^{-b_n} = {\chi}^{-b_1 -\dots - b_n} = {\chi}^{-b_1' - \dots - b_n'} = {\chi}^{-b_1'} \cdots {\chi}^{-b_n'}.
    \]
    From this we see that $\iota^{\side}$ indeed maps into $X_Q$.
\end{proof}
We will show that the morphisms of toric varieties induced by slithers between two given snakes are all naively $\A^1$-homotopic.
Not knowing this for now, we write 
$a \sim b \rel c$ for a snake den $c$ and slither words $a, b$ subordinate to it to mean that $\ove a = \ove b = s$ 
(i.e.\ $a$ and $b$ have the same domain)
and the embeddings of $X_s$ into $X_c$ induced by $a$ and $b$ are naively $\A^1$-homotopic. 
\begin{Lm}\label[Lm]{lm:simprop}
The relation $\sim$ has the following properties.
    \begin{enumerate}
        \item $\sim$ is stable under products. That is, $a\sim b\rel c$ implies $ax\sim bx \rel cy$ and $xa\sim xb\rel yc$ for 
        any snake den $y$ and
        any slither $x$ subordinate to it.
        \item $\sim$ is stable under post-composition. That is, if $a\sim b\rel c$ and $f\colon c\to d$ is a slither morphism then $f\circ a\sim f\circ b\rel d.$  
        \item For any binary strings $a, b$ we have
        \begin{enumerate}
            \item $\llle\elll\sim \elll \llle\rel\pr\rp.$
            \item $\pr a0\elll b\rp \sim \pr a\elll0b\rp\rel\pr a00b\rp$ and similarly $\pr a1\llle b\rp \sim \pr a\llle1b\rp\rel\pr a11b\rp$.
            \item $\pr \elll a\rp \sim\elll\pr a\rp\rel\pr 0 a\rp$ and $\pr \llle a\rp \sim\llle\pr a\rp \rel \pr 1 a\rp$.
            Also, $\pr  a\elll\rp \sim\pr a\rp\elll\rel\pr a0\rp$ and $\pr a\llle \rp \sim\pr a\rp\llle\rel\pr a1 \rp$. 
        \end{enumerate}
    \end{enumerate}
\begin{proof}
    (1) follows from the fact that if $f, g\colon X \to Y$ are naively $\A^1$-homotopic 
    and $h\colon Z \to W$ is any morphism then
    $f \times h, g \times h\colon X \times Z \to Y \times W$ are naively $\A^1$-homotopic.
    (2) follows from the fact that if 
    $f, g\colon X \to Y$ are naively $\A^1$-homotopic and $h\colon Y \to Z$ is any morphism then 
    $h \circ f, h\circ g\colon X \to Z$ are naively
    $\A^1$-homotopic.
    For (3), in each case one observes using \Cref{lm:is_chimney} and \Cref{obs:up_and_down}  that the two maps correspond to the inclusions of a polytope as top and bottom faces of a chimney over it and then applies \Cref{lm:chimneyhomotopy}. This is clear in (a) and in the second half of (c); the first half follows by dualizing. For the first relation in (b), one observes that the indices $i$ and $i+1$ of the two $0$s correspond to parallel elements of the matroid and that the $i$- and $(i+1)$-terminal faces of its base polytope are given by 
    $\pr a\elll 0b\rp$ and $\pr a0 \elll b\rp$, respectively; one then appeals to \Cref{obs:up_and_down}. The second relation follows similarly.
\end{proof}
\end{Lm}
We now translate the $\sim$ relations in  \Cref{lm:simprop} into allowable ``moves'' for slither diagrams inside an ambient snake diagram which preserve the homotopy type of the slither morphism. 
$(3a)$ allows us to pass back-and-forth between a vine passing over a square to one passing under it. By applying $(1)$ and $(2)$, one can make this move as part of a bigger diagram as long as the square is in the ambient snake. Similarly, $(3c)$, combined with $(1)$ and $(2)$, allows us to make the remaining changes depicted in \Cref{fig:moves}, as long as the grayed-out squares lie in the ambient snakes. The move allowed by $(3b)$ (combined with $(1)$ and $(2)$) operates more transparently at the level of slither words; we defer its explanation to \Cref{lm:allslithersim}.
\begin{figure}[ht]
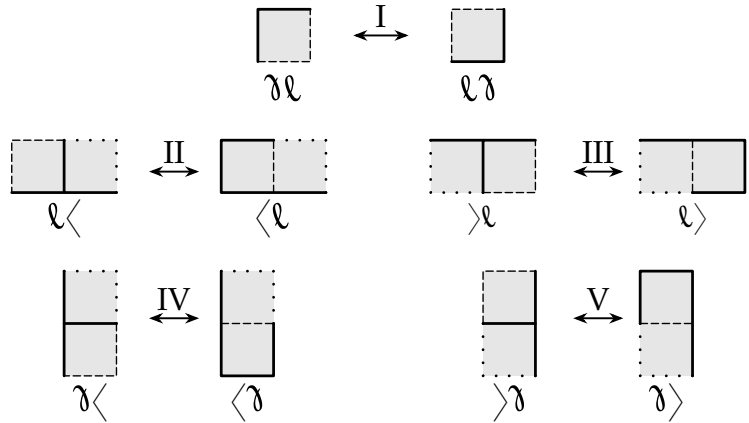

    \centering
    {\includestandalone[width=.6\textwidth]{figures/tikz_equivalences}}
    \qquad
    \caption{Allowable moves on slithers. A grayed-out square lies in the ambient snake. A thick line is an edge of the slither. A dashed line is an edge of the ambient snake which is absent from the slither. A dotted line is an edge of the ambient snake that may or may not be in the slither.}
    \label{fig:moves}
\end{figure}
\begin{Df}
    Let $a$ be a snake den and $c$ a snake. 
    For slither words $f, g\colon a\to c$, each unbound character of $f$ or $g$ corresponds to a character
    of $a$. We write $f\le g$ if 
    the set of unbound positions of $f$ (weakly) precedes the set of unbound positions of $g$ in the lexicographical order.
     Write $f < g$ to mean $f\le g$ and $g\not\le f$.
    More explicitly, $f < g$ means that
    the embeddings of $a$ into $c$ induced by $f$ and $g$ are distinct and that
    for the first character of $a$ at which these embeddings differ the position assigned by $f$ precedes that assigned by $g$.
    A \textbf{leftmost} slither $f$ is a \emph{minimum} in this preorder; that is, $f\le g$ for all $g\colon a\to c$.
\end{Df}

For convenience, we define $\widecheck{0} \deq \elll$
and $\widecheck{1} \deq \llle$.
\begin{Lm}\label[Lm]{lm:slitherprop}
    For any slither $f\colon a\to c$ where $c$ is a snake and any index $j$, 
    \begin{enumerate}
        \item If $f_{j} \in \{0, 1\}$ then $c_{j} = f_{j}$.
        \item If $f_j\in \{\elll, \llle\}$ 
        is enclosed between a matching pair $\pr \cdots \rp$, then $f_j = \widecheck{c}_j$.
        \item Suppose $f_i = \rp$ for some $i < j$
        and
        $f_j = \pr$.
        Then there exists some $k$
        with $i < k \leq j$ such that 
        $c_k \neq c_i$.
    \end{enumerate} 
    \begin{proof}
        All claims are clear for the trivial slither $c\colon c\to c$ and are collectively preserved by the substitution rules in \Cref{fig:local_moves}.
    \end{proof}
\end{Lm}
\begin{Lm}\label[Lm]{lm:allslithersim}
    Let $c$ be a snake and $a$ a snake den.
    The following hold:
    \begin{itemize}
        \item For any slithers $f, g\colon a \to c$, we have $f \sim g \rel c$.
        \item Suppose $f\colon a \to c$ is a leftmost slither. Let $j < k$ be indices such that $f_j = \rp$, $f_k = \pr$ and $f_h \in \{\elll, \llle\}$ whenever $j < h < k$.
        Then $c_j \neq c_k$ and the slither word 
        $\widetilde{f}$ obtained by replacing $f_j$ with $c_j$ and $f_k$ with $c_k$ is subordinate to $c$.
    \end{itemize}
    \begin{figure}[hbt!]
    \centering
    {\includestandalone[width=.8\textwidth]{figures/slither_moves_12}}
    \qquad
    \caption{Moving from $\llle \elll \elll \elll \llle \pr 1 \rp$ to $\pr \elll \elll 1 \rp \elll \llle \llle$ via allowable slither moves.}
    \label{fig:slither_move_ex}
\end{figure}
    \begin{proof}
        Let $g\colon a \to c$ be any slither. We start by showing that $g \sim f \rel c$ for some leftmost slither $f$;
        we will later show that all leftmost slithers are $\sim$-equivalent.
        This is achieved by moving from $g$ to $f$ using allowable moves. 
        An example is given in \Cref{fig:slither_move_ex}. The reader is encouraged to examine the figure and guess the algorithm before reading the formal proof.
    
        Let $\pi(1)< \dots< \pi(t)$ be the positions of the unbound characters in $g$. We say that $\pi(i)$ is \textbf{movable} if one of the following holds. 
        \begin{enumerate}[(i)]
            \item $i = 1$ and $\pi(i) > 1$.
            \item $i> 1$, $g_{\pi(i)} = \pr$ and
            $c_j \neq c_{\pi(i-1)}$ for some $\pi(i-1) < j < \pi(i)$.
            \item $g_{\pi(i)}\in\{0, 1\}$ and 
            $c_j = g_{\pi(i)}$ for some $\pi(i-1) < j < \pi(i)$.
            \item $g_{\pi(i)} = \rp$ and $\pi(i) > \pi(i-1)+1$.
        \end{enumerate}
        Suppose first that $g$ has a movable character. We will then show $g \sim g'$ for some slither $g'\colon a \to c$ with $g' < g$.
        Let $i$ be minimal such that $\pi(i)$ is movable. 

        We start with cases (i) and (ii), for which $g_{\pi(i)} = \pr$.
        The hypotheses force $g_{\pi(i)-1} \in \{\elll, \llle\}$. We deal with the case $g_{\pi(i) - 1} = \elll$; the other case being similar. If $c_{\pi(i)} = 0$, then we use move II in \Cref{fig:moves} to shift the $\pr$ to the left. Otherwise, we are in the situation on the left-hand side of \Cref{fig:stuck_moves}. The sequence of vines starting with the horizontal one depicted in the figure and continuing westward either eventually takes a southward turn or hits a scale of the slither; these two cases are depicted in \Cref{fig:stuck_cases}. In the first case, we apply move I repeatedly to shift the vines to the position depicted in red in \Cref{fig:stuck_cases}. We can then apply move IV.
        The resulting slither $g'$ satisfies 
        $g' \sim g \rel c$ and
        $g' < g$.
        The second case forces $c_j = c_{\pi(i-1)} = 0$ for all $\pi(i-1) < j  <\pi(i)$, contradicting the movability assumption. 
    \begin{figure}[htbp!]
    \centering

    \begin{minipage}[t]{0.48\textwidth}
        \centering
        \includestandalone[height=40pt]{figures/stuck_moves}
        \caption{Situations where the moves in \Cref{fig:moves} are not immediately available.}
        \label{fig:stuck_moves}
    \end{minipage}
    \begin{minipage}[t]{0.48\textwidth}
        \centering
        \includestandalone[height=40pt]{figures/stuck_cases}
        \caption{Westward behavior of the first slither in \Cref{fig:stuck_moves}.}
        \label{fig:stuck_cases}
    \end{minipage}
\end{figure}

We deal next with case (iii).
Consider any index $j$ with $\pi(i-1) < j < \pi(i)$ for which $c_j = g_{\pi(i)}$.
For simplicity, we assume $c_j = 0$; the case $c_j = 1$ is similar.
Then $g_j = \widecheck{c}_j = \elll$ by \Cref{lm:slitherprop}(2). 
Let $h$ be the slither word obtained from $g$ by {replacing} $g_j = \elll$ with $0$.
Let $g'$ be the slither word obtained from $h$ 
by replacing $h_{\pi(i)} = g_{\pi(i)} = 0$ with $\elll$.
Then $h$ is subordinate to $c$ since its slither diagram is obtained from that of $g$ by simply adding the interior edge lying on the $j$-th antidiagonal. Each of $g$ and $g'$ factors through $h$; explicitly, $g = h\circ k$ and $g' = h\circ k'$ where $k, k'\colon a\to \ove h$ are obtained by changing 
one of the two consecutive $0$'s in $\ove{h}$ corresponding to the $j$-th and $\pi(i)$-th characters of $h$ to an $\elll$.
We have $k\sim k'\rel \ove h$ by parts (3b) and (1) of \Cref{lm:simprop}. Thus $g\sim g'\rel c$ by  \Cref{lm:simprop}(2).

In case (iv), $g_{\pi(i)} = \rp$ and we can always apply one of the moves III or V in \Cref{fig:moves} to shift this $\rp$ to the left.

Suppose $g$ has no movable characters. We claim that $g$ is already leftmost. 
Let $h\colon a \to c$ be any slither and let $\rho(1) < \dots < \rho(t)$ be its unbound positions. We show $g\le h$. 
Let $i$ be minimal such that $\pi(i) \neq \rho(i)$.
If $i = 1$, then $\pi(i) = \pi(1) = 1 \leq \rho(i)$ since $g$ has no movable character of type (i).
If $i > 1$ and $a_i = \pr$, then \Cref{lm:slitherprop}(3) gives an index $k$ such that $c_k \neq c_{\rho(i-1)} = c_{\pi(i-1)}$ and $\pi(i-1) < k \leq \rho(i)$, whereas $g$ having no movable character of type (ii) forces $k \geq \pi(i)$; so again $\pi(i) \leq \rho(i)$.
If $a_i \in \{0, 1\}$, then since 
$g$ has no movable characters of type (iii)
and $c_{\rho(i)} = h_{\rho(i)} = a_i = g_{\pi(i)}$, we must have $\pi(i) \le \rho(i).$ 
Finally, if $a_i = \rp$, then 
$\pi(i) = \pi(i-1) + 1 = \rho(i-1) + 1 \leq \rho(i)$. 

To prove the first claim, it remains only to show that $f \sim f' \rel c$ for any two \emph{leftmost} slithers $f, f'\colon a \to c$.
We first prove that $f$ and $f'$ agree up to and including their final $\rp$. As a byproduct, we will also extract a proof of the second part of the lemma.
The positions of the unbound characters  in $f$ and $f'$ certainly agree. The same is true of bound characters between matching $\pr \cdots \rp$ by \Cref{lm:slitherprop}(2).
It remains to handle the characters between adjacent $\rp \cdots \pr$. 
Let $j < k$ be the indices in $f$ (hence also $f'$) of these $\rp$ 
and $\pr$.
Without loss of generality, we may assume that $f$ and $f'$ agree up to the $j$-th index.
Since $f$ has no movable character of type (ii),
we have $c_h = c_j$ whenever $j \leq h < k$.
By \Cref{lm:slitherprop}(3) we also have 
$c_j\ne c_k$.
For simplicity, we assume $c_j = 0$ and $c_k = 1$, the other case being similar.
The slither diagrams of $f$ and $f'$ up to the $j$-th antidiagonal agree and the local picture between this point and the $k$-th antidiagonal must be the one depicted on the right-hand side of \Cref{fig:stuck_cases}. 
All vines in this region are forced to be horizontal; that is, $f_{j+1} = \dots = f_{k-1} = \elll = f_{j+1}' = \dots = f_{k-1}'$.
Looking again at the right-hand side of \Cref{fig:stuck_cases}, we observe the effect on the diagram of $f$ resulting from changing $f_j$ to $0$ and $f_k$ to $1$ is simply to 
add as a new scale the rectangle at the bottom right of the picture. This proves the second claim in the lemma statement.

We have now shown that $f$ and $f'$ differ only in the characters that follow their final $\rp$.
Thus their diagrams coincide except possibly for a sequence of vines in their northeast.
One can therefore transform one into the other
using move I in \Cref{fig:moves}. This completes the proof.
\end{proof}
\end{Lm}
\begin{Rmk}
\Cref{lm:allslithersim} shows that there is always a lexicographically first slither $a \to c$, provided such slithers exist at all. In fact, one can show that lexicographically first slithers are in fact minima in the \emph{Gale} (pre)order on slithers. This justifies the name \emph{leftmost}.
\end{Rmk}
\begin{Cor}\label[Cor]{cor:homotopic}
    For any snake den $a$, any snake $c$, and any two slither morphisms $f, g\colon a\to c$, the induced embeddings $f\colon X_a\to X_c$ and $g\colon X_a\to X_c$ are naively $\A^1$-homotopic.
\end{Cor}
\begin{Prop}\label[Prop]{prop:homology_01_split}
    Let $a$ and $b$ be binary strings. In $A_{\bl}(X_{ a01b})$, we have $[\pr a\rp\pr b\rp] = [\pr a0\llle b\rp] + [\pr a\elll1 b\rp]$.
    Similarly, in $A_{\bl}(X_{a10b})$, we have $[\pr a\rp\pr b\rp] = [\pr a1\elll b\rp] + [\pr a\llle0 b\rp].$
    \begin{proof}
        We prove the first claim; the second is proved similarly. 
        Consider the character $\chi = e_i - e_{i+1}$, where $i$ and $i+1$ are the positions in $\pr a 01 b \rp$ of the adjacent characters $01$. By \Cref{lm:face_structure}, the only facets 
        of $P(a01b)$
        whose normals are not perpendicular to $\chi$ are $P(\pr a\rp\pr b\rp)$, $P(\pr a0\llle b\rp)$ and $P(\pr a\elll1 b\rp)$.
        These pair with $\chi$ to give $1$, $-1$, and $-1$, respectively. 
        By now considering $\chi$ as a rational function on $X_{a01b}$, we get $0 = \Div(\chi) = [\pr a\rp\pr b\rp] - [\pr a0\llle b\rp] - [\pr a\elll1 b\rp]$ in $A_{\bl}(X_{a01b})$.
    \end{proof}
\end{Prop}
\noindent
Recall the Hopf algebra $\nsym$ discussed in \Cref{ssec:hopf_stuff}.
For a binary string $s$, we will write $\nsym_s$ for the $\Z$-submodule of $\nsym$ spanned by $\{R_t\mid t\in \ebs(s)\}.$
\begin{Thm}\label{thm:homology_relation}
     Let $s$ be a binary string. There is an isomorphism of $\Z$-modules
     $\psi_s\colon A_\bl(X_s) \to \nsym_s$ mapping $[V(F)]\mapsto R_{\ove F}$ for each face $F$ of $P(s)$, identifying 
     the snake den $\ove{F}$ with its ribbon diagram (i.e.\ its snake den diagram).
     Under this isomorphism, the effective classes in $A_{\bl}(X_s)$ correspond exactly to the  elements in $\nsym_s$ which are nonnegative 
     $\Z$-linear combinations of noncommutative ribbon functions.
     \begin{proof}
    We define $\psi_s$ by sending $[r_s(t)] \mapsto R_t$ for each $t \in \ebs(s)$, which is well-defined because these classes form a basis of $A_{\bl}(X_s)$ by \Cref{cor:basis_of_A}. We now show that $\psi_s([V(F)]) = R_{\ove{F}}$ by induction on the number of components of the ribbon diagram of $\ove F$.
    Let $\sigma$ be the slither word of $F$.
    If the number of components is $1$ or $0$, 
    then $\ove{\sigma} = \ove{r_s(t)}$ for some $t \in \ebs(s)$, so the claim follows from
    \Cref{cor:homotopic} and \Cref{lm:a1_homologous}.
    Now suppose there are at least two components.
    Using the same two results, we may assume without loss of generality that $\sigma\colon \ove{\sigma} \to s$ is leftmost.
    Write
    $\ove \sigma = \pr w_1 \rp \pr w_2 \rp z$
    for some binary strings $w_1, w_2$ and snake den word $z$. 
    Then $\sigma = \star_1 \pr W_1 \rp \star_2 \pr W_2 \rp Z$ where $\star_1, \star_2$ are sequences of bound characters and 
    $\pr W_1 \rp, \pr W_2 \rp, Z$ prune to 
    $\pr w_1 \rp,  \pr w_2 \rp, z$, respectively.

    By the second half of \Cref{lm:allslithersim}, 
    either $\star_1 \pr W_1 0 \star_2 1 W_2 \rp Z$ or 
    $ \star_1 \pr W_1 1 \star_2 0 W_2 \rp Z$ is a slither subordinate to $s$.
    By symmetry, we may assume we are in the first case.
    We have $[\pr w_1 \rp \pr w_2 \rp z] = [\pr w_1 \elll 1w_2 \rp z] + [\pr w_1 0\llle w_2 \rp z]$ in $A_{\bl}(X_{\pr w_1  0 1 w_2 \rp z})$ by \Cref{prop:homology_01_split}.
    Taking the pushforward to $A_{\bl}(X_s)$, we get $[\star_1 \pr W_1 \rp \star_2 \pr W_2 \rp Z] = [\star_1 \pr W_1 0 \star_2 \llle W_2 \rp Z]+[\star_1 \pr W_1 \elll \star_2 1 W_2 \rp Z] $ in $A_\bl(X_s)$. 
    Applying the inductive hypothesis, we see that $ \psi_s([\star_1 \pr W_1 \rp \star_2 \pr W_2 \rp Z])$ is equal to
    \[
    R_{\pr w_11w_2 \rp z}+R_{\pr w_10w_2 \rp z}
    = 
    (R_{w_11w_2}+R_{w_10w_2})R_z
    = R_{w_1} R_{w_2} R_z = R_{\pr w_1\rp\pr w_2 \rp z}.
    \]
    Since the effective cycles in $A_\bl(X_s)$ are nonnegative combinations of classes of torus-invariant subvarieties \cite{int_theory_spherical}*{Corollary of Theorem 1}, the last claim follows from the description of $\psi_s$ on the classes $[V(F)]$.
     \end{proof}
\end{Thm}
\begin{Thm}\label{thm:homology_morphisms}
   There is a unique system $\{\psi_s\colon A_\bl(X_s) \to \nsym \}$ of homomorphisms of graded abelian groups indexed by snake dens $s$ satisfying:
   \begin{itemize}
       \item Functoriality: for any morphism $j\colon X_{s'} \to X_{s}$ from a snake den $s'$, we have $\psi_{s} \circ j_* = \psi_{s'}$.
       \item Product: if $s = tt'$ then $\psi_{s}([X_{s}]) = \psi_t([X_t])\psi_{t'}([X_{t'}])$.
       \item Normalization: $\psi_s([X_s]) = R_s$ if $s$ is a snake and $\psi_{\eps}([X_{\eps}]) = 1$.
   \end{itemize}
   For snake $s$, it coincides with the embedding $A_{\bl}(X_s) \to \nsym_s \to \nsym$ as defined in \Cref{thm:homology_relation}.
\end{Thm}
   \begin{proof}
       For any snake den, the product of the affine pavings of its constituent snakes is an affine paving, so gives a basis for its Chow homology. Writing $s = \pr s_1\rp\dots \pr s_k\rp$, the three conditions above force us to define $\psi_s$ by mapping $[P(r_{s_1}(t_1))\times \dots \times P(r_{s_k}(t_k))]\mapsto R_{t_1}\cdots R_{t_k}$
       for $t_i \in \ebs{(s_i)}$.
       Conversely, defining $\psi_s$ in this manner, it follows from \Cref{thm:homology_relation} that the resulting system of homomorphisms has the desired properties.
   \end{proof}

\section{Cohomology}\label{sec:cohomology}

Having gotten a hold on the Chow homology of snake den varieties, we now investigate their
Chow cohomology. In the case of snakes, we will identify it with a suitable quotient of the ring of quasisymmetric functions.
It follows from \Cref{lm:cycle_class} that all of what we say will carry over without change to the singular cohomology with integral coefficients of these varieties as analytic spaces over $\C$.

Let $X$ be a complete toric variety associated to a fan $\Sigma$. \cite{fulton94} gives the relation
\[
    \delta_*[X] = \sum_{\sigma, \tau} m_{\sigma, \tau}[V(\sigma)]\otimes [V(\tau)] \quad \text{in} \quad A_{\bl}(X) \otimes  A_{\bl}(X).
\]
Here, 
$\delta\colon X \to X \times X$ is the diagonal map,
and the sum is over all complementary-dimensional cones $\sigma, \tau$ in $\Sigma$. The coefficient $m_{\sigma, \tau}$ depends on the choice of a (sufficiently general) covector $v$, and is equal to the lattice index $[N: N_{\sigma} + N_{\tau}]$ whenever $\sigma \cap (\tau + v) \neq \emptyset$ and $0$ otherwise. 
The following proposition presents a case in which this expression becomes particularly simple.
Note that for any complete toric variety, the K\"unneth map in Chow homology is an isomorphism by \cite{int_theory_spherical}*{Theorem 2}.

\begin{Prop}\label[Prop]{prop:diagonal}
    Assume $P$ is an edge-unimodular lattice polytope stratified by a covector $v$ and by its opposite $-v$. Then letting $\delta\colon X_P \to X_P \times X_P$ be the diagonal embedding, we have
    \[
        \delta_* [V(P)] = \sum_{p} [V(P_p^{-v})] \otimes [V(P_p^{v})] \quad \text{in} \quad A_{\bl}(X_P \times X_P) \cong A_{\bl}(X_P) \otimes A_\bl(X_P),
    \]
    where the sum is over vertices $p$ of $P$.  
\end{Prop}
\begin{proof}
    Let $F$ and $G$ be any two complementary-dimensional faces of $P$. It suffices to show that
    in the relation above for the given covector $v$, we have
    $m_{\sigma_F, \sigma_G} = 0$ unless $F = P_p^{-v}, G = P_p^v$ for some vertex $p$ in $P$, in which case $m_{\sigma_F, \sigma_G} = 1$.
    Suppose $m_{\sigma_F, \sigma_G} \neq 0$.
    Then $\sigma_F \cap (\sigma_G + v) \neq \varnothing$. Hence there is some $u \in N_{\R}$ with $u \in \sigma_F$ and $u-v \in \sigma_G$. 
    As $t$ varies from $0$ to $1$, considering the face on which $u - tv$ is maximized gives a sequence $F_1, F_2, \dots, F_n$ of faces of $P$ with $n \geq 2$. The sequence alternates between vertices and higher-dimensional faces. If $F_i$ is a vertex then 
    $F_{i-1}^{-v} = F_i = F_{i+1}^{v}$.
    Write $p \middarrow q$ for vertices $p, q$ in $P$ to mean that there is path of edges (possibly of length $0$) joining $p$ to $q$ along which $v$ increases. We have
    \[
        F_n^{-v} \middarrow F_n^{v} = F_{n-1}^{-v}\middarrow \cdots  \middarrow F_2^{v} = F_1^{-v} \middarrow F_1^{v}.
    \]
    Whenever we have an edge from $p$ to $q$ with $v(p)< v(q)$, we have $p\in  P_q^v$. Since $v$ stratifies $P$, this forces $P_p^v \subseteq P_q^v$. By transitivity, the same is true if we only assume $p \middarrow q$.
    Using this on the $\middarrow$ sequence above,
    we deduce $P_{F_n^{v}}^v \subseteq P_{F_1^{-v}}^v$. Again using that $v$ stratifies $P$, we have $\dim P_p^v + \dim P_p^{-v} = \dim P$ by \cite{strat_poly}*{Lemma 2.14}. The choice of $u$ implies $F \subseteq F_1 \subseteq P_{F_1^{-v}}^{-v}$ and $G \subseteq F_n \subseteq P_{F_n^{v}}^v$. So
    \begin{equation*}
        \dim P =  \dim F + \dim G \leq \dim P_{F_1^{-v}}^{-v} + \dim P_{F_n^{v}}^v  = \dim P - \dim P_{F_1^{-v}}^v +  \dim P_{F_n^{v}}^v.
        \tag{$\star$}\label{eq:dim}
    \end{equation*}
    Thus $\dim P_{F_1^{-v}}^v \leq \dim P_{F_n^{v}}^v$, so that $P_{F_1^{-v}}^v = P_{F_n^{v}}^v$ by the containment shown earlier.
    This is only possible if $F_1^{-v} = F_n^{v} = p$. Thus, the inequality in \eqref{eq:dim} must be an equality. We conclude that $F = P_p^{-v}$ and $G = P_p^v$.

    Conversely, assume $F = P_p^{-v}, G= P_p^{v}$ for some vertex $p$ in $P$. Again by \cite{strat_poly}*{Lemma 2.14}, $\dim F + \dim G = \dim P$. 
    The edge-unimodularity assumption gives
    $M = M_{F} \oplus  M_{G}$
    where $M_F \deq M \cap \Lin F$
    and $M_G \deq M \cap \Lin G$,
    where $\Lin F$ (resp.\ $\Lin G$) is the linear span of $F - F$ (resp.\ $G - G$). 
    Taking duals gives $N = M^{\vee} = M_F^{\vee} \oplus M_G^{\vee}$. Then $M_F^{\vee}$ is identified with $M_G^{\perp}$ and $M_G^{\vee}$ is identified with $M_{F}^{\perp}$, so $N = M_F^{\perp} \oplus M_G^{\perp}$.
    Write $v = u + w$ for $u \in M_F^{\perp}$
    and $w \in M_G^{\perp}$.
    Then, $u$ is maximized on $F$ and $u-v$ is maximized on $G$. So $u \in \sigma_F$ and $u - v \in \sigma_G$. Since $N_{\sigma_F} = M_F^{\perp}$
    and $N_{\sigma_G} = M_G^{\perp}$, we also have 
    $m_{\sigma_F, \sigma_G} = [N : N_{\sigma_F} + N_{\sigma_G}] = [N : M_F^{\perp} + M_G^{\perp}] = 1$.
\end{proof}

We now apply \Cref{prop:diagonal} to the case of snake den matroid base polytopes. Our stratifying covector will be any \textbf{increasing} covector $v$, i.e.\ with coordinates $v_1 < v_2  < \dots$. We start by identifying the faces $P(s)_p^{\pm v}$ for this choice of $v$.

\begin{Df}
    For any snake den $s$ and any basis $B$ of $s$ let $s_B^+$ (resp.\ $s_B^-$) be the slither whose bases are the bases of $s$ which are bounded above (resp. below) by $B$ in the Gale order. 
\end{Df}

We observe that since $s_B^+$ and $s_B^-$ are slithers subordinate to $s$, they correspond to faces of $P(s)$.

\begin{Lm}\label[Lm]{lm:stratifying}
    Let $s$ be a snake den. Let $v$ be an increasing covector. Then both $v$ and $-v$ stratify $P(s)$. More precisely, for any basis $B$ of $s$, we have
     $P(s)_{e_B}^v= P({s_B^+})$ and $P(s)_{e_B}^{-v}= P({s_B^-})$. 
\end{Lm}
\begin{proof}
    We start by showing $P(s)_{e_B}^v = P({s_B^+})$. It will suffice to show
    that among the vertices $q$ of $P(s)$ adjacent to $p = e_B$, those satisfying
    $v(q) \leq v(p)$ are precisely the ones lying in the face $P(s_B^+)$. For this in turn, it suffices to show that if $p = e_B, q = e_{B'}$ are adjacent vertices in $P(s)$ then $v(p) < v(q)$ if and only if in the Gale order we have $B \leq_G B'$.
    By \Cref{thm:edge_dirs} we have that $B' = B - b + b'$ for some ground set elements $b \neq b'$.
    Since $v$ is an increasing vector, we then have $B \leq_G B'$ if and only if $b < b'$. 
    This is equivalent to $v(p) = v(e_B) \leq v(e_{B'}) = v(q)$.
    The stratification condition is now reduced to the claim that $p \in P(s)_q^v$ defines a partial order on vertices $p, q$ in $P$. This is in fact so, since it is given precisely by the Gale order on bases.
    The analogous statements for $-v$ are handled similarly.
\end{proof}

The following is immediate from \Cref{prop:diagonal} and \Cref{lm:stratifying} since matroid base polytopes are edge-unimodular. 

\begin{Prop}\label{prop:coproduct_formula}
    Let $s$ be a snake den and $\delta\colon X_{s} \to X_s\times X_s$ be the diagonal embedding. Then, 
    \[
        \delta_*[X_s] = \sum_{B \in \mathcal{B}(M(s))} [{s_{B}^-}] \otimes [{s_{B}^+}].
    \]
\end{Prop}
\noindent
This agrees nicely with $\nsym$. 
Recognizing that $s_B^+$ and $s_B^-$ are ribbons, we can write
\[
    \Delta R_s = \sum_{B \in \mathcal{B}(M(s))} R_{s_B^-} \otimes R_{s_B^+}.\tag{$\blacklozenge$}
    \label{eq:coprod_R}
\]
This is the identity for the coproduct in $\nsym$ in  \Cref{prop:coproduct_nsym}.
\begin{Cor}\label[Cor]{cor:diagonal}
    Let $s$ be any snake den and let $\delta\colon X_s \to X_s \times X_s$. Identifying $A_{\bl}(X_s \times X_s)$ with $A_{\bl}(X_s) \otimes A_{\bl}(X_s)$ via the K\"unneth isomorphism, there is a commutative diagram:
    \[
        \begin{tikzcd}
            A_{\bl}(X_s) \arrow[r, "\psi_s"] \arrow[d, "\delta_*"] 
              & \nsym \arrow[d, "\Delta"] \\
            A_{\bl}(X_s)\otimes A_{\bl}(X_s) \arrow[r, "\psi_s \otimes \psi_s"] 
              & \nsym \otimes \nsym. 
        \end{tikzcd}
    \]
    Recall from \Cref{thm:homology_relation} that, when $s$ is a snake, $\psi_s$ maps $A_{\bl}(X_s)$ isomorphically onto $\nsym_s$.
\end{Cor}
\begin{proof}
    By the functoriality of $\psi_s$, it is enough to prove that the diagram commutes in degree $\dim X_s$.
    Given \Cref{thm:homology_morphisms}  and \eqref{eq:coprod_R}, this is precisely the content of \Cref{prop:coproduct_formula}.
\end{proof}

\begin{Df}
    Let $s$ be a snake.
    As we have seen in \Cref{cor:basis_of_A}, the classes $[r_s(t)]$ as $t$ ranges over the extended binary strings in $\ebs(s)$
    of length $k-1$ form a $\Z$-basis of 
    $A_k(X_s)$. We write $c_t$ for the elements of the dual basis of $A^k(X_s)$. (The notation suppresses the dependence on $s$.)
\end{Df}
We can now state the dual form of \Cref{prop:coproduct_formula}, in Chow cohomology.
Recall from \Cref{ssec:hopf_stuff} the notation $F_t$ for the fundamental quasisymmetric function indexed by a binary string $t$.
\begin{Cor}\label[Cor]{cor:struct_cont}
    Let $s$ be a snake and let $u, v,t \in \ebs(s)$.  Write $\lambda_{u, v}^t$ for the (unique) structure constants satisfying
    \[
        c_u \cup c_v = \sum_{t\in \ebs(s)} \lambda_{u, v}^t c_t. 
    \]
    Then, $\lambda_{u, v}^t = [F_t](F_u F_v)$,
    where the symbol $[F_t]$ means we are extracting the coefficient of $F_t$ in the subsequent expression.
\end{Cor}
\begin{proof}
    Write $\inner{-,-}$ for the canonical pairings between Chow cohomology and homology and between $\qsym$ and $\nsym$. Since $\psi_s([r_s(t)]) = R_t$, the following is formal from \cite{int_theory_spherical}*{Theorem 2}, \Cref{lm:kunneth_chow}, and \Cref{cor:diagonal}. 
    \[
        \lambda_{u,v}^t = 
        \inner{c_u \cup c_v , [r_s(t)]}
        = \inner{c_u \otimes c_v,  \delta_* [r_s(t)]}
        = \inner{F_u \otimes F_v, \Delta R_t}
        = \langle F_u F_v, R_t \rangle = [F_t](F_u F_v). 
    \qedhere
    \]
\end{proof}

\begin{Thm}\label{thm:cohomology_functorial}
    Let $s$ be a snake den. Then, the $\Z$-module homomorphism $ \psi_s^{\vee}\colon \qsym \to A^\bl(X_s)$ dual to $\psi_s$ is a graded ring homomorphism. 
    Suppose further that $s$ is a snake.
    Then $\psi_s^{\vee}$ is surjective; it maps
     $F_t$ to $c_t$ if $t\in \ebs(s)$ and to $0$ otherwise.
     Its kernel is then the $\Z$-submodule $\Span_{\Z} \{F_t \mid t \notin \ebs(s) \}$. Thus for any snake $s$ we get a graded ring isomorphism $A^\bl(X_s) \cong \qsym_s \deq \qsym/\ide{F_t \mid t \notin \ebs(s)}.$
\end{Thm}
\begin{proof}
    The first claim for snake dens follows from the case of snakes
    and the functoriality in \Cref{thm:homology_morphisms}
    since for any snake $s'$ into which $s$ admits a morphism, $\psi_s^{\vee}$ factors as $\qsym \xrightarrow{\psi_{s'}^{\vee}}A^{\bl}(X_{s'})\to A^{\bl}(X_s)$, where the second map is pullback in Chow cohomology.
    We may therefore assume that $s$ is a snake for the rest of the proof.

    The rule for where $F_t$ maps is immediate upon dualizing 
    \Cref{thm:homology_morphisms}.
    The surjectivity and kernel are then clear from this description.
    It remains only to show that $\psi_s^{\vee}$ is a ring homomorphism.
    Let $F_u, F_v \in \qsym$. Write $F_u F_v = \sum_{t} [F_t](F_u F_v) F_t$. Then, $\psi_s^\vee(F_uF_v)= \sum_{t \in\ebs(s)} [F_t](F_u F_v) c_t$. By \Cref{cor:struct_cont}, we have $\psi_s^\vee(F_u)  \cup \psi_s^\vee(F_v) = \sum_{t \in \ebs (s)} [F_t](F_u F_v) c_t$ if $u, v\in \ebs (s)$ and $\psi_s^\vee(F_u) \cup \psi_s^\vee(F_v) = 0$ otherwise. So it is enough to show for all $t\in \ebs(s)$ that $[F_t](F_u F_v) = 0$ unless $u, v \in\ebs(s)$; this follows from \Cref{prop:quasisubword}.
\end{proof}

\section{Relationship to the Grassmannian}\label{sec:rel_to_grass}

Our goal in this section is to explain the maps on Chow homology and cohomology induced by the embedding of snake dens into the Grassmannian. Recall from \Cref{sec:t_orb} that for any snake den $M$ of rank $r$ on $[n]$ we have chosen a canonical realization inducing a map $\phi_M\colon X_{P(M)} \to \Gr(\rk M,E(M))$. First, we show that these morphisms interact well with slither morphisms. 

\begin{Df}\label[Df]{df:slither_grass}
    Let $S\colon N \to M$ be a slither. Write $f_S\colon E(N) \to E(M)$ for the inclusion of groundsets. Define an embedding $i_S\colon \Gr(\rk N, E(N)) \to \Gr(\rk M, E(M))$ by $i_S^*(p_I) = p_{B}$ if $I = f_S(B)\sqcup \Gamma(S)$ for some $B \in \binom{E(N)}{\rk(N)}$ and $i_S^*(p_I) = 0$ otherwise.\footnote{That the image is contained in $\Gr(\rk M, E(M))$ can be seen by checking that the pullbacks $i^*_S(p_I)$ satisfy the Plücker relations.}
\end{Df}

\begin{Lm}\label[Lm]{lm:little_square}
    Let $N$ and $M$ be snake den matroids and let $S \colon N \to M$ be a slither morphism. Let $i_S$ be as in \Cref{df:slither_grass}. Then, the following diagram commutes
    \[
        \begin{tikzcd}
            X_{P(N)} \arrow[r, "\phi_N"] \arrow[d, "S"'] & \Gr(\rk N, E(N)) \arrow[d, "i_S"] \\
            X_{P(M)} \arrow[r, "\phi_M"'] & \Gr(\rk M, E(M))
        \end{tikzcd}. 
    \]
\end{Lm}
\begin{proof}
    Let $I \in \binom{E(M)}{\rk M}$. 
    It is easy to see that $\phi_N^*(i_S^*(p_I)) = 0 = S^*(\phi_M^*(p_I))$ unless $I = f_S(B) \sqcup \Gamma(S)$ for some $B \in \mathcal{B}(N)$.
    If $I$ is of this form 
    then $\phi_N^*(i_S^*(p_I)) = \phi_N^*(p_B) = \chi^{-B}$ where $\chi^{-B}$ is the section of $\mathcal{O}(P(N))$ corresponding to $B$.
    In the other direction, 
    $\phi_M^*(p_I) = \chi^{-I} \in \mathcal{O}(P(M))$ since $I$ is a basis of $M$ and $S^*(\chi^{-I}) = \chi^{-B}$ since the
    coordinate morphism induced by $S$ sends 
    $e_B$ to $e_I$.
\end{proof}

Note that the map $i_S \colon \Gr(\rk N,E(N)) \to \Gr(\rk M, E(M))$ is, up to a reordering of the groundset, the standard inclusion of $\Gr(r,n)$ in $\Gr(k, m)$ where $r = \rk N, n = |E(N)|$, $k = \rk M$, and $m = |E(M)|$. This implies that on homology, we have $(i_S)_*([\Omega_\lambda]) = [\Omega_\lambda]$. 
This gives the commutative diagram
\[
    \begin{tikzcd}[
      column sep={7em,between origins},
      row sep={3em,between origins}]
      A_{\bl}\bigl(\operatorname{Gr}(\rk N,E(N))\bigr)
        \arrow[rr, "{(i_S)_*}"] \arrow[dr, hook]
      && A_{\bl}\bigl(\operatorname{Gr}(\rk M,E(M))\bigr)
        \arrow[dl, hook'] \\
      & \sym &
    \end{tikzcd}
\]
By \Cref{thm:homology_morphisms}, we similarly have the commutative diagram.
\[
\begin{tikzcd}[
  column sep={7em,between origins},
  row sep={3em,between origins}
]
  A_{\bl}(X_{P(N)})
    \arrow[rr, "{S_*}"] \arrow[dr]
  && A_{\bl}(X_{P(M)})
    \arrow[dl] \\
  & \nsym &
\end{tikzcd}
\]
The natural expectation is that these should be compatible with the surjection $\nsym \to \sym$ taking $R_\alpha \mapsto s_\alpha$. That is, the pushforward in Chow homology induced by the inclusion of a toric Richardson subvariety
into the Grassmannian should agree with
 the map sending non-commutative ribbon functions to
 the corresponding
 ribbon Schur functions. We now show this is indeed the case. 
\begin{Prop}\label[Prop]{prop:compat}
    Let $S\colon N \to M$ be a slither morphism. Then the following diagram commutes. 
    \[
        \begin{tikzcd}[
  column sep={8em,between origins},
  row sep={4em,between origins}
]
  & \nsym \arrow[rr]
      \arrow[from=4-1, to=1-2]
      \arrow[from=4-3, to=1-4]
  && \sym \\
  A_\bl(X_{P(N)})
    \arrow[ur]
    \arrow[rr, crossing over, "(\phi_N)_*" below]
    \arrow[dd, "{S_*}"']
  && A_\bl\bigl(\operatorname{Gr}(\rk N,E(N))\bigr)
       \arrow[ur]
       \arrow[dd, "{(i_S)_*}"']
  & \\
  & & & \\
  A_\bl(X_{P(M)}) \arrow[rr, "(\phi_M)_*" below]
  && A_\bl\bigl(\operatorname{Gr}(\rk M,E(M))\bigr) &
\end{tikzcd}
    \]
\end{Prop}
\begin{proof}
Given the discussion above, it remains only to prove the commutativity of the square
\[
\begin{tikzcd}[
  row sep={3em,between origins},
  column sep={10em,between origins}
]
  A_{\bl}(X_{P(N)}) \arrow[r] \arrow[d]
    & A_{\bl}\bigl(\operatorname{Gr}(\rk N,E(N))\bigr) \arrow[d] \\
  \nsym \arrow[r] & \sym
\end{tikzcd}
\]
By \Cref{lm:little_square}, it is enough to chase $[X_{P(N)}]$. 
From the discussion in \Cref{subsec:richardson_and_snake}, we get  $(\phi_N)_*([X_{P(N)}]) = [\Omega_{\lambda_B}^{\lambda_A}]$, where
$N$ is the lattice path matroid associated to the Gale interval $[A,B]_G$.
The image of  $[\Omega_{\lambda_B}^{\lambda_A}]$ in 
$\sym$ is $s_{\lambda_B/\lambda_A}$.
Since \Cref{thm:homology_morphisms} shows that  $[X_{P(N)}]$ maps to $R_{\lambda_B/\lambda_A}$ in $\nsym$, we are done. 
\end{proof}

\begin{Obs}
    It follows formally from \Cref{prop:compat} that the following diagram also commutes.
    \[
        \begin{tikzcd}[
  column sep={8em,between origins},
  row sep={4em,between origins}
]
  & \qsym \arrow[rr, <-]
      \arrow[from=4-1, to=1-2, <-]
      \arrow[from=4-3, to=1-4, <-]
  && \sym \\
  A^\bl(X_{P(N)})
    \arrow[ur, <-]
    \arrow[rr, <-, crossing over, "\phi_N^*" below]
    \arrow[dd, <-, "{{S}^*}"']
  && A^\bl\bigl(\operatorname{Gr}(\rk N,E(N))\bigr)
       \arrow[ur, <-]
       \arrow[dd, <-, "{(i_S)^*}"']
  & \\
  & & & \\
  A^\bl(X_{P(M)}) \arrow[rr, <-, "\phi_M^*" below]
  && A^\bl\bigl(\operatorname{Gr}(\rk M,E(M))\bigr) &
\end{tikzcd}
    \]
\end{Obs}

\section{An infinite quasisymmetric Grassmannian}\label{sec:infinite}

We now restrict ourselves to working over $\C$ and consider varieties with their analytic topology.
We will require some basic facts of algebraic topology; the books \cite{hatcher}, \cite{may99}, and \cite{may12} provide more than enough background.
Our goal will be to assemble the finite-dimensional snake den varieties into an infinite-dimensional ind-variety $Q_{\infty}$. 
The (singular) homology and cohomology groups of this space will be $\nsym$ and $\qsym$ respectively. 
We will equip $Q_{\infty}$ with the structure of a (weak) $H$-group. As a result, $H_{\bl}(Q_{\infty}, \Z)$ and $H^{\bl}(Q_{\infty}, \Z)$ become dual graded Hopf algebras. The isomorphisms 
$H_{\bl}(Q_{\infty}, \Z) \cong \nsym$ and $H^{\bl}(Q_{\infty}, \Z) \cong \qsym$ will be in the sense of Hopf algebras.
The previously constructed morphisms 
$H_{\bl}(X_s, \Z) \to \nsym$ and 
$\qsym\to H^{\bl}(X_s, \Z)$ for snake dens $s$
will be induced by inclusion maps 
$X_s \hookrightarrow Q_{\infty}$.\footnote{In Sections \ref{sec:homology} and \ref{sec:cohomology} we constructed such morphisms for Chow homology. As pointed out there, \Cref{lm:cycle_class} allows us to transport these results unchanged to singular homology and cohomology.}
Thus the behavior of the finite snake varieties is to be seen as a finite shadow of the story for $Q_{\infty}$.

\subsection{The construction}\label{subsec:construction}
We recall the following two observations about the category of snake den varieties:
\begin{itemize}
    \item Any two snake dens admit slither morphisms  to a common snake.
    \item Any two slither morphisms from a snake den variety into a snake variety are homotopic (\Cref{cor:homotopic}). Hence any two slither morphisms between snake den varieties are coequalized by (i.e.\ become equal after post-composing with) a further slither
    morphism into a snake variety.
\end{itemize}
These two observations combine to say that snake dens form a  filtered subcategory of the homotopy category of topological spaces.
Hence, morally, it should be possible to take the direct limit of all snake den varieties in this category. To avoid a plethora of technical issues, we will instead take the directed limit of a sequence of snakes in the category of topological spaces.
Thus we start by choosing an infinite sequence of extended snakes
$s(0), s(1), s(2), s(3), \dots$ (represented by their snake words) along with slithers 
$t(i, i+1)\colon s(i) \to s(i+1)$.\footnote{This implies that each $s(i)$ is a subword of $s(i+1)$.} 
For convenience, we insist that $s(0) = \eps$ be a point.
We also insist that the number of alternations between $0$ and $1$ in $s(i)$ be unbounded as $i \longrightarrow \infty$. A concrete choice (for $i \geq 1$) is given by setting $s(i)$ to be the initial segment of length $i-1$ taken from the infinite binary string $10101010\dots$ and setting 
$t(i, i+1) = \text{$\pr s(i) \rp \elll$ or $\pr s(i) \rp \llle$}$, depending on the parity of $i$. Later, we shall see that the choices we make here are almost entirely inconsequential.

We write $X(i)$ for $X_{s(i)}$ and still write 
$t(i, i+1)$ for the induced morphism of toric varieties.
More generally, for $i \leq j$ we write $t(i,j)\colon X(i) \to X(j)$ for 
$t(j-1, j) \circ \dots \circ t(i+1, i+2) \circ t(i, i+1)$. In particular, $t(i,i)$ is the identity map $X(i) \to X(i)$.
Equipped with their analytic topology over $\C$, each $X(i)$ is a compact Hausdorff space by \cite{SGA1}*{Propositions 3.1(viii), 3.2(v)}; it is simply connected by \cite{fult93toric}*{\textsection~3.2}.
Let $Q_{\infty}$ be the topological direct limit $\varinjlim_i X(i)$ and let $t(i)\colon X(i) \to Q_{\infty}$ for $i \geq 0$ be the canonical inclusions. It is easy to show, by a combination of Urysohn's lemma and Tietze's extension theorem \cite{TTD}*{Theorems 1.1.1, 1.1.2}, that $Q_{\infty}$ is a normal Hausdorff space. 
Given that each $X(i)$ is simply connected, 
the following observation implies that $Q_{\infty}$ is too.
\begin{Prop}
[{\cite{TTD}*{Proposition 1.4.5}}]\label{prop:to_finite}
    Let $X_1 \hookrightarrow X_2 \hookrightarrow \dots$ be an infinite sequence of embeddings of Hausdorff topological spaces. 
    Then any continuous map $Y \to \varinjlim_i X_i$ from a compact Hausdorff space $Y$ factors through $X_j$ for some finite $j$. 
\end{Prop}

We now explain how to embed \emph{any} snake den into $Q_{\infty}$. First suppose $s$ is a snake. The choice of the sequence $s(0), s(1), s(2), \dots$ ensures that we can find some $s(i)$ such that the binary string $s$ is a substring of $s(i)$. Then we embed $X_s$ into $Q_{\infty}$ using the composition 
$X_s \hookrightarrow X(i) \hookrightarrow Q_{\infty}$, where the first map is a slither morphism. If $s = s(i)$, this agrees with the colimit inclusion $X(i) \hookrightarrow Q_{\infty}$. If $s$ is a general snake den, we first embed $X_s$ into a snake variety $X_{s'}$ via a slither morphism and then embed $X_s$ into 
$Q_{\infty}$ using $X_s \hookrightarrow X_{s'} \hookrightarrow Q_{\infty}$. 
Any embedding $X_s \to Q_{\infty}$ obtained by such a procedure will be called a \textbf{slither embedding}.
There are infinitely many slither embeddings of any given snake den $X_s$ into $Q_{\infty}$; however, they are all \emph{homotopic}.
Indeed, given a pair of inclusions $X_s \hookrightarrow X(i) \hookrightarrow Q_{\infty}$
and $X_s \hookrightarrow X(j) \hookrightarrow Q_{\infty}$ as above, the two resulting slither embeddings are homotopic in 
$X(\max(i,j))$ by \Cref{cor:homotopic}. 

This uniqueness of slither embeddings up to homotopy ensures that we have well-defined \emph{functorial} maps
$H_{\bl}(X_s, \Z) \to H_{\bl}(Q_{\infty}, \Z)$ and $H^{\bl}(Q_{\infty}, \Z) \to H^{\bl}(X_{s}, \Z)$ for all snake dens $s$, the latter being a ring homomorphism.
By \emph{functoriality}, we mean the commutativity of the diagrams
\[
\begin{tikzcd}
    H_{\bl}(X_s, \Z)\ar[rr]\ar[rd] && H_{\bl}(X_{s'}, \Z)\ar[ld] && H^{\bl}(Q_{\infty}, \Z)\ar[rd]\ar[ld]\\
    &H_{\bl}(Q_{\infty}, \Z) && H^{\bl}(X_{s'}, \Z) && H^{\bl}(X_{s}, \Z)\ar[ll, <-]
\end{tikzcd}
\]
In the next subsection, we will show that 
\[
    \varinjlim H_{\bl}(X_s, \Z) = H_{\bl}(Q_{\infty}, \Z) \cong \nsym \qquad \text{and} \qquad 
    \text{$\varprojlim H^{j}(X_s, \Z) = H^{j}(Q_{\infty}, \Z) \cong \qsym^j$ for each $j$},
\]
the maps to $\nsym$ and $\qsym$ being the ones constructed in Sections \ref{sec:homology} and \ref{sec:cohomology}.
The sense of the isomorphisms must be clarified. A priori, $H_{\bl}(Q_{\infty}, \Z)$ is simply a graded abelian group and $H^{\bl}(Q_{\infty}, \Z)$ a graded ring. However, we will show how  to equip $Q_{\infty}$ with the structure of a (weak) $H$-group, so that its total homology and cohomology groups become dual graded Hopf algebras. The isomorphisms 
$H_{\bl}(Q_{\infty}, \Z) \cong \nsym$ and $H^{\bl}(Q_{\infty}, \Z) \cong \qsym$ will be isomorphisms of graded Hopf algebras.

The following technical considerations will be useful below.
\begin{Df}\label{df:cofibration}
    An embedding of topological spaces $i\colon X \to Y$ is a \textbf{cofibration} if, for each continuous map $g\colon Y \to Z$ and each homotopy 
    $H_t\colon X \to Z$ (with parameter $t \in [0, 1]$) satisfying 
    $H_0 = g \circ i$, there exists an extension $\widetilde{H}_t\colon Y \to Z$ such that 
    $\widetilde{H}_0 = g$ and
    $\widetilde{H}_t \circ i = H_t$.
\end{Df}

\begin{Prop}\label{prop:CW}
    Every projective $\C$-variety $X$ (with its analytic topology) admits the structure of a finite CW complex.
    Each inclusion of a closed subvariety $Y \subseteq X$ into $X$ is a cofibration. 
\end{Prop}
\begin{proof}
It follows from a theorem of Lojasiewicz \cite{Lojasiewicz} (see also Hironaka \cite{hironaka} and 
\cite{RAG}*{Theorem 9.2.1}), along with the fact that every complex projective algebraic variety is a real affine variety \cite{RAG}*{Proposition 3.4.6}, 
that $X$ admits a triangulation of which $Y$ is a subcomplex. This proves the first claim, and implies that 
$(X, Y)$ is a CW-pair. Any CW-pair gives rise to a cofibration by \cite{hatcher}*{Proposition 0.16}.
\end{proof}
\noindent
The following is an easy consequence of \Cref{df:cofibration} and \Cref{prop:product_colim} below (applied to $Z = [0, 1]$).
\begin{Obs}[{cf.\ \cite{spectra}*{Lemma 1.2.2.4}}]\label{obs:cofibration_limit}
    If $X_1 \hookrightarrow X_2 \hookrightarrow \dots$ is an infinite sequence of cofibrations, then each inclusion 
    $X_i \hookrightarrow \varinjlim_j X_j$ is a cofibration.
\end{Obs}
\noindent
\Cref{prop:CW} and \Cref{obs:cofibration_limit} together show that each slither embedding is a cofibration.

\begin{Df}
    A \textbf{well-pointed} topological space is a topological space $X$ along with a choice of basepoint $x \in X$ such that the inclusion $x \hookrightarrow X$ is a cofibration. A \textbf{pointed map} between well-pointed topological spaces $X, Y$ is a continuous map $X \to Y$ sending the basepoint of $X$ to the basepoint of $Y$.
    A \textbf{pointed homotopy} between pointed maps is a homotopy through pointed maps.
\end{Df}

\begin{Df}\label{df:weak_homotopy}
    The \textbf{weak homotopy category} is the category in which the objects are well-pointed topological spaces homotopy equivalent to CW complexes and the morphisms are pointed continuous maps \emph{considered up to pointed weak homotopy}. Here, two continuous maps 
    $f_1, f_2\colon X \to Y$ are \textbf{(pointedly) weakly homotopic} if $f_1 \circ g$ is (pointedly) homotopic to $f_2 \circ g$ for every (pointed) continuous map $g\colon Z \to X$ where $Z$ is a \emph{finite} CW complex. A \textbf{weak $H$-group} is a group object in the weak homotopy category.
\end{Df}

Explicitly, a weak $H$-group consists of a pointed space $X$ along with pointed continuous maps $\mu\colon X \times X \to X$ and 
$\iota \colon X \to X$ such that the diagrams below commute up to weak homotopy. 
From left to right, these are the associativity, unitality, and invertibility axioms. 
Here
$e\colon \text{pt} \to X$ is the inclusion of the basepoint,
$\delta\colon X \to X \times X$ is the diagonal map, and $c\colon X \to X$ is the constant map to the basepoint.
\[
\begin{tikzcd}[column sep=0.7cm]
    X \times X \times X\ar[r, "\mu \times \Id"]\ar[d, "\Id \times \mu" left] & X \times X\ar[d, "\mu"] && X\ar[rd, "\Id"] 
    \ar[r, "e \times \Id"]\ar[d, "\Id \times e" left] & X\times X\ar[d, "\mu"] &&
    X\ar[d, "\delta" left] \ar[r, "c"] & X & X\ar[l, "c" above]\ar[d,"\delta"]\\
    X \times X\ar[r, "\mu" below] & X && X \times X\ar[r, "\mu" below] & X && X \times X \ar[r, "\Id \times \iota" below] & X \times X\ar[u, "\mu" right] & X \times X\ar[l, "\iota \times \Id"]
\end{tikzcd}
\]
Before describing the $H$-group structure on $Q_{\infty}$, we check that this space belongs to the weak homotopy category. Since each $X(i)$ is a CW complex by \Cref{prop:CW}, \cite{birkhoff}*{Proposition 3.2} shows that $Q_{\infty}$ has the homotopy type of a CW complex.\footnote{The example at the start of \cite{birkhoff}*{\textsection~3.2} shows that an ind-variety need not admit the structure of a CW complex.}
The space
$Q_{\infty}$ has a canonical basepoint, namely $t(0): X(0) \hookrightarrow Q_{\infty}$, whose inclusion is a cofibration by \Cref{obs:cofibration_limit}. To define the product map $\mu$, we will need some topological preliminaries.
We start with the following observation.
\begin{Obs}\label{obs:extension}
Let $h\colon X \to Y$ be a cofibration, and let
$f\colon X \to Z$ and $g\colon Y \to Z$ be continuous maps such that 
$g \circ h$ is homotopic to $f$. 
Then there exists a continuous map 
$\widetilde{g}\colon Y \to Z$ homotopic to $g$ such that $\widetilde{g} \circ h = f$.
\end{Obs}
\begin{proof}
    Let $H_t\colon X \to Z$ be a homotopy with 
    $H_0 = g \circ h$ and $H_1 = f$.
    Applying \Cref{df:cofibration}, we get a homotopy $\widetilde{H}_t\colon Y \to Z$ such that $\widetilde{H}_0 = g$ and 
    $\widetilde{H}_1 \circ h = f$. Set $\widetilde{g} \deq \widetilde{H}_1$.
\end{proof}
\begin{Lm}\label{lm:pointedness}
    Let $X, Y$ be well-pointed topological spaces, with $Y$ simply connected.
    Then each continuous map $f\colon X \to Y$ is weakly homotopic to a pointed map; two pointed continuous maps $f, g\colon X \to Y$ are pointedly weakly homotopic if and only if they are weakly homotopic.
\end{Lm}
\begin{proof}
    Since $Y$ is path-connected, all inclusions of a point into $Y$ are homotopic; so the first statement follows from \Cref{obs:extension}.
    For the second statement, the forward implication is trivial. We prove the converse. Suppose $h\colon Z \to X$ is a pointed map from a finite well-pointed CW complex $Z$. We must show that $f \circ h$ and $g \circ h$ are homotopic as pointed maps $Z \to Y$. By hypothesis, they are homotopic as unpointed maps; the rest follows from \cite{may12}*{Lemma 1.4.2}.
\end{proof}
Since all the spaces we deal with are simply connected, \Cref{lm:pointedness} ensures that we can safely ignore all pointedness requirements in \Cref{df:weak_homotopy}. We will henceforth do so without further comment.
\begin{Prop}\label{prop:product_colim}
    Let $X_1 \hookrightarrow X_2 \hookrightarrow \dots$ and 
    $Y_1 \hookrightarrow Y_2 \hookrightarrow \dots$ be infinite sequences of embeddings of compact Hausdorff topological spaces, and let $Z$ be another such space. Then 
    there are natural homeomorphisms
    \[
        Z \times \varinjlim_{i} X_i \cong \varinjlim_{i} Z \times X_i \qquad \text{and} \qquad \varinjlim_{i,j} X_i \times Y_j \cong \varinjlim_i X_i \times \varinjlim_j Y_j.
    \]
\end{Prop}
\begin{proof}
    The first statement is the special case of the second. The second is \cite{HSTH}*{Theorem 4.1}.
\end{proof}
\noindent
We can now define the multiplication map $\mu$ on $Q_{\infty}$. 
\begin{Prop}\label{prop:define_mult}
    Up to weak homotopy, there exists a unique continuous map 
    $\mu\colon Q_{\infty} \times Q_{\infty} \to Q_{\infty}$ making the following diagram commute up to homotopy for all $i, j$.
    \[
    \begin{tikzcd}
        X(i) \times X(j) \ar[d, hook, "t(i) \times t(j)" left]\ar[r, equals] & \ar[r, equals] X_{\pr s(i) \rp} \times X_{\pr s(j) \rp}& X_{\pr s(i) \rp \pr s(j) \rp}\ar[d, hook]\\
        Q_{\infty} \times Q_{\infty}\ar[rr, "\mu" below] && Q_{\infty}
    \end{tikzcd}
    \]
    Here, the inclusion
    $X_{\pr s(i) \rp \pr s(j) \rp} \to Q_{\infty}$ is an arbitrary slither embedding.
\end{Prop}
\begin{proof}
    We start by noting that 
    $\varinjlim_{i,j} X(i) \times X(j) = Q_{\infty} \times Q_{\infty}$ by \Cref{prop:product_colim}.
    For uniqueness of $\mu$, suppose $\mu_1, \mu_2$ are two maps as above. Let $Z$ be any finite CW complex and $f\colon Z \to Q_{\infty} \times Q_{\infty}$ a continuous map.
    By \Cref{prop:to_finite}, $f$ factors through $X(i) \times X(j)$ for some finite $i,j$. The commutativity of the above diagram then implies that $\mu_1 \circ f$ and 
    $\mu_2 \circ f$ are homotopic. We conclude that $\mu_1$ and $\mu_2$ are weakly homotopic.
    
    For existence, it suffices to ensure commutativity when $i = j$ (this again uses \Cref{cor:homotopic}). Since $Q_{\infty} \times Q_{\infty} = \varinjlim_{i} X(i) \times X(i)$ (again using \Cref{prop:product_colim}), we may define $\mu$ by defining 
    maps $\mu_i\colon X(i) \times X(i) \to Q_{\infty}$ such that $\mu_{i} = \mu_{i+1}|_{X(i) \times X(i)}$. We define $\mu_i$ by induction on $i$. To start off, we let
    $\mu_0$ send $\text{pt} = X_{s(0)} \times X_{s(0)}$ to the basepoint of $Q_{\infty}$. In the general case, suppose we have defined $\mu_i$.
    The
    composition 
    \[
        X(i) \times X(i) \hookrightarrow X(i+1) \times X(i+1) = X_{\pr s(i+1) \rp\pr s(i+1)\rp} \hookrightarrow Q_{\infty}
    \]
    is homotopic to the composition
    \[
        X(i) \times X(i) = X_{\pr s(i) \rp\pr s(i)\rp}  \hookrightarrow X_{\pr s(i+1) \rp\pr s(i+1)\rp} \hookrightarrow Q_{\infty},
    \]
    where the middle arrow is given by a slither morphism. Since the latter  of the two is homotopic to $\mu_i$ by the induction hypothesis, \Cref{obs:extension} shows that we can find $\mu_{i+1} \colon X(i+1) \times X(i+1) \to Q_{\infty}$ which is homotopic to $X(i+1) \times X(i+1) = X_{\pr s(i+1) \rp\pr s(i+1)\rp} \hookrightarrow Q_{\infty}$ such that 
    $\mu_{i+1}|_{X(i) \times X(i)} = \mu_i$. 
\end{proof}
\noindent
Having defined $\mu$ (in the weak homotopy category), we now check that it satisfies the axioms.
\begin{Prop}\label{prop:exist_H_group}
    Up to weak homotopy, there is a unique weak $H$-group structure on $Q_{\infty}$ for which $\mu$ is the multiplication map.
\end{Prop}
\begin{proof}
We start by verifying the unitality and associativity axioms for $\mu$.
First observe that, by construction, the composition 
$X(i) \times X(j) \to Q_{\infty} \times Q_{\infty} \xrightarrow[]{\mu} Q_{\infty}$ for any $i, j$ is homotopic to a slither embedding 
$X_{\pr s(i) \rp\pr s(j) \rp} \to Q_{\infty}$. 
Setting $j = 0$, 
the composition 
$X(i) \to Q_{\infty} \xrightarrow[]{\Id \times e} Q_{\infty} \times Q_{\infty} \xrightarrow[]{\mu} Q_{\infty}$ is thereby shown to be homotopic to $t(i)$. Since the image of $Z \to Q_{\infty}$ for any finite CW complex $Z$ lands in $X(i)$ for some finite $i$, this proves half of the unitality axiom. The other half is proven similarly.

The same trick proves associativity, provided we can show that the composition 
\[
X(i) \times X(j) \times X(k) \to Q_{\infty} \times Q_{\infty} \times Q_{\infty} \xrightarrow{\mu \times \Id} Q_{\infty} \times Q_{\infty} \xrightarrow[]{\mu} Q_{\infty}
\]
is homotopic to a slither embedding $X_{\pr s(i) \rp \pr s(j) \rp \pr s(k) \rp} \to Q_{\infty}$ (the other half will follow by symmetry). The composition of the first two maps is homotopic to 
${X_{\pr s(i) \rp \pr s(j) \rp} \times X(k) \xrightarrow[]{\sigma \times t(k)} Q_{\infty} \times Q_{\infty}}$ where $\sigma$ is a slither embedding. Choose $m$ such that $\sigma$ lands in $X(m)$. Then it is easy to see that 
\[
    X_{\pr s(i) \rp \pr s(j) \rp \pr s(k) \rp} = X_{\pr s(i) \rp \pr s(j) \rp} \times X(k) \to X(m) \times X(k) = X_{\pr s(m) \rp \pr s(k) \rp}
\]
is a slither morphism (given by the concatenation of the slither representing $X_{\pr s(i) \rp \pr s(j) \rp} \to X(m)$ with the identity slither $\pr s(k) \rp$). It therefore suffices to check that $X(m) \times X(k) \hookrightarrow Q_{\infty} \times Q_{\infty} \xrightarrow[]{\mu} Q_{\infty}$ is homotopic to a slither embedding; but this is true by construction.

Since $Q_{\infty}$ has the homotopy type of a CW complex, \cite{W67}*{Appendix} implies that the invertibility axiom is redundant in our setting.\footnote{The statement in \cite{W67} gives separate left- and right-inverses for $\mu$, but the associativity of $\mu$ forces these to be equal up to weak homotopy.}
The usual group-theory proof shows that the inverse map $\iota\colon Q_{\infty} \to Q_{\infty}$ is unique up to weak homotopy.
\end{proof}

\begin{Rmk}
    In the proof of \Cref{prop:exist_H_group}, the existence of the inverse map is proved indirectly. One can show that it is induced, in a suitable sense, from maps on finite snakes.
    For any snake $s$, the $P(s^*) = P(s)^*$ differs from $-P(s)$ only by a translation.
    As a result, we get a ``duality map'' 
    $\ast_s \colon X_s \to X_{s^*}$. 
    Write $\dagger_s\colon X_s \to X_s$ for the complex conjugation map. Then the $H$-group inverse $\iota\colon Q_{\infty} \to Q_{\infty}$ is induced by the composite maps $\ast_s \circ \dagger_s = \dagger_{s^*} \circ \ast_s$.
\end{Rmk}

\subsection{Homology and cohomology}
We say that a topological space is \textbf{good} if
it has the homotopy type of a CW complex and
its integral homology is free over $\Z$ and finitely generated in each degree.
Two weakly homotopic morphisms of topological spaces induce the same map on homology. 
The same is true on cohomology, provided the two spaces are good, since in this case homology and cohomology are dual graded $\Z$-modules (by the universal coefficient theorem for cohomology \cite{hatcher}*{Theorem 3.2}).
Thus homology and cohomology descend to the full subcategory of good spaces in the \emph{weak} homotopy category. 
Restricted to this subcategory, cohomology is a contravariant functor $H^{\bl}(-,\Z)$ to the category $\zalg$ of graded-commutative $\Z$-algebras (with the half-degree convention of \Cref{sec:chow_section}).  The K\"unneth theorem for cohomology \cite{hatcher}*{Theorem 3.15} shows that products of good spaces are good and that $H^{\bl}(-,\Z)$ sends products of \emph{good} spaces in the weak homotopy category to coproducts in $\zalg$, i.e.\ products in the opposite category $\zalg^{\opp}$. 
It also sends the terminal object in the weak homotopy category (namely, a point) to the terminal object in $\zalg^{\opp}$ (namely, $\Z$).
Functors that preserve finite products send group objects to group objects.
Graded-commutative Hopf algebras are exactly the group objects in 
$\zalg^{\opp}$.\footnote{See the discussion after \cite{GR20}*{Definition 1.3.3} for sign-related subtleties in the definition of a Hopf algebra. Since in our setting the cohomology vanishes in odd degree, no issue arises.}
It follows that the cohomology of any weak $H$-group $X$ (with a good underlying space) is canonically equipped with the structure of a graded-commutative Hopf algebra.
The comultiplication map in the Hopf algebra is given by the pullback map $\mu^*\colon H^{\bl}(X, \Z) \to H^{\bl}(X \times X, \Z)$ induced by the multiplication map $\mu$. 
The universal coefficient theorem for cohomology then shows that the homology of $X$ is naturally equipped with a Hopf algebra structure as well, namely that of the graded dual of the Hopf algebra $H^{\bl}(X, \Z)$.

We now show that $Q_{\infty}$ is good and relate
the Hopf algebra structures on its integral homology and cohomology to the homology and cohomology of finite snake dens and to the Hopf algebras $\nsym$ and $\qsym$.
It follows from \cite{may12}*{Proposition 2.5.6} 
that the total homology group of $Q_{\infty}$ is the direct limit of the total homology groups of
$X(i)$. In symbols,
\[
    H_\bl(Q_{\infty}, \Z) = \varinjlim H_\bl(X(i)) \cong \varinjlim\nsym_{s(i)}  =\nsym,
\]
where we have used \Cref{thm:homology_morphisms}.
This shows in particular that $Q_{\infty}$ is a good space.
The induced map on homology $H_\bl(X(i))\to H_\bl(Q_{\infty})$ is given by the map $\psi_{s(i)}$ in that theorem. More generally,
for any snake den $s$ and any slither embedding $X_s \hookrightarrow Q_{\infty}$, the fact that this embedding factors as $X_s \hookrightarrow X(i) \hookrightarrow Q_{\infty}$ (for large enough $i$) shows by 
\Cref{thm:homology_morphisms} that the pushforward in homology $H_{\bl}(X_s, \Z) \to H_{\bl}(Q_{\infty}, \Z)$ coincides with $\psi_{s}$.

Taking the \emph{graded dual} of the isomorphism  $H_{\bl}(Q_{\infty}, \Z) \cong \nsym$ gives the isomorphism of \emph{graded $\Z$-modules}
${H^{\bl}(Q_{\infty}, \Z) \cong \qsym}$.
We upgrade this to an isomorphism of Hopf algebras. 
The unit and counit are the inclusion and projection to $\Z$ in both cases, so these are preserved.
We will check that the isomorphism preserves both multiplication and comultiplication. It will then follow that the antipode is preserved too.

We start by checking that multiplication is preserved.
Since the duality functor is an involutive endofunctor on the category of finitely generated free $\Z$-modules, $H_{\bl}(Q_{\infty}, \Z) = \varinjlim H_{\bl}(X(i), \Z)$ implies
\[
    H^\bl(Q_{\infty}, \Z) = \varprojlim H^\bl(X(i), \Z) \cong \varprojlim\qsym_{s(i)}  =\qsym,
\]
where we take limits as \emph{graded} objects, i.e.\ we take limits in each degree separately.
These are a priori isomorphisms of graded {$\Z$-modules}. However, since the maps 
$H^{\bl}(Q_{\infty}) \to H^{\bl}(X(i))$ and 
$H^{\bl}(X(j)) \to H^{\bl}(X(i))$ are algebra homomorphisms, the equality 
$H^\bl(Q_{\infty}) = \varprojlim H^\bl(X(i))$ holds as graded algebras.
Similarly, 
$H^{\bl}(X(i)) \cong \qsym_{s(i)}$ is a graded algebra isomorphism by \Cref{thm:cohomology_functorial}.
So $\varprojlim H^\bl(X(i)) \cong \qsym$ holds as graded algebras. In particular, the isomorphism $H^\bl(Q_{\infty}) \cong \qsym$ preserves multiplication.

We now check that comultiplication is preserved.
For simplicity, we prove instead the dual statement that the isomorphism $H_{\bl}(Q_{\infty}, \Z) \cong \nsym$ preserves multiplication.
Note that the multiplication in $H_{\bl}(Q_{\infty}, \Z)$ is given by $m\colon H_\bl(Q_{\infty}, \Z)\otimes H_\bl(Q_{\infty}, \Z) \cong H_{\bl}(Q_{\infty} \times Q_{\infty}, \Z) \xrightarrow{\mu_*} H_\bl(Q_{\infty}, \Z)$ 
where $\mu_*$ is the pushforward in homology and the first map is the K\"unneth isomorphism in homology
(see \cite{hatcher}*{Theorem 3B.6}, \cite{may99}*{\textsection~17.2}).
Identifying $H_{\bl}(Q_{\infty}, \Z)$ with 
$\nsym$, it suffices to check that $m(R_s \otimes R_t) = R_sR_t$ for each pair of binary strings $s,t$. Choose $i$ such that $s$ and $t$ are both substrings of $s(i)$. \Cref{prop:define_mult} shows that the following diagram commutes up to homotopy.
\[
\begin{tikzcd}
    && X(i) \times X(i)\ar[dd, hook]\ar[r, hook] & X_{\pr s(i) \rp\pr s(i) \rp}\ar[dd, hook]\\
    X_{\pr s \rp\pr t\rp} \ar[r, equals] & X_s \times X_t\ar[ur, hook]\ar[dr, hook]\\
    && Q_{\infty} \times Q_{\infty}\ar[r, "\mu"] & Q_{\infty}
\end{tikzcd}
\]
The class $m(R_s \otimes R_t)$ is by definition the pushforward to $Q_{\infty}$ of the fundamental class 
$[X_s \times X_t]$ along the bottom path. The top-right path sends this class to $R_s R_t$ by \Cref{thm:homology_morphisms}. The commutativity of the diagram up to homotopy gives the desired equality.

Our discussion thus far has proved the following result.
\begin{Thm}\label{thm:X_infty_full_structure}
    The weak $H$-group structure on $Q_{\infty}$ endows $H_{\bl}(Q_{\infty}, \Z)$ and $H^{\bl}(Q_{\infty}, \Z)$ with Hopf algebra structures. These two Hopf algebras are 
    graded duals of one another, isomorphic to 
    $\nsym$ and $\qsym$ respectively.
    The induced maps $H_{\bl}(X_s, \Z) \to H_{\bl}(Q_{\infty}, \Z)$
    and $H^{\bl}(Q_{\infty}, \Z) \to H^{\bl}(X_s, \Z)$ for all snake dens $s$ thereby coincide with the maps constructed in Theorems \ref{thm:homology_morphisms} and \ref{thm:cohomology_functorial}. 
\end{Thm}
Though we have so far only checked this when $s = s(i)$ for some $i$, the general case follows from the fact that $X_s$ admits a slither embedding into some $X(i)$, along with Theorems \ref{thm:homology_morphisms} and \ref{thm:cohomology_functorial}.

\begin{Rmk}\label{rmk:stability}
    A refinement of the last statement in \Cref{thm:X_infty_full_structure} is that $\nsym \cong H_{\bl}(Q_{\infty}, \Z)$ and $\qsym \cong H^{\bl}(Q_{\infty}, \Z)$ are the \emph{stable} homology and cohomology of $X_s$ for snakes $s$. More precisely, the maps
    $H_{\bl}(X_s, \Z) \to H_{\bl}(Q_{\infty}, \Z)$
    and $H^{\bl}(Q_{\infty}, \Z) \to H^{\bl}(X_s, \Z)$
    are isomorphisms in degrees $\leq d$, provided the binary string for $s$ has at least $2d$ alternations between $0$ and $1$.
    Furthermore, the first map is always an injection and the second is always a surjection.
\end{Rmk}

\begin{Rmk}\label{rmk:Grass}
    One can construct a weak $H$-group homomorphism $Q_{\infty} \to \Gr$ from $Q_{\infty}$ to the infinite Grassmannian.
    The map is constructed using the maps $\phi_M$ in \Cref{sec:rel_to_grass}.
    The product structures are compatible because both are induced by taking direct sums (of matroids or vector spaces).
    The results in \Cref{sec:rel_to_grass} then imply that the induced map on homology is the usual (graded) Hopf algebra homomorphism 
    $\nsym \to \sym$, and likewise the induced map on cohomology is the usual Hopf algebra homomorphism 
    $\sym \to \qsym$.
\end{Rmk}

We conclude this section by showing that, although our construction of $Q_{\infty}$ involved an infinite number of choices, this object is well-defined up to isomorphism in the weak homotopy category.
Suppose that we choose a second sequence of snakes 
$s'(0), s'(1), s'(2), \dots$ along with slithers 
$t'(i,i+1)\colon s'(i) \to s'(i+1)$ giving rise to a direct limit 
$Q_{\infty}' = \varinjlim_i X(i)'$, where $X(i)' \deq X_{s'(i)}$.
As before, we assume that the number of alternations between $0$ and $1$ in $s'(i)$ grows without bound.
Just as in \Cref{subsec:construction}, we  then construct a map 
$\mu'\colon Q_{\infty}' \times Q_{\infty}' \to Q_{\infty}'$ giving rise to a weak $H$-group structure on $Q_{\infty}'$.
\begin{Thm}\label{thm:unique}
The weak $H$-groups $Q_{\infty}$ and $Q_{\infty}'$ are isomorphic.
\end{Thm}
\begin{proof}
    We start by constructing a {homotopy equivalence}  (not weak!) between $Q_{\infty}$ and $Q_{\infty}'$. We construct a map 
    $\eta\colon Q_{\infty} \to Q_{\infty}'$ much as in \Cref{prop:define_mult}. Namely, it suffices to exhibit maps $\eta_i\colon X(i) \to Q_{\infty}'$ such that $\eta_{i+1} |_{X(i)} = \eta_i$. 
    We note that, as in the construction of $Q_{\infty}$, we have slither embeddings 
    $X(i) = X_{s(i)} \hookrightarrow X(j)' \hookrightarrow Q_{\infty}'$ (where $j$ depends on $i$) and these slither embeddings are uniquely defined up to homotopy.
    We will ensure that each map $\eta_i$ is homotopic to a slither embedding.
    
    The construction proceeds by induction on $i$. We map $X(0)$ to the basepoint of $X(0)'$. Having defined $\eta_i$, we observe that $X(i) \xrightarrow[]{\eta_i} Q_{\infty}'$ is homotopic to the slither embedding 
    $X(i) \hookrightarrow X(i+1) \hookrightarrow Q_{\infty}'$. By \Cref{obs:extension}, we can find $\eta_{i+1}\colon X(i + 1) \to Q_{\infty}'$ homotopic to the slither embedding such that 
    $\eta_{i+1}|_{X(i)} = \eta_i$.
    In this way, we construct a map $Q_{\infty} \to Q_{\infty}'$. We claim that this map induces an isomorphism on all homology groups. In fact, 
    \Cref{rmk:stability} for $Q_{\infty}$ and $Q_{\infty}'$ shows that the commutative diagram
    \[
        \begin{tikzcd}
            & H_{\bl}(Q_{\infty}, \Z)\ar[rd, "\eta_*"] &\\
            H_{\bl}(X(i), \Z)\ar[ur, "t(i)_*"]\ar[rr, "(\eta_i)_*" below] &&  H_{\bl}(Q_{\infty}', \Z)
        \end{tikzcd}
    \]
    consists of isomorphisms when we restrict to degrees $\leq d$ for large enough $i$. Since $d$ can be made arbitrarily large, $\eta_*$ is an isomorphism. Since $Q_{\infty}$ and $Q_{\infty}'$ are simply connected and are homotopy equivalent to CW complexes, the dual Whitehead theorems \cite{dual_W} imply that $\eta$ is a homotopy equivalence.

    It remains only to show that multiplication in $Q_{\infty}$ and $Q_{\infty}'$ coincide up to weak homotopy. For each $i$, and for $j$ sufficiently large, consider in the weak homotopy category the diagram
\[
\begin{tikzcd}[
  row sep={0cm,between origins},
  column sep={0cm,between origins},
  cells={nodes={inner sep=3pt,outer sep=0pt}},
  arrows={-{Stealth[length=2mm]},line width=.55pt},
  labels={font=\small,fill=white,inner sep=2pt}
]
X(i) \times X(i)
&[1.5cm] {} &[2.3cm] 
X_{\pr s(i)\rp\pr s(i)\rp} &[1.5cm] {} \\[1.2cm]
{} & 
Q_{\infty} \times Q_{\infty} & {} & 
Q_{\infty} \\[2cm]
X(j)' \times X(j)' & {} & X_{\pr s'(j) \rp \pr s'(j) \rp} & {} \\[1.2cm]
{} & 
Q_{\infty}' \times Q_{\infty}' & {} & 
Q_{\infty}'
\arrow[from=1-1,to=1-3,equals]
\arrow[from=1-1,to=3-1,"\eta_i \times \eta_i"']
\arrow[from=1-3,to=3-3,]
\arrow[from=3-1,to=3-3,equals]
\arrow[from=1-1,to=2-2]
\arrow[from=1-3,to=2-4]
\arrow[from=3-1,to=4-2]
\arrow[from=3-3,to=4-4]
\arrow[from=2-2,to=2-4,crossing over,"\mu"{pos=.35}]
\arrow[from=2-2,to=4-2,crossing over,"\eta \times \eta"'{pos=.38}]
\arrow[from=2-4,to=4-4,"\eta"]
\arrow[from=4-2,to=4-4,"\mu'"']
\end{tikzcd}
\]
The two lateral faces commute by construction. The top and bottom faces commute by \Cref{prop:define_mult} for $Q_{\infty}$ and $Q_{\infty}'$. The back face commutes because all maps involved are homotopic to slither morphisms. By a diagram chase, it follows that $\eta \circ \mu$ and $\mu' \circ (\eta \times \eta)$ agree when restricted to $X(i) \times X(i)$. Since every map to $Q_{\infty} \times Q_{\infty}$ from a finite CW complex factors through $X(i) \times X(i)$ for \emph{some} $i$, this shows that 
$\eta \circ \mu$ and $\mu' \circ (\eta \times \eta)$ are weakly homotopic.
\end{proof}

\Cref{thm:unique} shows that snake varieties limit to a well-defined group object in the weak homotopy category with homology and cohomology given by $\nsym$ and $\qsym$ respectively.
In \cite{BR08}, it is shown that another space, the loop space of the suspension of $\Pb^{\infty}_{\C}$ (denoted $\Omega \Sigma \C \Pb^\infty$), admits an $H$-group structure such that $H^\bl(\Omega \Sigma \C \Pb^\infty, \Z)\cong \qsym$ and $H_\bl(\Omega \Sigma \C \Pb^\infty, \Z)\cong \nsym$ as Hopf algebras. Adhering to the principle that there are no coincidences in mathematics, we advance the following conjecture.
\begin{Cj}\label{cj:homotopic2james}
    $Q_{\infty}$ is isomorphic to $\Omega \Sigma \C \Pb^\infty$ as a group object in the weak homotopy category.
    In particular, $Q_{\infty}$ and $\Omega \Sigma \C \Pb^\infty$ are homotopy equivalent as topological spaces.\footnote{The first part implies that they are weakly homotopy equivalent. Since they both have the homotopy type of a CW complex, they are homotopy equivalent by Whitehead's theorem.}
\end{Cj}
\noindent
One can further motivate the conjecture as follows. In \cite{PS26}, the space $\Omega \Sigma \C \Pb^{\infty}$ is shown to be homotopy equivalent to a direct limit $J \C \Pb^{\infty}$ of spaces $J_m \Pb^n_{\C} \deq (\Pb^n_{\C})^m/\!\sim$, where $\sim$ is the equivalence relation that identifies the $m$ natural copies of $(\Pb^n_{\C})^{m-1}$ inside the product 
$(\Pb^n_{\C})^m$. The spaces $(\Pb^n_{\C})^m$ may be viewed as the snake den $\underbrace{\pr 0^{n-1} \rp \dots \pr 0^{n-1} \rp}_{\text{$m$ components}}$
so admit slither embeddings into $Q_{\infty}$. 
The various induced embeddings of $(\Pb^n_{\C})^{m-1}$ into $Q_{\infty}$ are all homotopic. So it is not unreasonable to expect that, after replacing the map $(\Pb^n_{\C})^m \to Q_{\infty}$ with a homotopic one, the various copies of $(\Pb^n_{\C})^{m-1}$ can be made to coincide, giving rise to a map
$J_m \Pb^n \to Q_{\infty}$.
One can then hope to be able to pass to the direct limit to get a map $J \C \Pb^{\infty} \to Q_{\infty}$. The construction would ensure that this map is compatible with the $H$-group structures and induces an isomorphism on homology, so is a homotopy equivalence by the dual Whitehead theorems.

\appendix
\section{Numerology}\label{sec:numerology}
\noindent
 Chase \cite{C76} proved two interesting results about the number of subwords of given length taken from a fixed word.

\begin{Thm}[{\cite{C76}*{Theorems 3.1 and 4.1(ii)}}]\label{thm:chase}
    Let $\sigma_i$ be the number of distinct subwords of length $i$ taken from a string $s$ of length $m$ in some finite alphabet. Then 
    \begin{enumerate}
        \item $\sigma_i^2 \ge \sigma_{i-1}\sigma_{i+1}$ for all $0 < i < m.$
        \item $\sigma_i\le \sigma_{m-i}$ whenever $0\le i\le m-i.$
     \end{enumerate}
\end{Thm}
\noindent
We recall the definition of the $f$- and $h$-vectors of a polytope.
If $\dim P = d$, the $f$-vector is the tuple $(f_0, \dots, f_d)$
in which $f_i$ is the number of faces of dimension $i$ of $P$. The $h$-vector is defined by the identity $\sum_{i=0}^dh_ix^i = \sum_{i=0}^d f_i(x-1)^i$. 
The following result is well-known to experts, but we were unable to find a direct reference.
\begin{Prop}\label[Prop]{prop:hvectorbetti}
    Let $P$ be a lattice polytope of dimension $d$ and suppose the toric variety $X_P$ admits an affine paving. Then the Betti vector $(b_0, \dots, b_d)$ of $X_P$ coincides with the $h$-vector of $P.$
    \begin{proof}
        Let $(f_0, \dots, f_d)$ be the $f$-vector of $P.$ Then the class of $X_P$ in the Grothendieck ring of varieties is $\sum_i f_i[\Gm]^i = \sum_i f_i\left([\A^1]-1\right)^i$ whereas the affine paving gives $[X_P] = \sum_i b_i[\A^1]^i$.
        The claim follows, since there are no linear relations between the powers of $[\A^1]$ in this ring.
    \end{proof}
\end{Prop}
\noindent
\begin{proof}[Proof of Theorem~\ref{thm:bettinumbersprop}]
Let $b_i$ be the number of cells of dimension $i$ in the affine paving of $X_s$. 
Claim $(5)$ then follows from $\# \A^n(\mathbb{F}_q) = q^n$. Claim $(4)$ follows from the fact that the cell closures in the affine paving form a basis for the homology of $X_s$. 
Claim (3) follows from \Cref{prop:hvectorbetti}.
Claims $(1)$ and $(2)$ follow from \Cref{df:base2word} and the results proven in \Cref{sec:bijection}. 
    The log-concavity property follows from \Cref{thm:chase}. The top-heavy claim is clear for $i = 0$ since $b_0= 1$. For $i \geq 1$, it is equivalent to the assertion $\sigma_a\le \sigma_b$ for $a\le b\le m-a$, where $m = d-1$ is the length of $s$. The latter follows from log-concavity and the inequality $\sigma_a\le \sigma_{m-a}.$
\end{proof}
\begin{Rmk}
    Let $X$ be a projective variety over an algebraically closed field 
    which admits an affine paving. Let $b_i$ be the number of cells of dimension $i$. Then, $b_i \leq b_j$ for all $i \leq j \leq \dim X -i$. This follows from \cite{purity_paper}*{Theorem 3.1}.\footnote{
    The proof in \cite{purity_paper} is stated under the assumption of an affine stratification, but the proof uses only that it is an affine paving.}
    This fact was used to great effect in \cite{purity_paper} and \cite{HW17}.
    Since the dimension of a cell in $X_s$ corresponding to a subword of $s$ is one \emph{more} than the length of the word,
    this method produces in our setting a slightly smaller set of top-heavy inequalities than provided by \Cref{thm:chase}.
\end{Rmk}
In general, if the $h$-vector is nonnegative and log-concave, then so is the $f$-vector \cite{B84}*{Corollary 8.4}. In our case, the $h$-vector is log-concave \emph{except} at $i = 1$. Proving \Cref{cor:flogconcave} will therefore require some additional maneuvering. 

Recall that the $f$-vector can be obtained from the $h$-vector as follows \cite{S85}*{p. 213}. 
\begin{figure}[h!]
    \centering
    {\includestandalone{figures/triangle}}
    \caption{Obtaining the $f$-vector $5,8,5,1$ from the $h$-vector $1,1,2,1.$}
    \label{fig:h2f}
\end{figure}
Draw a Pascal-like triangle with $d+2$ rows by placing the $h$-vector in reverse order on the left side followed by an additional $0$, and $1$'s on the right side. 
Fill in the remaining entries down to row $d+1$ (rows are $0$-indexed) so that each entry is the sum of the two entries immediately above it.
The $f$-vector can be read off from the bottom row, ignoring the initial $0$. To be more precise, 
we write $T(i, j)$ for the 
entry in position $j$ in row $i$ (positions are also $0$-indexed). 
Then $T$ is uniquely determined by 
$T(i, 0) = h_{d-i}$ for $0 \leq i \leq d$, 
$T(d+1, 0) = 0$, $T(i,i) = 1$, and $T(i,j) = T(i-1,j-1) + T(i-1, j)$ for $0< j < i$.
The $f$-vector is given by $f_i = T(d+1, i+1).$ See \Cref{fig:h2f} for an example.
\begin{Lm}
    $T(i, 1)^2\ge T(i-1, 1)T(i+1, 1).$ That is, the $T(-, 1)$ diagonal is log-concave.
    \begin{proof}
        The sequence $T(i, 1)$ for $1 \leq i \leq d$ is obtained from the log-concave sequence $T(i, 0)$
        for $0 \leq i \leq d-1$
        by taking the first $d$ terms of its convolution with the log-concave sequence $1, \dots, 1$, so is log-concave by  \cite{M69}*{Theorem 1}. We have $T(d+1,1) -T(d,1)=T(d,0) = b_0 = 1$ and $T(d,1) - T(d-1,1) = T(d-1, 0) = b_1 = 1$. 
        It follows from this that $T(d,1)^2 \ge T(d+1, 1)T(d-1,1).$
    \end{proof}
\end{Lm}
\begin{proof}[Proof of \Cref{cor:flogconcave}]
    The triangular array $(T(i, j))_{1\le j\le i\le d+1}$ is the convolution array in the sense of \cite{FM20} with initial side sequence given by $T(-, 1)$ and convolution multiplier sequence given by the all $1$'s sequence. 
    \cite{FM20}*{Proposition 3}
    therefore shows that each row, in particular row $d+1$, is log-concave.
\end{proof}
\begin{proof}[Proof of \Cref{cor:hvectordominate}]
    \Cref{thm:bettinumbersprop} shows that $h_i$ is the number of lattice paths in the skew diagram of $s$ with exactly $i$ loose steps weakly above the lowest vertical loose step. \cite{S10}*{Corollary 3.4} shows that $h_i'$ is the number of lattice paths with exactly $i$ vertical loose steps.
    Since every such step evidently lies weakly above the lowest vertical loose step, the claim follows.
\end{proof}

\section{Toric Positroids}\label{sec:toric_positroids}

\begin{Df}
    A \textbf{positroid cell} in $\Gr(r, n)$ is any isomorphic image of an open Richardson variety 
    under the projection map from the full flag variety to the Grassmannian. 
    A \textbf{positroid variety} in $\Gr(r, n)$ is the closure of a positroid cell.
\end{Df}
\begin{Rmk}
The projection from the full flag variety to the Grassmannian of any \emph{closed} Richardson variety is the closure of a unique positroid cell, hence a positroid variety.
\end{Rmk}
\noindent
It is a fact (cf.\ \cite{Knutson_Lam_Speyer_2013})
that the matroid realized by a general point in a positroid cell (equivalently, positroid variety) is a positroid in the sense of \Cref{df:pos} and that, moreover, this induces a bijection between  positroid cells (equivalently, positroid varieties) in $\Gr(r, n)$ and positroids of rank $r$ on $[n]$.
For such a positroid $M$ we write $\mathring{\Pi}_M$ for the associated positroid cell and $\Pi_M$ for the closure of this cell.

\begin{Lm}\label[Lm]{lm:max_tori}
    Let $X$ be a projective toric variety with dense torus $T$. Let $T'$ be a torus acting on $X$. Then, after possibly conjugating the $T$-action by an automorphism of $X$, the action of $T'$ factors through $T$. 
\end{Lm}
\begin{proof}
    Let $A = \Aut(X)$ be the automorphism group of $X$. By \cite{brion_aut}*{Theorem 1}, $A$ is a linear algebraic group. Moreover, $T$ is a maximal torus in $A$. To see this, suppose $T\subseteq S$ for some torus $S$ in $A$. Let $x$ be a point in the dense $T$-orbit of $X$. Let $s \in S$ be a point in the stabilizer of $x$. We claim $s$ stabilizes every point in $X$. This is clear for points in $S \cdot x$ since $S$ is commutative. Since $T\cdot x \subseteq S\cdot x$ is dense in $X$ the rest follows by continuity. We conclude that $x$ has trivial stabilizer in $S$. Thus, $\dim S  = \dim S \cdot x \leq \dim X = \dim T$. 

    The result now follows from the fact that every torus in a linear algebraic group is contained in a maximal torus, and that maximal tori in a linear algebraic group are all conjugate. 
\end{proof}
Recall that a matroid $N$ is a \textbf{minor}  of a matroid $M$ if $E(N) \subseteq E(M)$ and there exists $S \subseteq E(M) - E(N)$ such that 
$\mathcal{B}(N) = \{I \subseteq E(N) \mid I \sqcup S \in \mathcal{B}(M)\}$.
An argument similar to the following lemma appears as \cite{CanSaha2025}*{Theorem 3.1}.
\begin{Lm}\label[Lm]{lm:toric_implies_no_u24}
    Let $M$ be a positroid of rank $r$ on $[n]$. Assume $\Pi_M$ admits the structure of a toric variety. Then, $M$ has no minor isomorphic to $U_{2,4}$.
\end{Lm}
\begin{proof}
    Assume otherwise.
    By \cite{Oxley}*{Lemma 3.3.2}, we can find 
    disjoint $A , B \subseteq [n]$ satisfying:
    $A$ is independent of rank $r-2$, $|B| = 4$, and $A \cup C \in \mathcal{B}(M)$ for all $C \in \binom{B}{2}$. Define
    \[
        \Omega = \left\{ L \in \Gr(r, n) \mid \Bbbk ^A  \subseteq L \subseteq \Bbbk ^{A \cup B} \right\}. 
    \]
    By \cite{Knutson_Lam_Speyer_2013}*{Theorem 5.14}, the homogeneous ideal $I(\Pi_M)$ cutting out $\Pi_M$ is generated by the Plücker coordinates $\{p_I \mid I \notin \mathcal{B}(M)\}$. The ideal $I(\Omega)$ on the other hand is generated by the Plücker coordinates $\{p_I \mid \nexists C \in \binom{B}{2}, A \cup C = I\}$. Thus $I(\Pi_M) \subseteq I(\Omega)$ and so $\Omega \subseteq \Pi_M$.
    
    Let $X = \Gr(r, n)$ and let $T$ be its standard torus. Let $v$ be a cocharacter of $T$ inducing a total order on $[n]$ which places all of $A$ before all of $B$, and places $A \cup B$ before all other elements. Let $B = \{b_1, \dots, b_4\}$ where $b_1< b_2 < b_3 < b_4$ in the total order. Let $I = A \cup \{b_3, b_4\}$. Then, $\Omega$ coincides with the \BB cell closure $\ove{X_{\Bbbk^I}^{+}}$.

    Now, let $T'$ be the dense torus of $\Pi_M$. By \Cref{lm:max_tori}, after
    conjugating the $T'$-action by a suitable automorphism of $\Pi_M$, we may assume that the $T$-action on $\Pi_M$ factors through $T'$. Then $v$ defines a cocharacter of $T'$, so $\Omega$ is a $T'$-invariant subvariety of $\Pi_M$. This implies that $\Omega \cong \Gr(2,4)$ is toric. 
    But $\Gr(2,4)$ is \emph{not} toric, for instance because its cohomology is not generated in degree 1. 
\end{proof}

The following result is closely related to \cite{GKSB26}*{Theorem 3.15}.
\begin{Thm}\label[Thm]{thm:toric_pos}
    Assume $M$ is a positroid of rank $r$ on $[n]$. Let $T$ be the standard torus acting on $\Gr(r, n)$. The following are equivalent: 
    \begin{enumerate}
        \item $\mathring{\Pi}_M$ consists of a single $T$-orbit. 
        \item $\Pi_M$ admits the structure of a toric variety.
        \item $\Pi_M$ has a dense $T$-orbit.
        \item $M$ has no minor isomorphic to $U_{2,4}$. 
        \item $M$ is binary. 
        \item $M$ is regular. 
        \item $M$ is graphic.
        \item $M$ is a direct sum of series-parallel matroids (allowing loops and coloops). 
    \end{enumerate}
    \begin{proof}
        Note that (5) through (8) are equivalent by \cite{QR25}*{Corollary 4.1} and (4)$\Leftrightarrow$(5) is \cite{Oxley}*{Theorem 6.5.4}. 
        (1)$\Rightarrow$(3) is trivial.
        (3)$\Rightarrow$(2) is clear since $T$-orbit closures in $\Gr(r, n)$ are normal \cite{MichalekSturmfels2021}*{Section 13.2}. (2)$\Rightarrow$(4) is \Cref{lm:toric_implies_no_u24}. 
        We have (6)+(7)$\Rightarrow$(1) as follows.
        By \Cref{thm:unique_torus}, there is a unique $T$-orbit in $\Gr(r, n)$ whose points realize $M$. 
        By \cite{QR25}*{Theorem 4.3}(iii), 
        this $T$-orbit coincides with $\mathring{\Pi}_M$.
    \end{proof}
\end{Thm}

\begin{Cor}\label{cor:toric_richardsons_in_GR}
    Let $X$ be a Richardson variety in $\Gr(r, n)$. Let $T$ be the standard torus. The following are equivalent.  
    \begin{enumerate}
    \setcounter{enumi}{1}
        \item $X$ admits the structure of a toric variety. 
        \item $X$ has a dense $T$-orbit. 
        \item The skew diagram of $X$ contains no $2\times 2$ box. 
    \end{enumerate}
\end{Cor}
\begin{proof}
    Any Richardson variety in $\Gr(r, n)$ is a positroid variety. Criterion (Z) in the corollary statement is equivalent to criterion (Z) in \Cref{thm:toric_pos} 
    for $Z = 2,3,4$. This is clear for $Z = 2,3$.
    We argue that the same is true for $Z = 4$.
    If the skew diagram of $X$ contains a $2\times 2$ box then it has a $U_{2,4}$ minor by \cite{JM18}*{Lemma 4.1}.
    Conversely, if there are no $2\times 2$ boxes in the skew diagram, it must be a slither, so the corresponding matroid is regular by the discussion in \Cref{sec:t_orb}.
\end{proof}

\bibliography{bib}

\end{document}

%% file: bourbaki-warning.tex
\usepackage{manfnt,graphicx}

\newcommand{\warningsymbolheight}{3em}
\newcommand{\warningsymbol}{%
  \resizebox*{!}{\warningsymbolheight}{\normalfont\textdbend}%
}
\newsavebox{\warningiconbox}

\newenvironment{warning}{%
  \par\addvspace{.7\baselineskip}%
  \sbox{\warningiconbox}{\warningsymbol}%
  \noindent
  \begin{minipage}[c]{\wd\warningiconbox}%
    \usebox{\warningiconbox}%
  \end{minipage}%
  \hspace{.9em}%
  \begin{minipage}[c]{\dimexpr\linewidth-\wd\warningiconbox-.9em\relax}%
    \normalfont
    \setlength{\parindent}{0pt}%
    \setlength{\parskip}{.4\baselineskip}%
    \textbf{Warning.}\enspace\ignorespaces
}{%
  \end{minipage}%
  \par\addvspace{.7\baselineskip}%
  \ignorespacesafterend
}

%% file: figures/tuong_poset.tex
\begin{tikzpicture}[
    line cap=round,
    line join=round,
    old edge/.style={
      draw=black!27,
      line width=0.62pt,
      dash pattern=on 0.55pt off 1.15pt
    },
    kept edge/.style={draw=black!88, line width=0.95pt},
    cover/.style={
      draw=black!43,
      line width=0.68pt
    },
    vertex backing/.style={
      fill=white,
      minimum width=1.60cm,
      minimum height=1.78cm,
      inner sep=0pt
    },
    word label/.style={font=\large, inner sep=0pt}
  ]

  \newcommand{\edgea}[1]{\draw[#1] (0,0) -- (1,0);}
  \newcommand{\edgeb}[1]{\draw[#1] (0,0) -- (0,1);}
  \newcommand{\edgec}[1]{\draw[#1] (0,1) -- (1,1);}
  \newcommand{\edged}[1]{\draw[#1] (1,0) -- (1,1);}
  \newcommand{\edgee}[1]{\draw[#1] (1,0) -- (2,0);}
  \newcommand{\edgef}[1]{\draw[#1] (2,0) -- (2,1);}
  \newcommand{\edgeg}[1]{\draw[#1] (1,1) -- (2,1);}
  \newcommand{\edgeh}[1]{\draw[#1] (1,1) -- (1,2);}
  \newcommand{\edgei}[1]{\draw[#1] (1,2) -- (2,2);}
  \newcommand{\edgej}[1]{\draw[#1] (2,1) -- (2,2);}

  \newcommand{\fullskeleton}{%
    \edgea{old edge}\edgeb{old edge}\edgec{old edge}\edged{old edge}%
    \edgee{old edge}\edgef{old edge}\edgeg{old edge}\edgeh{old edge}%
    \edgei{old edge}\edgej{old edge}%
  }

  \newcommand{\rendervertex}[3]{%
    \node[vertex backing] at (#1) {};
    \begin{scope}[shift={(#1)}, x=0.48cm, y=0.48cm]
      \begin{scope}[shift={(-1,-0.12)}]
        \fullskeleton
        #3
      \end{scope}
    \end{scope}
    \node[word label] at ([yshift=-0.66cm]#1) {$#2$};
  }

  \coordinate (top) at (0,10.75);

  \coordinate (u1) at (-7.40,7.65); 
  \coordinate (u2) at (-3.70,7.65); 
  \coordinate (u3) at ( 0.00,7.65); 
  \coordinate (u4) at ( 3.70,7.65); 
  \coordinate (u5) at ( 7.40,7.65); 

  \coordinate (v7) at (-9.10,4.15); 
  \coordinate (v2) at (-6.50,4.15); 
  \coordinate (v3) at (-3.90,4.15); 
  \coordinate (v4) at (-1.30,4.15); 
  \coordinate (v5) at ( 1.30,4.15); 
  \coordinate (v6) at ( 3.90,4.15); 
  \coordinate (v8) at ( 6.50,4.15); 
  \coordinate (v1) at ( 9.10,4.15); 

  \coordinate (w1) at (-7.40,0.45); 
  \coordinate (w2) at (-3.70,0.45); 
  \coordinate (w3) at ( 0.00,0.45); 
  \coordinate (w4) at ( 3.70,0.45); 
  \coordinate (w5) at ( 7.40,0.45); 

  \foreach \coververtex in {u1,u2,u3,u4,u5}
    \draw[cover] ([yshift=-0.88cm]top) -- ([yshift=0.88cm]\coververtex);

  \draw[cover] ([yshift=-0.88cm]u1) -- ([yshift=0.88cm]v7);
  \draw[cover] ([yshift=-0.88cm]u1) -- ([yshift=0.88cm]v2);
  \draw[cover] ([yshift=-0.88cm]u1) -- ([yshift=0.88cm]v4);

  \draw[cover] ([yshift=-0.88cm]u2) -- ([yshift=0.88cm]v2);
  \draw[cover] ([yshift=-0.88cm]u2) -- ([yshift=0.88cm]v3);
  \draw[cover] ([yshift=-0.88cm]u2) -- ([yshift=0.88cm]v5);

  \draw[cover] ([yshift=-0.88cm]u3) -- ([yshift=0.88cm]v1);
  \draw[cover] ([yshift=-0.88cm]u3) -- ([yshift=0.88cm]v3);
  \draw[cover] ([yshift=-0.88cm]u3) -- ([yshift=0.88cm]v6);
  \draw[cover] ([yshift=-0.88cm]u3) -- ([yshift=0.88cm]v7);

  \draw[cover] ([yshift=-0.88cm]u4) -- ([yshift=0.88cm]v4);
  \draw[cover] ([yshift=-0.88cm]u4) -- ([yshift=0.88cm]v6);
  \draw[cover] ([yshift=-0.88cm]u4) -- ([yshift=0.88cm]v8);

  \draw[cover] ([yshift=-0.88cm]u5) -- ([yshift=0.88cm]v5);
  \draw[cover] ([yshift=-0.88cm]u5) -- ([yshift=0.88cm]v1);
  \draw[cover] ([yshift=-0.88cm]u5) -- ([yshift=0.88cm]v8);

  \draw[cover] ([yshift=-0.88cm]v1) -- ([yshift=0.88cm]w4);
  \draw[cover] ([yshift=-0.88cm]v1) -- ([yshift=0.88cm]w5);

  \draw[cover] ([yshift=-0.88cm]v2) -- ([yshift=0.88cm]w1);
  \draw[cover] ([yshift=-0.88cm]v2) -- ([yshift=0.88cm]w3);

  \draw[cover] ([yshift=-0.88cm]v3) -- ([yshift=0.88cm]w1);
  \draw[cover] ([yshift=-0.88cm]v3) -- ([yshift=0.88cm]w4);

  \draw[cover] ([yshift=-0.88cm]v4) -- ([yshift=0.88cm]w2);
  \draw[cover] ([yshift=-0.88cm]v4) -- ([yshift=0.88cm]w3);

  \draw[cover] ([yshift=-0.88cm]v5) -- ([yshift=0.88cm]w3);
  \draw[cover] ([yshift=-0.88cm]v5) -- ([yshift=0.88cm]w4);

  \draw[cover] ([yshift=-0.88cm]v6) -- ([yshift=0.88cm]w2);
  \draw[cover] ([yshift=-0.88cm]v6) -- ([yshift=0.88cm]w5);

  \draw[cover] ([yshift=-0.88cm]v7) -- ([yshift=0.88cm]w1);
  \draw[cover] ([yshift=-0.88cm]v7) -- ([yshift=0.88cm]w2);

  \draw[cover] ([yshift=-0.88cm]v8) -- ([yshift=0.88cm]w3);
  \draw[cover] ([yshift=-0.88cm]v8) -- ([yshift=0.88cm]w5);

  \rendervertex{top}{\pr 01\rp}{%
    \edgea{kept edge}\edgeb{kept edge}\edgec{kept edge}\edged{kept edge}%
    \edgee{kept edge}\edgef{kept edge}\edgeg{kept edge}\edgeh{kept edge}%
    \edgei{kept edge}\edgej{kept edge}%
  }

  \rendervertex{u1}{\pr 0\rp \llle}{%
    \edgea{kept edge}\edgeb{kept edge}\edgec{kept edge}\edged{kept edge}%
    \edgee{kept edge}\edgef{kept edge}\edgeg{kept edge}\edgej{kept edge}%
  }
  \rendervertex{u2}{\elll\pr 1\rp}{%
    \edgea{kept edge}\edged{kept edge}\edgee{kept edge}\edgef{kept edge}%
    \edgeg{kept edge}\edgeh{kept edge}\edgei{kept edge}\edgej{kept edge}%
  }
  \rendervertex{u3}{\pr\rp\pr\rp}{%
    \edgea{kept edge}\edgeb{kept edge}\edgec{kept edge}\edged{kept edge}%
    \edgeg{kept edge}\edgeh{kept edge}\edgei{kept edge}\edgej{kept edge}%
  }
  \rendervertex{u4}{\pr \elll 1\rp}{%
    \edgea{kept edge}\edgeb{kept edge}\edgec{kept edge}\edgee{kept edge}%
    \edgef{kept edge}\edgeg{kept edge}\edgeh{kept edge}\edgei{kept edge}%
    \edgej{kept edge}%
  }
  \rendervertex{u5}{\pr 0 \llle \rp}{%
    \edgea{kept edge}\edgeb{kept edge}\edgec{kept edge}\edged{kept edge}%
    \edgee{kept edge}\edgef{kept edge}\edgeh{kept edge}\edgei{kept edge}%
    \edgej{kept edge}%
  }

  \rendervertex{v1}{\pr\rp \llle \elll}{%
    \edgea{kept edge}\edgeb{kept edge}\edgec{kept edge}\edged{kept edge}%
    \edgeh{kept edge}\edgei{kept edge}%
  }
  \rendervertex{v2}{\elll \pr\rp \llle}{%
    \edgea{kept edge}\edged{kept edge}\edgee{kept edge}\edgef{kept edge}%
    \edgeg{kept edge}\edgej{kept edge}%
  }
  \rendervertex{v3}{\elll \llle \pr\rp}{%
    \edgea{kept edge}\edged{kept edge}\edgeg{kept edge}\edgeh{kept edge}%
    \edgei{kept edge}\edgej{kept edge}%
  }
  \rendervertex{v4}{\pr \elll \rp \llle}{%
    \edgea{kept edge}\edgeb{kept edge}\edgec{kept edge}\edgee{kept edge}%
    \edgef{kept edge}\edgeg{kept edge}\edgej{kept edge}%
  }
  \rendervertex{v5}{\elll \pr \llle \rp}{%
    \edgea{kept edge}\edged{kept edge}\edgee{kept edge}\edgef{kept edge}%
    \edgeh{kept edge}\edgei{kept edge}\edgej{kept edge}%
  }
  \rendervertex{v6}{\llle \elll \pr\rp}{%
    \edgeb{kept edge}\edgec{kept edge}\edgeg{kept edge}\edgeh{kept edge}%
    \edgei{kept edge}\edgej{kept edge}%
  }
  \rendervertex{v7}{\pr\rp \elll \llle}{%
    \edgea{kept edge}\edgeb{kept edge}\edgec{kept edge}\edged{kept edge}%
    \edgeg{kept edge}\edgej{kept edge}%
  }
  \rendervertex{v8}{\pr \elll \llle \rp}{%
    \edgea{kept edge}\edgeb{kept edge}\edgec{kept edge}\edgee{kept edge}%
    \edgef{kept edge}\edgeh{kept edge}\edgei{kept edge}\edgej{kept edge}%
  }

  \rendervertex{w1}{\elll\llle\elll\llle}{%
    \edgea{kept edge}\edged{kept edge}\edgeg{kept edge}\edgej{kept edge}%
  }
  \rendervertex{w2}{\llle \elll\elll\llle}{%
    \edgeb{kept edge}\edgec{kept edge}\edgeg{kept edge}\edgej{kept edge}%
  }
  \rendervertex{w3}{\elll \elll\llle \llle}{%
    \edgea{kept edge}\edgee{kept edge}\edgef{kept edge}\edgej{kept edge}%
  }
  \rendervertex{w4}{\elll\llle\llle \elll}{%
    \edgea{kept edge}\edged{kept edge}\edgeh{kept edge}\edgei{kept edge}%
  }
  \rendervertex{w5}{\llle \elll\llle \elll}{%
    \edgeb{kept edge}\edgec{kept edge}\edgeh{kept edge}\edgei{kept edge}%
  }

\end{tikzpicture}

%% file: figures/skew_diagrams.tex
\begin{tikzpicture}[
    x=0.62cm,
    y=0.62cm,
    cell fill/.style={fill=black!38},
    box edge/.style={
      draw=black!85,
      line width=0.7pt,
      line cap=butt,
      line join=miter
    },
    boundary path/.style={
      draw=black!85,
      line width=0.7pt,
      line cap=butt,
      line join=miter
    },
    classification label/.style={font=\normalsize},
    interval label/.style={font=\large}
  ]

  \newcommand{\cellat}[2]{%
    \path[cell fill] (#1,#2) rectangle ++(1,1);%
    \draw[box edge] (#1,#2) rectangle ++(1,1);%
  }

  \begin{scope}
    \foreach \x/\y in {0/0,2/2,3/2}{\cellat{\x}{\y}}

    \draw[boundary path]
      (0,0) -- (1,0) -- (1,2) -- (4,2) -- (4,3);
    \draw[boundary path]
      (0,0) -- (0,1) -- (1,1) -- (1,2) --
      (2,2) -- (2,3) -- (4,3);

    \node[classification label] at (2,3.75)
      {non-snake den toric Richardson};
    \node[interval label] at (2,-1.05)
      {$\{1,3,5\}\leq_{\mathrm G}\{2,3,7\}$};
  \end{scope}

  \begin{scope}[xshift=5.3cm]
    \foreach \x/\y in {
      0/0,
      0/1,1/1,2/1,
      0/2,1/2,2/2,3/2
    }{\cellat{\x}{\y}}
    \node[classification label] at (2,3.75)
      {non-toric Richardson};
    \node[interval label] at (2,-1.05)
      {$\{1,2,3\}\leq_{\mathrm G}\{2,5,7\}$};
  \end{scope}

  \begin{scope}[xshift=10.6cm]
    \foreach \x/\y in {
      0/0,1/0,
      1/1,2/1,
      2/2
    }{\cellat{\x}{\y}}

    \draw[boundary path]
      (0,0) -- (2,0) -- (2,1) -- (3,1) -- (3,3) -- (4,3);
    \draw[boundary path]
      (0,0) -- (0,1) -- (1,1) -- (1,2) --
      (2,2) -- (2,3) -- (4,3);

    \node[classification label] at (2,3.75)
      {non-snake den toric Richardson};
    \node[interval label] at (2,-1.05)
      {$\{1,3,5\}\leq_{\mathrm G}\{3,5,6\}$};
  \end{scope}

  \begin{scope}[xshift=2.65cm,yshift=-3.6cm]
    \foreach \x/\y in {
      0/0,1/0,
      1/1,2/1,
      2/2,3/2
    }{\cellat{\x}{\y}}

    \draw[boundary path]
      (0,0) -- (2,0) -- (2,1) -- (3,1) --
      (3,2) -- (4,2) -- (4,3);
    \draw[boundary path]
      (0,0) -- (0,1) -- (1,1) -- (1,2) --
      (2,2) -- (2,3) -- (4,3);

    \node[classification label] at (2,3.75)
      {snake toric Richardson};
    \node[interval label] at (2,-1.05)
      {$\{1,3,5\}\leq_{\mathrm G}\{3,5,7\}$};
  \end{scope}

  \begin{scope}[xshift=7.95cm,yshift=-3.6cm]
    \foreach \x/\y in {
      0/0,1/0,
      1/1,
      2/2,3/2
    }{\cellat{\x}{\y}}

    \draw[boundary path]
      (0,0) -- (2,0) -- (2,2) -- (4,2) -- (4,3);
    \draw[boundary path]
      (0,0) -- (0,1) -- (1,1) -- (1,2) --
      (2,2) -- (2,3) -- (4,3);

    \node[classification label] at (2,3.75)
      {snake den toric Richardson};
    \node[interval label] at (2,-1.05)
      {$\{1,3,5\}\leq_{\mathrm G}\{3,4,7\}$};
  \end{scope}

\end{tikzpicture}

%% file: figures/snakes_and_dens.tex
\begin{tikzpicture}[
    x=0.62cm,
    y=0.62cm,
    line cap=round,
    line join=round,
    cell fill/.style={fill=black!38},
    box edge/.style={
      draw=black!85,
      line width=0.7pt,
      line cap=butt,
      line join=miter
    },
    traversal/.style={
      draw=pathblue,
      line width=1.05pt,
      -{Stealth[length=3.2pt,width=3.6pt]}
    },
    endpoint/.style={circle, fill=black!88, inner sep=1.35pt},
    pinch point/.style={circle, fill=black!88, inner sep=1.55pt},
    bit label/.style={
      circle,
      fill=white,
      text=pathblue,
      font=\scriptsize\bfseries,
      inner sep=0.8pt
    },
    panel label/.style={font=\normalsize, align=center}
  ]

  \definecolor{pathblue}{RGB}{0,114,178}
  \providecommand{\pr}{\mathopen{(}}
  \providecommand{\rp}{\mathclose{)}}

  \newcommand{\fillcell}[2]{%
    \path[cell fill] (#1,#2) rectangle ++(1,1);%
  }

  \begin{scope}[xshift=5.15cm]
    \foreach \x/\y in {0/0,1/0,1/1,2/2,2/3,3/3}{%
      \fillcell{\x}{\y}%
    }

    \draw[box edge]
      (0,0) -- (2,0) -- (2,2) -- (1,2) --
      (1,1) -- (0,1) -- cycle;
    \draw[box edge] (1,0) -- (1,1);
    \draw[box edge] (1,1) -- (2,1);

    \draw[box edge]
      (2,2) -- (3,2) -- (3,3) -- (4,3) --
      (4,4) -- (2,4) -- cycle;
    \draw[box edge] (2,3) -- (3,3);
    \draw[box edge] (3,3) -- (3,4);

    \node[endpoint] at (0,0) {};
    \node[pinch point] at (2,2) {};
    \node[endpoint] at (4,4) {};

    \node[panel label] at (2,-0.82)
      {$\text{skew diagram for }\pr 01\rp\,\pr 10\rp$};
  \end{scope}

  \begin{scope}
    \foreach \x/\y in {0/0,1/0,1/1,2/1,3/1,3/2,4/2,4/3}{%
      \fillcell{\x}{\y}%
      \draw[box edge] (\x,\y) rectangle ++(1,1);%
    }

    \draw[traversal] (0.50,0.50) -- (1.50,0.50);
    \draw[traversal] (1.50,0.50) -- (1.50,1.50);
    \draw[traversal] (1.50,1.50) -- (2.50,1.50);
    \draw[traversal] (2.50,1.50) -- (3.50,1.50);
    \draw[traversal] (3.50,1.50) -- (3.50,2.50);
    \draw[traversal] (3.50,2.50) -- (4.50,2.50);
    \draw[traversal] (4.50,2.50) -- (4.50,3.50);

    \node[bit label] at (1.00,0.50) {$0$};
    \node[bit label] at (1.50,1.00) {$1$};
    \node[bit label] at (2.00,1.50) {$0$};
    \node[bit label] at (3.00,1.50) {$0$};
    \node[bit label] at (3.50,2.00) {$1$};
    \node[bit label] at (4.00,2.50) {$0$};
    \node[bit label] at (4.50,3.00) {$1$};

    \node[endpoint] at (0.50,0.50) {};
    \node[endpoint] at (4.50,3.50) {};
    \node[panel label] at (2.5,-0.82)
      {$\text{skew diagram for }\pr 0100101\rp$};
  \end{scope}

\end{tikzpicture}

%% file: figures/slither_word_diagram.tex
\begin{tikzpicture}[
    x=1.28cm,
    y=1.28cm,
    line cap=round,
    line join=round,
    old edge/.style={
      draw=black!25,
      line width=0.68pt,
      dash pattern=on 0.7pt off 1.35pt
    },
    kept edge/.style={draw=black!88, line width=1.25pt},
    antidiagonal/.style={
      draw=black!42,
      line width=0.72pt,
      dash pattern=on 2.3pt off 1.8pt
    },
    crossing/.style={
      circle,
      fill=black!62,
      draw=white,
      line width=0.45pt,
      inner sep=1.55pt
    },
    cut number/.style={
      circle,
      draw=black!52,
      fill=white,
      text=black!72,
      line width=0.65pt,
      font=\scriptsize,
      inner sep=1.7pt
    }
  ]

  \newcommand{\edgea}[1]{\draw[#1] (0,0) -- (1,0);}
  \newcommand{\edgeb}[1]{\draw[#1] (0,0) -- (0,1);}
  \newcommand{\edgec}[1]{\draw[#1] (0,1) -- (1,1);}
  \newcommand{\edged}[1]{\draw[#1] (1,0) -- (1,1);}
  \newcommand{\edgee}[1]{\draw[#1] (1,0) -- (2,0);}
  \newcommand{\edgef}[1]{\draw[#1] (2,0) -- (2,1);}
  \newcommand{\edgeg}[1]{\draw[#1] (1,1) -- (2,1);}
  \newcommand{\edgeh}[1]{\draw[#1] (1,1) -- (1,2);}
  \newcommand{\edgei}[1]{\draw[#1] (1,2) -- (2,2);}
  \newcommand{\edgej}[1]{\draw[#1] (2,1) -- (2,2);}

  \draw[antidiagonal] (-0.38,0.88) -- (0.88,-0.38);
  \draw[antidiagonal] (-0.38,1.88) -- (1.88,-0.38);
  \draw[antidiagonal] ( 0.12,2.38) -- (2.38, 0.12);
  \draw[antidiagonal] ( 1.12,2.38) -- (2.38, 1.12);

  \node[cut number] at (-0.38,0.88) {$1$};
  \node[cut number] at (-0.38,1.88) {$2$};
  \node[cut number] at ( 0.12,2.38) {$3$};
  \node[cut number] at ( 1.12,2.38) {$4$};

  \edgea{old edge}\edgeb{old edge}\edgec{old edge}\edged{old edge}%
  \edgee{old edge}\edgef{old edge}\edgeg{old edge}\edgeh{old edge}%
  \edgei{old edge}\edgej{old edge}%

  \draw[old edge] (0,1) -- (0,2) -- (1,2);

  \edgea{kept edge}\edgeb{kept edge}\edgec{kept edge}%
  \edgee{kept edge}\edgef{kept edge}\edgeg{kept edge}\edgej{kept edge}%

  \node[crossing] at (0,0.5) {};
  \node[crossing] at (0.5,0) {};
  \node[crossing] at (0.5,1) {};
  \node[crossing] at (1.5,0) {};
  \node[crossing] at (1.5,1) {};
  \node[crossing] at (2,0.5) {};
  \node[crossing] at (2,1.5) {};

  \node[anchor=north] at (1,-0.68) {$
    \begin{array}{c|c|c|c}
      k=1 & k=2 & k=3 & k=4 \\
      \hline
      \pr & \elll & \rp & \llle
    \end{array}
  $};

\end{tikzpicture}

%% file: figures/hydra.tex
\begin{tikzpicture}[
  line cap=round,line join=round,
  kept/.style={draw=black!88,line width=1.25pt},
  absent/.style={draw=black!25,line width=.68pt,
    dash pattern=on .7pt off 1.35pt},
  cut/.style={draw=black!42,line width=.72pt,
    dash pattern=on 2.3pt off 1.8pt},
  seam/.style={circle,draw=black!52,fill=white,
    inner sep=1.7pt,line width=.65pt},
  crossing/.style={circle,fill=black!62,draw=white,
    line width=.45pt,inner sep=1.55pt},
  move/.style={->,draw=black!42,line width=.72pt}
]
\begin{scope}[yshift=0cm]
\begin{scope}[shift={(0,0)},x=.5cm,y=.5cm]
\draw[cut] (.15,3.35) -- (2.85,.65);
\draw[kept] (0,0) -- (1,0);
\draw[kept] (0,0) -- (0,1);
\draw[kept] (0,1) -- (1,1);
\draw[kept] (1,0) -- (1,1);
\draw[kept] (1,0) -- (2,0);
\draw[kept] (1,1) -- (2,1);
\draw[kept] (2,0) -- (2,1);
\draw[kept] (1,1) -- (1,2);
\draw[absent] (1,2) -- (2,2);
\draw[kept] (2,1) -- (2,2);
\draw[kept] (1,2) -- (1,3);
\draw[kept] (1,3) -- (2,3);
\draw[kept] (2,2) -- (2,3);
\draw[kept] (2,2) -- (3,2);
\draw[kept] (2,3) -- (3,3);
\draw[kept] (3,2) -- (3,3);
\draw[kept] (3,2) -- (4,2);
\draw[absent] (3,3) -- (4,3);
\draw[kept] (4,2) -- (4,3);
\draw[kept] (3,3) -- (3,4);
\draw[kept] (3,4) -- (4,4);
\draw[kept] (4,3) -- (4,4);
\end{scope}
\begin{scope}[shift={(3.2,0)},x=.5cm,y=.5cm]
\draw[cut] (.15,3.35) -- (2.85,.65);
\draw[kept] (0,0) -- (1,0);
\draw[kept] (0,0) -- (0,1);
\draw[kept] (0,1) -- (1,1);
\draw[kept] (1,0) -- (1,1);
\draw[kept] (1,0) -- (2,0);
\draw[kept] (1,1) -- (2,1);
\draw[kept] (2,0) -- (2,1);
\draw[kept] (1,1) -- (1,2);
\draw[kept] (1,3) -- (2,3);
\draw[kept] (2,2) -- (2,3);
\draw[kept] (2,2) -- (3,2);
\draw[kept] (2,3) -- (3,3);
\draw[kept] (3,2) -- (3,3);
\draw[kept] (3,2) -- (4,2);
\draw[absent] (3,3) -- (4,3);
\draw[kept] (4,2) -- (4,3);
\draw[kept] (3,3) -- (3,4);
\draw[kept] (3,4) -- (4,4);
\draw[kept] (4,3) -- (4,4);
\node[seam] at (1,2) {};
\node[seam] at (2,2) {};
\node[seam] at (1,3) {};
\node[seam] at (2,1) {};
\end{scope}
\begin{scope}[shift={(6.4,0)},x=.5cm,y=.5cm]
\draw[kept] (0,0) -- (1,0);
\draw[kept] (0,0) -- (0,1);
\draw[kept] (0,1) -- (1,1);
\draw[kept] (1,0) -- (1,1);
\draw[kept] (1,0) -- (2,0);
\draw[kept] (1,1) -- (2,1);
\draw[kept] (2,0) -- (2,1);
\draw[kept] (1,1) -- (1,2);
\draw[kept] (1,2) -- (2,2);
\draw[kept] (2,1) -- (2,2);
\draw[kept] (2,1) -- (3,1);
\draw[kept] (2,2) -- (3,2);
\draw[kept] (3,1) -- (3,2);
\draw[kept] (3,1) -- (4,1);
\draw[absent] (3,2) -- (4,2);
\draw[kept] (4,1) -- (4,2);
\draw[kept] (3,2) -- (3,3);
\draw[kept] (3,3) -- (4,3);
\draw[kept] (4,2) -- (4,3);
\node[crossing] at (1,2) {};
\node[crossing] at (2,1) {};
\end{scope}
\draw[move] (2.25,1) -- (2.95,1);
\draw[move] (5.45,1) -- (6.15,1);
\end{scope}
\begin{scope}[yshift=-3cm]
\begin{scope}[shift={(0,0)},x=.5cm,y=.5cm]
\draw[cut] (.15,3.35) -- (2.85,.65);
\draw[kept] (0,0) -- (0,1);
\draw[kept] (0,0) -- (1,0);
\draw[kept] (1,0) -- (1,1);
\draw[kept] (0,1) -- (1,1);
\draw[kept] (0,1) -- (0,2);
\draw[kept] (1,1) -- (1,2);
\draw[kept] (0,2) -- (1,2);
\draw[kept] (1,1) -- (2,1);
\draw[absent] (2,1) -- (2,2);
\draw[kept] (1,2) -- (2,2);
\draw[kept] (2,1) -- (3,1);
\draw[kept] (3,1) -- (3,2);
\draw[kept] (2,2) -- (3,2);
\draw[kept] (2,2) -- (2,3);
\draw[kept] (3,2) -- (3,3);
\draw[kept] (2,3) -- (3,3);
\draw[kept] (2,3) -- (2,4);
\draw[absent] (3,3) -- (3,4);
\draw[kept] (2,4) -- (3,4);
\draw[kept] (3,3) -- (4,3);
\draw[kept] (4,3) -- (4,4);
\draw[kept] (3,4) -- (4,4);
\end{scope}
\begin{scope}[shift={(3.2,0)},x=.5cm,y=.5cm]
\draw[cut] (.15,3.35) -- (2.85,.65);
\draw[kept] (0,0) -- (0,1);
\draw[kept] (0,0) -- (1,0);
\draw[kept] (1,0) -- (1,1);
\draw[kept] (0,1) -- (1,1);
\draw[kept] (0,1) -- (0,2);
\draw[kept] (1,1) -- (1,2);
\draw[kept] (0,2) -- (1,2);
\draw[kept] (1,1) -- (2,1);
\draw[kept] (3,1) -- (3,2);
\draw[kept] (2,2) -- (3,2);
\draw[kept] (2,2) -- (2,3);
\draw[kept] (3,2) -- (3,3);
\draw[kept] (2,3) -- (3,3);
\draw[kept] (2,3) -- (2,4);
\draw[absent] (3,3) -- (3,4);
\draw[kept] (2,4) -- (3,4);
\draw[kept] (3,3) -- (4,3);
\draw[kept] (4,3) -- (4,4);
\draw[kept] (3,4) -- (4,4);
\node[seam] at (2,1) {};
\node[seam] at (2,2) {};
\node[seam] at (3,1) {};
\node[seam] at (1,2) {};
\end{scope}
\begin{scope}[shift={(6.4,0)},x=.5cm,y=.5cm]
\draw[kept] (0,0) -- (0,1);
\draw[kept] (0,0) -- (1,0);
\draw[kept] (1,0) -- (1,1);
\draw[kept] (0,1) -- (1,1);
\draw[kept] (0,1) -- (0,2);
\draw[kept] (1,1) -- (1,2);
\draw[kept] (0,2) -- (1,2);
\draw[kept] (1,1) -- (2,1);
\draw[kept] (2,1) -- (2,2);
\draw[kept] (1,2) -- (2,2);
\draw[kept] (1,2) -- (1,3);
\draw[kept] (2,2) -- (2,3);
\draw[kept] (1,3) -- (2,3);
\draw[kept] (1,3) -- (1,4);
\draw[absent] (2,3) -- (2,4);
\draw[kept] (1,4) -- (2,4);
\draw[kept] (2,3) -- (3,3);
\draw[kept] (3,3) -- (3,4);
\draw[kept] (2,4) -- (3,4);
\node[crossing] at (2,1) {};
\node[crossing] at (1,2) {};
\end{scope}
\draw[move] (2.25,1) -- (2.95,1);
\draw[move] (5.45,1) -- (6.15,1);
\end{scope}
\end{tikzpicture}

%% file: figures/pruning_slither.tex
\begin{tikzpicture}[
  line cap=round,line join=round,
  kept/.style={draw=black!88,line width=1.25pt},
  absent/.style={draw=black!25,line width=.68pt,
    dash pattern=on .7pt off 1.35pt},
  move/.style={-{Stealth[length=5pt]},draw=black!42,line width=.72pt}
]
\begin{scope}[x=.4375cm,y=.4375cm]
  \draw[kept] (0,0) rectangle (1,1);

  \draw[kept] (1,1) -- (2,1) -- (2,2) -- (4,2)
    -- (4,3) -- (5,3) -- (5,5) -- (4,5)
    -- (4,4) -- (3,4) -- (3,3) -- (1,3) -- cycle;
  \draw[kept] (3,3) -- (4,3);
  \draw[absent] (1,2) -- (2,2) -- (2,3);
  \draw[absent] (3,2) -- (3,3);
  \draw[absent] (4,3) -- (4,4) -- (5,4);

  \draw[kept] (5,5) -- (6,5) -- (6,6);

  \draw[kept] (6,6) -- (10,6) -- (10,7) -- (11,7)
    -- (11,8) -- (9,8) -- (9,7) -- (6,7) -- cycle;
  \draw[kept] (7,6) -- (7,7);
  \draw[kept] (9,7) -- (10,7);
  \draw[absent] (8,6) -- (8,7);
  \draw[absent] (9,6) -- (9,7);
  \draw[absent] (10,7) -- (10,8);

  \draw[kept] (11,8) -- (13,8);
\end{scope}

\draw[move] (6.1,1.75) -- (7.1,1.75);

\begin{scope}[shift={(7.65,.65625)},x=.4375cm,y=.4375cm]
  \draw[kept] (0,0) rectangle (1,1);
  \draw[kept] (1,1) rectangle (2,3);
  \draw[kept] (1,2) -- (2,2);
  \draw[kept] (2,3) -- (4,3) -- (4,5) -- (3,5)
    -- (3,4) -- (2,4) -- cycle;
  \draw[kept] (3,3) -- (3,4) -- (4,4);
\end{scope}
\end{tikzpicture}

%% file: figures/local_moves.tex
\begin{tikzpicture}[
  x=1cm,y=1cm,line cap=round,line join=round,
  old edge/.style={draw=black!25,line width=.68pt,
    dash pattern=on .7pt off 1.35pt},
  kept edge/.style={draw=black!88,line width=1.25pt},
  table rule/.style={draw=black!30,line width=.45pt}
]
\draw[table rule] (2,0) rectangle (14.5,-4.4);
\foreach \y in {-.7,-2.2,-2.9}
  \draw[table rule] (2,\y) -- (14.5,\y);
\foreach \j in {1,...,9}
  \draw[table rule] ({2+1.25*\j},0) -- ({2+1.25*\j},-4.4);
\node at (2.625,-0.35) {$0$};
\begin{scope}[shift={(2.175,-1.675)},x=.45cm,y=.45cm]
\draw[kept edge] (0,0) -- (1,0);
\draw[kept edge] (0,0) -- (0,1);
\draw[kept edge] (0,1) -- (1,1);
\draw[kept edge] (1,0) -- (1,1);
\draw[kept edge] (1,0) -- (2,0);
\draw[kept edge] (1,1) -- (2,1);
\draw[kept edge] (2,0) -- (2,1);
\end{scope}
\node at (2.625,-2.55) {$\elll$};
\begin{scope}[shift={(2.175,-3.875)},x=.45cm,y=.45cm]
\draw[kept edge] (0,0) -- (1,0);
\draw[kept edge] (0,0) -- (0,1);
\draw[kept edge] (0,1) -- (1,1);
\draw[old edge] (1,0) -- (1,1);
\draw[kept edge] (1,0) -- (2,0);
\draw[kept edge] (1,1) -- (2,1);
\draw[kept edge] (2,0) -- (2,1);
\end{scope}
\node at (3.875,-0.35) {$1$};
\begin{scope}[shift={(3.65,-1.9)},x=.45cm,y=.45cm]
\draw[kept edge] (0,0) -- (1,0);
\draw[kept edge] (0,0) -- (0,1);
\draw[kept edge] (0,1) -- (1,1);
\draw[kept edge] (1,0) -- (1,1);
\draw[kept edge] (0,1) -- (0,2);
\draw[kept edge] (0,2) -- (1,2);
\draw[kept edge] (1,1) -- (1,2);
\end{scope}
\node at (3.875,-2.55) {$\llle$};
\begin{scope}[shift={(3.65,-4.1)},x=.45cm,y=.45cm]
\draw[kept edge] (0,0) -- (1,0);
\draw[kept edge] (0,0) -- (0,1);
\draw[old edge] (0,1) -- (1,1);
\draw[kept edge] (1,0) -- (1,1);
\draw[kept edge] (0,1) -- (0,2);
\draw[kept edge] (0,2) -- (1,2);
\draw[kept edge] (1,1) -- (1,2);
\end{scope}
\node at (5.125,-0.35) {$\pr 0$};
\begin{scope}[shift={(4.675,-1.675)},x=.45cm,y=.45cm]
\draw[kept edge] (0,0) -- (1,0);
\draw[kept edge] (0,0) -- (0,1);
\draw[kept edge] (0,1) -- (1,1);
\draw[kept edge] (1,0) -- (1,1);
\draw[kept edge] (1,0) -- (2,0);
\draw[kept edge] (1,1) -- (2,1);
\draw[kept edge] (2,0) -- (2,1);
\end{scope}
\node at (5.125,-2.55) {$\elll\pr$};
\begin{scope}[shift={(4.675,-3.875)},x=.45cm,y=.45cm]
\draw[kept edge] (0,0) -- (1,0);
\draw[old edge] (0,0) -- (0,1);
\draw[old edge] (0,1) -- (1,1);
\draw[kept edge] (1,0) -- (1,1);
\draw[kept edge] (1,0) -- (2,0);
\draw[kept edge] (1,1) -- (2,1);
\draw[kept edge] (2,0) -- (2,1);
\end{scope}
\node at (6.375,-0.35) {$\pr 1$};
\begin{scope}[shift={(6.15,-1.9)},x=.45cm,y=.45cm]
\draw[kept edge] (0,0) -- (1,0);
\draw[kept edge] (0,0) -- (0,1);
\draw[kept edge] (0,1) -- (1,1);
\draw[kept edge] (1,0) -- (1,1);
\draw[kept edge] (0,1) -- (0,2);
\draw[kept edge] (0,2) -- (1,2);
\draw[kept edge] (1,1) -- (1,2);
\end{scope}
\node at (6.375,-2.55) {$\llle\pr$};
\begin{scope}[shift={(6.15,-4.1)},x=.45cm,y=.45cm]
\draw[old edge] (0,0) -- (1,0);
\draw[kept edge] (0,0) -- (0,1);
\draw[kept edge] (0,1) -- (1,1);
\draw[old edge] (1,0) -- (1,1);
\draw[kept edge] (0,1) -- (0,2);
\draw[kept edge] (0,2) -- (1,2);
\draw[kept edge] (1,1) -- (1,2);
\end{scope}
\node at (7.625,-0.35) {$0\rp$};
\begin{scope}[shift={(7.175,-1.675)},x=.45cm,y=.45cm]
\draw[kept edge] (0,0) -- (1,0);
\draw[kept edge] (0,0) -- (0,1);
\draw[kept edge] (0,1) -- (1,1);
\draw[kept edge] (1,0) -- (1,1);
\draw[kept edge] (1,0) -- (2,0);
\draw[kept edge] (1,1) -- (2,1);
\draw[kept edge] (2,0) -- (2,1);
\end{scope}
\node at (7.625,-2.55) {$\rp\elll$};
\begin{scope}[shift={(7.175,-3.875)},x=.45cm,y=.45cm]
\draw[kept edge] (0,0) -- (1,0);
\draw[kept edge] (0,0) -- (0,1);
\draw[kept edge] (0,1) -- (1,1);
\draw[kept edge] (1,0) -- (1,1);
\draw[old edge] (1,0) -- (2,0);
\draw[kept edge] (1,1) -- (2,1);
\draw[old edge] (2,0) -- (2,1);
\end{scope}
\node at (8.875,-0.35) {$1\rp$};
\begin{scope}[shift={(8.65,-1.9)},x=.45cm,y=.45cm]
\draw[kept edge] (0,0) -- (1,0);
\draw[kept edge] (0,0) -- (0,1);
\draw[kept edge] (0,1) -- (1,1);
\draw[kept edge] (1,0) -- (1,1);
\draw[kept edge] (0,1) -- (0,2);
\draw[kept edge] (0,2) -- (1,2);
\draw[kept edge] (1,1) -- (1,2);
\end{scope}
\node at (8.875,-2.55) {$\rp\llle$};
\begin{scope}[shift={(8.65,-4.1)},x=.45cm,y=.45cm]
\draw[kept edge] (0,0) -- (1,0);
\draw[kept edge] (0,0) -- (0,1);
\draw[kept edge] (0,1) -- (1,1);
\draw[kept edge] (1,0) -- (1,1);
\draw[old edge] (0,1) -- (0,2);
\draw[old edge] (0,2) -- (1,2);
\draw[kept edge] (1,1) -- (1,2);
\end{scope}
\node at (10.125,-0.35) {$01$};
\begin{scope}[shift={(9.675,-1.9)},x=.45cm,y=.45cm]
\draw[kept edge] (0,0) -- (1,0);
\draw[kept edge] (0,0) -- (0,1);
\draw[kept edge] (0,1) -- (1,1);
\draw[kept edge] (1,0) -- (1,1);
\draw[kept edge] (1,0) -- (2,0);
\draw[kept edge] (1,1) -- (2,1);
\draw[kept edge] (2,0) -- (2,1);
\draw[kept edge] (1,1) -- (1,2);
\draw[kept edge] (1,2) -- (2,2);
\draw[kept edge] (2,1) -- (2,2);
\end{scope}
\node at (10.125,-2.55) {$\rp\pr$};
\begin{scope}[shift={(9.675,-4.1)},x=.45cm,y=.45cm]
\draw[kept edge] (0,0) -- (1,0);
\draw[kept edge] (0,0) -- (0,1);
\draw[kept edge] (0,1) -- (1,1);
\draw[kept edge] (1,0) -- (1,1);
\draw[old edge] (1,0) -- (2,0);
\draw[kept edge] (1,1) -- (2,1);
\draw[old edge] (2,0) -- (2,1);
\draw[kept edge] (1,1) -- (1,2);
\draw[kept edge] (1,2) -- (2,2);
\draw[kept edge] (2,1) -- (2,2);
\end{scope}
\node at (11.375,-0.35) {$10$};
\begin{scope}[shift={(10.925,-1.9)},x=.45cm,y=.45cm]
\draw[kept edge] (0,0) -- (1,0);
\draw[kept edge] (0,0) -- (0,1);
\draw[kept edge] (0,1) -- (1,1);
\draw[kept edge] (1,0) -- (1,1);
\draw[kept edge] (0,1) -- (0,2);
\draw[kept edge] (0,2) -- (1,2);
\draw[kept edge] (1,1) -- (1,2);
\draw[kept edge] (1,1) -- (2,1);
\draw[kept edge] (1,2) -- (2,2);
\draw[kept edge] (2,1) -- (2,2);
\end{scope}
\node at (11.375,-2.55) {$\rp\pr$};
\begin{scope}[shift={(10.925,-4.1)},x=.45cm,y=.45cm]
\draw[kept edge] (0,0) -- (1,0);
\draw[kept edge] (0,0) -- (0,1);
\draw[kept edge] (0,1) -- (1,1);
\draw[kept edge] (1,0) -- (1,1);
\draw[old edge] (0,1) -- (0,2);
\draw[old edge] (0,2) -- (1,2);
\draw[kept edge] (1,1) -- (1,2);
\draw[kept edge] (1,1) -- (2,1);
\draw[kept edge] (1,2) -- (2,2);
\draw[kept edge] (2,1) -- (2,2);
\end{scope}
\node at (12.625,-0.35) {$\pr\rp$};
\begin{scope}[shift={(12.4,-1.675)},x=.45cm,y=.45cm]
\draw[kept edge] (0,0) -- (1,0);
\draw[kept edge] (0,0) -- (0,1);
\draw[kept edge] (0,1) -- (1,1);
\draw[kept edge] (1,0) -- (1,1);
\end{scope}
\node at (12.625,-2.55) {$\llle\elll$};
\begin{scope}[shift={(12.4,-3.875)},x=.45cm,y=.45cm]
\draw[old edge] (0,0) -- (1,0);
\draw[kept edge] (0,0) -- (0,1);
\draw[kept edge] (0,1) -- (1,1);
\draw[old edge] (1,0) -- (1,1);
\end{scope}
\node at (13.875,-0.35) {$\pr\rp$};
\begin{scope}[shift={(13.65,-1.675)},x=.45cm,y=.45cm]
\draw[kept edge] (0,0) -- (1,0);
\draw[kept edge] (0,0) -- (0,1);
\draw[kept edge] (0,1) -- (1,1);
\draw[kept edge] (1,0) -- (1,1);
\end{scope}
\node at (13.875,-2.55) {$\elll\llle$};
\begin{scope}[shift={(13.65,-3.875)},x=.45cm,y=.45cm]
\draw[kept edge] (0,0) -- (1,0);
\draw[old edge] (0,0) -- (0,1);
\draw[old edge] (0,1) -- (1,1);
\draw[kept edge] (1,0) -- (1,1);
\end{scope}
\end{tikzpicture}

%% file: figures/01_example.tex
\begin{tikzpicture}[
  x=1.28cm,y=1.28cm,line cap=round,line join=round,
  kept edge/.style={draw=black!88,line width=1.25pt},
  old edge/.style={draw=black!25,line width=.68pt,
    dash pattern=on .7pt off 1.35pt},
  antidiagonal/.style={draw=black!42,line width=.72pt,
    dash pattern=on 2.3pt off 1.8pt},
  vertex/.style={circle,fill=black!75,draw=white,
    line width=.45pt,inner sep=2pt}
]
\newcommand{\continuations}{%
  \draw[kept edge] (-.55,0) -- (0,0);
  \draw[kept edge] (-.55,1) -- (0,1);
  \draw[kept edge] (2,1) -- (2.55,1);
  \draw[kept edge] (2,2) -- (2.55,2);
}
\begin{scope}
  \continuations
  \draw[antidiagonal] (.45,1.55) -- (2.35,-.35);
  \node[anchor=south east,font=\small,text=black!65]
    at (.45,1.55) {$x+y=i$};
  \draw[kept edge] (0,0) -- (2,0) -- (2,2) -- (1,2)
    -- (1,1) -- (0,1) -- cycle;
  \draw[kept edge] (1,0) -- (1,1) -- (2,1);
  \node[vertex,label={[font=\small,xshift=-6pt]above left:$v$}] at (1,1) {};
\end{scope}

\draw[-{Stealth[length=5pt]},draw=black!52,line width=.8pt]
  (3.25,1) -- (4.75,1);

\begin{scope}[shift={(6,0)}]
  \continuations
  \draw[old edge] (1,0) -- (2,0) -- (2,1);
  \draw[kept edge] (0,0) rectangle (1,1);
  \draw[kept edge] (1,1) rectangle (2,2);
  \node[vertex,label={[font=\small,xshift=-6pt]above left:$v$}] at (1,1) {};
\end{scope}

\end{tikzpicture}

%% file: figures/unfolded.tex
\tikzset{every picture/.style={line width=0.75pt}}

\begin{tikzpicture}[
  scale=1.05,
  line cap=round,
  line join=round,
  every node/.style={font=\small},
  vertex/.style={circle,fill=black,inner sep=2.1pt},
  cut edge/.style={draw=black,line width=1.0pt},
  fold edge/.style={draw=black!55,line width=0.85pt,densely dashed},
  face label/.style={inner sep=0pt},
  edge label/.style={inner sep=0pt},
  vertex label/.style={inner sep=0pt,font=\small\bfseries}
]

  \coordinate (v14) at (-2.40, 2.40);
  \coordinate (v24) at ( 2.40, 2.40);
  \coordinate (v23) at ( 2.40,-2.40);
  \coordinate (v13) at (-2.40,-2.40);

  \coordinate (aT) at ( 0, 5.15);
  \coordinate (aR) at ( 5.15, 0);
  \coordinate (aB) at ( 0,-5.15);
  \coordinate (aL) at (-5.15, 0);

  \fill[black!2] (v14)--(v24)--(v23)--(v13)--cycle;
  \fill[black!4] (v14)--(aT)--(v24)--cycle;
  \fill[black!6] (v24)--(aR)--(v23)--cycle;
  \fill[black!4] (v23)--(aB)--(v13)--cycle;
  \fill[black!6] (v13)--(aL)--(v14)--cycle;

  \draw[cut edge]
    (v14)--(aT)--(v24)--(aR)--(v23)--(aB)--(v13)--(aL)--cycle;
  \draw[fold edge] (v14)--(v24)--(v23)--(v13)--cycle;

  \node[face label] at (0,0)       {$\pr\rp\pr\rp$};
  \node[face label] at (0,3.50)    {$\pr 0\rp\llle$};
  \node[face label] at (3.50,0)    {$\elll\pr 1\rp$};
  \node[face label] at (0,-3.50)   {$\pr 0\llle\rp$};
  \node[face label] at (-3.50,0)   {$\pr\elll 1\rp$};

  \node[edge label] at ( 0, 1.90) {$\pr\rp\elll\llle$};
  \node[edge label] at ( 1.90,0)  {$\elll\llle\pr\rp$};
  \node[edge label] at ( 0,-1.90) {$\pr\rp\llle\elll$};
  \node[edge label] at (-1.90,0)  {$\llle\elll\pr\rp$};

  \node[edge label] at (-1.55, 4.05) {$\pr\elll\rp\llle$};
  \node[edge label] at ( 1.55, 4.05) {$\elll\pr\rp\llle$};

  \node[edge label] at ( 4.05, 1.55) {$\elll\pr\rp\llle$};
  \node[edge label] at ( 4.05,-1.55) {$\elll\pr\llle\rp$};

  \node[edge label] at ( 1.55,-4.05) {$\elll\pr\llle\rp$};
  \node[edge label] at (-1.55,-4.05) {$\pr\elll\llle\rp$};

  \node[edge label] at (-4.05,-1.55) {$\pr\elll\llle\rp$};
  \node[edge label] at (-4.05, 1.55) {$\pr\elll\rp\llle$};

  \node[vertex] at (v14) {};
  \node[vertex] at (v24) {};
  \node[vertex] at (v23) {};
  \node[vertex] at (v13) {};
  \node[vertex] at (aT) {};
  \node[vertex] at (aR) {};
  \node[vertex] at (aB) {};
  \node[vertex] at (aL) {};

  \node[vertex label,anchor=south east]
    at ([xshift=-5pt,yshift=5pt]v14) {$\llle\elll\elll\llle$};
  \node[vertex label,anchor=south west]
    at ([xshift=5pt,yshift=5pt]v24) {$\elll\llle\elll\llle$};
  \node[vertex label,anchor=north west]
    at ([xshift=5pt,yshift=-5pt]v23) {$\elll\llle\llle\elll$};
  \node[vertex label,anchor=north east]
    at ([xshift=-5pt,yshift=-5pt]v13) {$\llle\elll\llle\elll$};

  \node[vertex label,anchor=south]
    at ([yshift=6pt]aT) {$\elll\elll\llle\llle$};
  \node[vertex label,anchor=west]
    at ([xshift=6pt]aR) {$\elll\elll\llle\llle$};
  \node[vertex label,anchor=north]
    at ([yshift=-6pt]aB) {$\elll\elll\llle\llle$};
  \node[vertex label,anchor=east]
    at ([xshift=-6pt]aL) {$\elll\elll\llle\llle$};

\end{tikzpicture}

%% file: figures/BB_face_word.tex
\renewcommand{\arraystretch}{1.5}
\(
\begin{array}{rcccccccccccc}
  1s
    & \longrightarrow & 1 & 0 & 1 & 0 & 0 & 1 & 1 & 0 & 1 & 0 & {} \\
  \text{right-most embedding of }1t
    & \longrightarrow & {} & {} & 1 & {} & 0 & {} & 1 & 0 & {} & 0 & {} \\
  \text{the slither word $r_s(t)$}
    & \longrightarrow & $\llle$ & \elll & \pr & \elll & 0 & \llle & 1 & 0 & \llle & 0 & \rp
\end{array}
\)

%% file: figures/subseqbij_2.tex
\begin{tikzpicture}[
    x=1.15cm,
    y=1.15cm,
    line cap=round,
    line join=round,
    box edge/.style={draw=black!85, line width=0.7pt},
    dotted edge/.style={draw=black!75, line width=0.7pt, densely dotted},
    ghost box/.style={
      draw=black!25,
      line width=0.55pt,
      dash pattern=on 2.2pt off 2.2pt
    },
    hatched box/.style={
      pattern={Lines[angle=45, distance=5pt, line width=0.3pt]},
      pattern color=black!38
    },
    blue squiggle/.style={
      draw=pathblue,
      line width=1.4pt,
      decorate,
      decoration={snake, amplitude=1.05pt, segment length=8pt}
    },
    blue edge/.style={draw=pathblue, line width=1.9pt},
    red edge/.style={draw=pathred, line width=1.9pt},
    marked vertex/.style={circle, fill=black!90, inner sep=1.25pt},
    separator/.style={draw=black!28, line width=0.6pt}
  ]

  \definecolor{pathblue}{RGB}{0,114,178}
  \definecolor{pathred}{RGB}{213,94,0}

  \newcommand{\horizontalhatchedbase}{%
    \path[hatched box] (0,0) rectangle (1,1);
    \draw[dotted edge] (0,1) -- (0,0) -- (1,0);
    \draw[box edge] (0,1) -- (2,1);
    \draw[box edge] (1,0) -- (2,0) -- (2,1);
    \draw[box edge] (1,0) -- (1,1);
  }

  \newcommand{\verticalhatchedbase}{%
    \path[hatched box] (0,0) rectangle (1,1);
    \draw[dotted edge] (0,1) -- (0,0) -- (1,0);
    \draw[box edge] (0,1) -- (0,2) -- (1,2);
    \draw[box edge] (1,0) -- (1,2);
    \draw[box edge] (0,1) -- (1,1);
  }

  \begin{scope}[xshift=3.5cm]
    \node at (0,6.3) {\huge $s=\widetilde{s}\,0$};
    \node at (0,2.65) {\huge $s=\widetilde{s}\,1$};
  \end{scope}

  \begin{scope}[xshift=0cm, yshift=5.3cm]
    \horizontalhatchedbase
    \draw[blue squiggle] (0,0) -- (1,1);
    \draw[red edge] (1,1) -- (2,1);
    \node[marked vertex] at (1,1) {};
    \node[marked vertex] at (2,1) {};
  \end{scope}

  \begin{scope}[xshift=4.7cm, yshift=5.3cm]
    \horizontalhatchedbase
    \draw[blue squiggle] (0,-1) -- (1,0);
    \draw[blue edge] (1,0) -- (1,1);
    \draw[red edge] (1,0) -- (2,0) -- (2,1);
    \node[marked vertex] at (1,1) {};
    \node[marked vertex] at (2,1) {};
  \end{scope}

  \begin{scope}[xshift=0.575cm]
    \verticalhatchedbase
    \draw[blue squiggle] (0,0) -- (1,1);
    \draw[red edge] (1,1) -- (1,2);
    \node[marked vertex] at (1,1) {};
    \node[marked vertex] at (1,2) {};
  \end{scope}

  \begin{scope}[xshift=5.275cm]
    \verticalhatchedbase
    \draw[blue squiggle] (-1,0) -- (0,1);
    \draw[blue edge] (0,1) -- (1,1);
    \draw[red edge] (0,1) -- (0,2) -- (1,2);
    \node[marked vertex] at (1,1) {};
    \node[marked vertex] at (1,2) {};
  \end{scope}

  \coordinate (row divider height) at (0,3.217);
  \draw[separator]
    (current bounding box.west |- row divider height) --
    (current bounding box.east |- row divider height);

\end{tikzpicture}

%% file: figures/tikz_equivalences.tex
\begin{tikzpicture}[
  solid edge/.style={
    draw=black,
    line width=1.5pt,
    line cap=round,
    line join=round
  },
  thin edge/.style={
    draw=black,
    line width=0.75pt,
    line cap=round,
    line join=round,
    dash pattern=on 4pt off 2pt
  },
  dotted edge/.style={
    solid edge,
    dash pattern=on 0pt off 7pt
  },
  equivalence/.style={
    <->,
    >=Stealth,
    line width=1.05pt,
    shorten <=2pt,
    shorten >=2pt
  },
  arrow label/.style={
    midway,
    above=3pt,
    inner sep=1pt,
    font=\large\rmfamily\upshape
  }
]

\path[use as bounding box] (-0.2,-0.4) rectangle (14.1,7.75);

\begin{scope}[shift={(0,0.5)}]
  \fill[gray!20] (4.7,6.0) rectangle (5.7,7.0);
  \fill[gray!20] (8.4,6.0) rectangle (9.4,7.0);

  \draw[solid edge] (4.7,6.0) -- (4.7,7.0) -- (5.7,7.0);
  \draw[thin edge] (4.7,6.0) -- (5.7,6.0) -- (5.7,7.0);

  \draw[equivalence] (6.45,6.50)
    -- node[arrow label] {\huge I} (7.65,6.50);

  \draw[solid edge] (9.4,7.0) -- (9.4,6.0) -- (8.4,6.0);
  \draw[thin edge] (9.4,7.0) -- (8.4,7.0) -- (8.4,6.0);

  \node at (5.2,5.5) {\Huge \llle\elll};
  \node at (8.9,5.5) {\Huge \elll\llle};
\end{scope}

\node at (1,3.5) {\Huge \elll\pr};
\node at (5,3.5) {\Huge \pr\elll};

\fill[gray!20] (0,4.0) rectangle (2,5.0);
\draw[dotted edge] (1,5.0) -- (2,5.0) -- (2,4);
\draw[solid edge] (1,5.0) -- (1,4.0) -- (2,4.0);
\draw[solid edge] (0.0,4.0) -- (1,4);
\draw[thin edge] (1,5) -- (0,5) -- (0,4);

\draw[equivalence] (2.55,4.40)
  -- node[arrow label] {\huge II} (3.65,4.40);

\fill[gray!20] (4,4.0) rectangle (6,5.0);
\draw[solid edge] (5,5) -- (4,5.0) -- (4,4) -- (6,4);
\draw[dotted edge] (5,5.0) -- (6,5.0) -- (6,4);
\draw[thin edge] (5,4) -- (5,5);

\node at (9,3.5) {\huge \rp\elll};
\node at (13,3.5) {\huge \elll\rp};

\fill[gray!20] (8,4.0) rectangle (10,5.0);
\fill[gray!20] (12,4.0) rectangle (14,5.0);

\draw[dotted edge] (8,5.0) -- (8,4) -- (9,4);
\draw[solid edge] (8,5.0) -- (9,5.0);
\draw[solid edge] (9,4) -- (9,5.0) -- (10,5);
\draw[thin edge] (9,4) -- (10,4) -- (10,5);

\draw[equivalence] (10.65,4.40)
  -- node[arrow label] {\huge III} (11.75,4.40);

\draw[dotted edge] (12.0,5.0) -- (12.0,4) -- (13,4);
\draw[solid edge] (12.0,5.0) -- (14,5.0)
  -- (14,4) -- (13,4);
\draw[thin edge] (13,5) -- (13,4);

\node at (1.5,0) {\Huge \llle\pr};
\node at (4.5,0) {\Huge \pr\llle};

\fill[gray!20] (1,0.5) rectangle (2,2.5);
\fill[gray!20] (4,0.5) rectangle (5,2.5);

\draw[dotted edge] (1,2.5) -- (2,2.5) -- (2,1.5);
\draw[thin edge] (2,1.5) -- (2,0.5) -- (1,0.5);
\draw[solid edge] (1,2.5) -- (1,0.5);
\draw[solid edge] (1,1.5) -- (2,1.5);

\draw[equivalence] (2.55,1.60)
  -- node[arrow label] {\huge IV} (3.65,1.60);

\draw[solid edge] (4,2.5) -- (4,0.5)
  -- (5,0.5) -- (5,1.5);
\draw[dotted edge] (5,1.5) -- (5,2.5) -- (4,2.5);
\draw[thin edge] (4,1.5) -- (5,1.5);

\node at (9.5,0) {\Huge \rp\llle};
\node at (12.5,0) {\Huge \llle\rp};

\fill[gray!20] (9,0.5) rectangle (10,2.5);
\fill[gray!20] (12,0.5) rectangle (13,2.5);

\draw[dotted edge] (9,1.5) -- (9,0.5) -- (10,0.5);
\draw[solid edge] (10,0.5) -- (10,2.5);
\draw[solid edge] (10,1.5) -- (9,1.5);
\draw[thin edge] (9,1.5) -- (9,2.5) -- (10,2.5);

\draw[equivalence] (10.65,1.60)
  -- node[arrow label] {\huge V} (11.75,1.60);

\draw[dotted edge] (12.0,1.5) -- (12.0,0.5) -- (13,0.5);
\draw[solid edge] (12.0,1.5) -- (12,2.5)
  -- (13,2.5) -- (13,0.5);
\draw[thin edge] (12,1.5) -- (13,1.5);

\end{tikzpicture}

%% file: figures/slither_moves_12.tex
\begingroup
\newcommand{\slAmbient}{%
  \fill[black!6] (0,0) rectangle (1,1);
  \fill[black!6] (1,0) rectangle (2,1);
  \fill[black!6] (2,0) rectangle (3,1);
  \fill[black!6] (2,1) rectangle (3,2);
  \fill[black!6] (3,1) rectangle (4,2);
  \fill[black!6] (3,2) rectangle (4,3);
  \fill[black!6] (3,3) rectangle (4,4);
  \draw[ambient]
      (0,0) -- (0,1)
      (0,0) -- (1,0)
      (0,1) -- (1,1)
      (1,0) -- (1,1)
      (1,0) -- (2,0)
      (1,1) -- (2,1)
      (2,0) -- (2,1)
      (2,0) -- (3,0)
      (3,0) -- (3,1)
      (2,1) -- (3,1)
      (2,1) -- (2,2)
      (2,2) -- (3,2)
      (3,1) -- (3,2)
      (3,1) -- (4,1)
      (4,1) -- (4,2)
      (3,2) -- (4,2)
      (3,2) -- (3,3)
      (4,2) -- (4,3)
      (3,3) -- (4,3)
      (3,3) -- (3,4)
      (3,4) -- (4,4)
      (4,3) -- (4,4);
}
\begin{tikzpicture}[
  x=3.4mm,y=3.4mm,
  font=\small,
  ambient/.style={draw=black!30,line width=.3pt,dash pattern=on 1pt off 1pt},
  slither/.style={draw=black,line width=.8pt,line cap=round,line join=round},
  move/.style={->,>=stealth,line width=.45pt}
]

  \begin{scope}[shift={({0},{0})}]
    \slAmbient
    \draw[slither]
      (0,0) -- (0,1)
      (0,1) -- (1,1)
      (1,1) -- (2,1)
      (2,1) -- (3,1)
      (3,1) -- (3,2)
      (3,2) -- (3,3)
      (3,2) -- (4,2)
      (4,2) -- (4,3)
      (3,3) -- (4,3)
      (3,3) -- (3,4)
      (3,4) -- (4,4)
      (4,3) -- (4,4);
    \node[below=5pt] at (2,0) {$\llle \elll \elll \elll \llle \langle 1 \rangle$};
  \end{scope}

  \begin{scope}[shift={({6},{0})}]
    \slAmbient
    \draw[slither]
      (0,0) -- (0,1)
      (0,1) -- (1,1)
      (1,1) -- (2,1)
      (2,1) -- (3,1)
      (3,1) -- (3,2)
      (3,1) -- (4,1)
      (3,2) -- (3,3)
      (4,1) -- (4,2)
      (4,2) -- (4,3)
      (3,3) -- (4,3)
      (3,3) -- (3,4)
      (3,4) -- (4,4)
      (4,3) -- (4,4);
    \node[below=5pt] at (2,0) {$\llle \elll \elll \elll \langle \llle 1 \rangle$};
  \end{scope}

  \begin{scope}[shift={({12},{0})}]
    \slAmbient
    \draw[slither]
      (0,0) -- (0,1)
      (0,1) -- (1,1)
      (1,1) -- (2,1)
      (2,1) -- (2,2)
      (2,1) -- (3,1)
      (2,2) -- (3,2)
      (3,1) -- (4,1)
      (3,2) -- (3,3)
      (4,1) -- (4,2)
      (4,2) -- (4,3)
      (3,3) -- (4,3)
      (3,3) -- (3,4)
      (3,4) -- (4,4)
      (4,3) -- (4,4);
    \node[below=5pt] at (2,0) {$\llle \elll \elll \langle \elll \llle 1 \rangle$};
  \end{scope}

  \begin{scope}[shift={({18},{0})}]
    \slAmbient
    \draw[slither]
      (0,0) -- (1,0)
      (1,0) -- (1,1)
      (1,1) -- (2,1)
      (2,1) -- (2,2)
      (2,1) -- (3,1)
      (2,2) -- (3,2)
      (3,1) -- (4,1)
      (3,2) -- (3,3)
      (4,1) -- (4,2)
      (4,2) -- (4,3)
      (3,3) -- (4,3)
      (3,3) -- (3,4)
      (3,4) -- (4,4)
      (4,3) -- (4,4);
    \node[below=5pt] at (2,0) {$\elll \llle \elll \langle \elll \llle 1 \rangle$};
  \end{scope}

  \begin{scope}[shift={({24},{0})}]
    \slAmbient
    \draw[slither]
      (0,0) -- (1,0)
      (1,0) -- (2,0)
      (2,0) -- (2,1)
      (2,1) -- (2,2)
      (2,1) -- (3,1)
      (2,2) -- (3,2)
      (3,1) -- (4,1)
      (3,2) -- (3,3)
      (4,1) -- (4,2)
      (4,2) -- (4,3)
      (3,3) -- (4,3)
      (3,3) -- (3,4)
      (3,4) -- (4,4)
      (4,3) -- (4,4);
    \node[below=5pt] at (2,0) {$\elll \elll \llle \langle \elll \llle 1 \rangle$};
  \end{scope}

  \begin{scope}[shift={({30},{0})}]
    \slAmbient
    \draw[slither]
      (0,0) -- (1,0)
      (1,0) -- (2,0)
      (2,0) -- (2,1)
      (2,0) -- (3,0)
      (2,1) -- (2,2)
      (3,0) -- (3,1)
      (2,2) -- (3,2)
      (3,1) -- (4,1)
      (3,2) -- (3,3)
      (4,1) -- (4,2)
      (4,2) -- (4,3)
      (3,3) -- (4,3)
      (3,3) -- (3,4)
      (3,4) -- (4,4)
      (4,3) -- (4,4);
    \node[below=5pt] at (2,0) {$\elll \elll \langle \llle \elll \llle 1 \rangle$};
  \end{scope}

  \begin{scope}[shift={({30},{-8.5})}]
    \slAmbient
    \draw[slither]
      (0,0) -- (1,0)
      (1,0) -- (1,1)
      (1,0) -- (2,0)
      (1,1) -- (2,1)
      (2,0) -- (3,0)
      (2,1) -- (2,2)
      (3,0) -- (3,1)
      (2,2) -- (3,2)
      (3,1) -- (4,1)
      (3,2) -- (3,3)
      (4,1) -- (4,2)
      (4,2) -- (4,3)
      (3,3) -- (4,3)
      (3,3) -- (3,4)
      (3,4) -- (4,4)
      (4,3) -- (4,4);
    \node[below=5pt] at (2,0) {$\elll \langle \elll \llle \elll \llle 1 \rangle$};
  \end{scope}

  \begin{scope}[shift={({24},{-8.5})}]
    \slAmbient
    \draw[slither]
      (0,0) -- (0,1)
      (0,0) -- (1,0)
      (0,1) -- (1,1)
      (1,0) -- (2,0)
      (1,1) -- (2,1)
      (2,0) -- (3,0)
      (2,1) -- (2,2)
      (3,0) -- (3,1)
      (2,2) -- (3,2)
      (3,1) -- (4,1)
      (3,2) -- (3,3)
      (4,1) -- (4,2)
      (4,2) -- (4,3)
      (3,3) -- (4,3)
      (3,3) -- (3,4)
      (3,4) -- (4,4)
      (4,3) -- (4,4);
    \node[below=5pt] at (2,0) {$\langle \elll \elll \llle \elll \llle 1 \rangle$};
  \end{scope}

  \begin{scope}[shift={({18},{-8.5})}]
    \slAmbient
    \draw[slither]
      (0,0) -- (0,1)
      (0,0) -- (1,0)
      (0,1) -- (1,1)
      (1,0) -- (2,0)
      (1,1) -- (2,1)
      (2,0) -- (3,0)
      (3,0) -- (3,1)
      (2,1) -- (3,1)
      (2,1) -- (2,2)
      (2,2) -- (3,2)
      (3,1) -- (4,1)
      (3,2) -- (3,3)
      (4,1) -- (4,2)
      (3,3) -- (3,4)
      (4,2) -- (4,3)
      (3,4) -- (4,4)
      (4,3) -- (4,4);
    \node[below=5pt] at (2,0) {$\langle \elll \elll 1 \elll \llle \llle \rangle$};
  \end{scope}

  \begin{scope}[shift={({12},{-8.5})}]
    \slAmbient
    \draw[slither]
      (0,0) -- (0,1)
      (0,0) -- (1,0)
      (0,1) -- (1,1)
      (1,0) -- (2,0)
      (1,1) -- (2,1)
      (2,0) -- (3,0)
      (3,0) -- (3,1)
      (2,1) -- (3,1)
      (2,1) -- (2,2)
      (2,2) -- (3,2)
      (3,1) -- (4,1)
      (3,2) -- (3,3)
      (4,1) -- (4,2)
      (3,3) -- (4,3)
      (4,2) -- (4,3)
      (4,3) -- (4,4);
    \node[below=5pt] at (2,0) {$\langle \elll \elll 1 \elll \llle \rangle \llle$};
  \end{scope}

  \begin{scope}[shift={({6},{-8.5})}]
    \slAmbient
    \draw[slither]
      (0,0) -- (0,1)
      (0,0) -- (1,0)
      (0,1) -- (1,1)
      (1,0) -- (2,0)
      (1,1) -- (2,1)
      (2,0) -- (3,0)
      (3,0) -- (3,1)
      (2,1) -- (3,1)
      (2,1) -- (2,2)
      (2,2) -- (3,2)
      (3,1) -- (4,1)
      (3,2) -- (4,2)
      (4,1) -- (4,2)
      (4,2) -- (4,3)
      (4,3) -- (4,4);
    \node[below=5pt] at (2,0) {$\langle \elll \elll 1 \elll \rangle \llle \llle$};
  \end{scope}

  \begin{scope}[shift={({0},{-8.5})}]
    \slAmbient
    \draw[slither]
      (0,0) -- (0,1)
      (0,0) -- (1,0)
      (0,1) -- (1,1)
      (1,0) -- (2,0)
      (1,1) -- (2,1)
      (2,0) -- (3,0)
      (3,0) -- (3,1)
      (2,1) -- (3,1)
      (2,1) -- (2,2)
      (2,2) -- (3,2)
      (3,1) -- (3,2)
      (3,2) -- (4,2)
      (4,2) -- (4,3)
      (4,3) -- (4,4);
    \node[below=5pt] at (2,0) {$\langle \elll \elll 1 \rangle \elll \llle \llle$};
  \end{scope}

  \foreach \s in {0,6,12,18,24}{
    \draw[move] ({\s+4.3},2) -- ({\s+5.7},2);
  }
  \draw[move] (32,-2.1) -- (32,-4.1);
  \foreach \s in {6,12,18,24,30}{
    \draw[move] ({\s-.3},-6.5) -- ({\s-1.7},-6.5);
  }
\end{tikzpicture}
\endgroup

%% file: figures/stuck_moves.tex
\begin{tikzpicture}[
  solid edge/.style={
    draw=black,
    line width=1.2pt,
    line cap=round,
    line join=round
  },
  dotted edge/.style={
    solid edge,
    dash pattern=on 0pt off 4pt
  },
  equivalence/.style={
    <->,
    >=Stealth,
    line width=1.05pt,
    shorten <=2pt,
    shorten >=2pt
  }
]



\fill[gray!20] (2,3.0) rectangle (4, 4.0);
\fill[gray!20] (3,3.0) rectangle (4, 5.0);

\draw[dotted edge] (3,5.0) -- (4,5.0) -- (4,4);
\draw[solid edge]  (3,5.0) -- (3,4.0) -- (4,4.0);
\draw[solid edge]  (2.0,4.0) -- (3,4);
\begin{scope}[shift={(3,0)}]
    \fill[gray!20] (2,3.0) rectangle (3, 5.0);
\fill[gray!20] (2,4.0) rectangle (4, 5.0);

\draw[dotted edge] (3,5.0) -- (4,5.0) -- (4,4);
\draw[solid edge]  (3,5.0) -- (3,4.0) -- (4,4.0);
\draw[solid edge]  (3.0,3.0) -- (3,4);
\end{scope}

\end{tikzpicture}

%% file: figures/stuck_cases.tex
\begin{tikzpicture}[
  solid edge/.style={
    draw=black,
    line width=1.2pt,
    line cap=round,
    line join=round
  },
redd edge/.style={
  draw=red,
  line width=1.2pt,
  line cap=round,
  line join=round
},
  dotted edge/.style={
    solid edge,
    dash pattern=on 0pt off 4pt
  },
  equivalence/.style={
    <->,
    >=Stealth,
    line width=1.05pt,
    shorten <=2pt,
    shorten >=2pt
  }
]

\fill[gray!20] (1,3) rectangle (4,4);
\fill[gray!20] (3,3) rectangle (4,5);

\draw[dotted edge] (3,5) -- (4,5) -- (4,4);
\draw[solid edge] (3,5) -- (3,4) -- (4,4);
\draw[solid edge] (1,3) -- (1,4) -- (3,4);
\draw[redd edge] (1,3) -- (3,3) -- (3,4);

\begin{scope}[shift={(5,0)}]
  \fill[gray!20] (0,3) rectangle (4,4);
  \fill[gray!20] (3,3) rectangle (4,5);

  \draw[dotted edge] (3,5) -- (4,5) -- (4,4);
  \draw[solid edge] (3,5) -- (3,4) -- (4,4);
  \draw[solid edge] (1,3) -- (1,4) -- (3,4);
  \draw[solid edge] (0,4) -- (1,4);
  \draw[dotted edge] (1,3) -- (0,3) -- (0,4);
\end{scope}

\end{tikzpicture}

%% file: figures/triangle.tex
\begin{tikzpicture}[
  x=12mm, y=10mm,
  every node/.style={font=\large},scale=0.5
]
  \foreach \r/\entries in {
    0/{1},
    1/{2,1},
    2/{1,3,1},
    3/{1,4,4,1},
  } {
    \foreach \entry [count=\c from 0] in \entries {
      \node at ({\c-\r/2}, {-\r}) {$\entry$};
    }
  }

  \foreach \entry [count=\c from 0] in {0,5,8,5,1} {
    \node at ({\c-2}, -4) {${\entry}$};
  }
\end{tikzpicture}

%% file: bib.bib
@article {NST25,
    AUTHOR = {Nadeau, Philippe and Spink, Hunter and Tewari, Vasu},
     TITLE = {Quasisymmetric divided differences},
   JOURNAL = {S\'em. Lothar. Combin.},
  FJOURNAL = {S\'eminaire Lotharingien de Combinatoire},
    VOLUME = {93B},
      YEAR = {2025},
     PAGES = {Art. 87, 11},
      ISSN = {1286-4889},
   MRCLASS = {05E05 (05E10)},
  MRNUMBER = {4972548},
}

@misc{BGNST25,
      title={Equivariant quasisymmetry and noncrossing partitions}, 
      author={Nantel Bergeron and Lucas Gagnon and Philippe Nadeau and Hunter Spink and Vasu Tewari},
      year={2025},
      eprint={2504.15234},
      archivePrefix={arXiv},
      primaryClass={math.CO},
      url={https://arxiv.org/abs/2504.15234}, 
      note={arXiv:2504.15234},
}

@article {BGNST26,
    AUTHOR = {Bergeron, Nantel and Gagnon, Lucas and Nadeau, Philippe and
              Spink, Hunter and Tewari, Vasu},
     TITLE = {The quasisymmetric flag variety: a toric complex on
              noncrossing partitions},
   JOURNAL = {Forum Math. Pi},
  FJOURNAL = {Forum of Mathematics. Pi},
    VOLUME = {14},
      YEAR = {2026},
     PAGES = {Paper No. e18, 45},
      ISSN = {2050-5086},
   MRCLASS = {14M15 (05E05 14N15 52B20)},
  MRNUMBER = {5100456},
       DOI = {10.1017/fmp.2026.10034},
       URL = {https://doi.org/10.1017/fmp.2026.10034},
}

@misc{NST24,
      title={The geometry of quasisymmetric coinvariants}, 
      author={Philippe Nadeau and Hunter Spink and Vasu Tewari},
      year={2024},
      eprint={2410.12643},
      archivePrefix={arXiv},
      primaryClass={math.AG},
      url={https://arxiv.org/abs/2410.12643}, 
      note ={arXiv:2410.12643}
}

@article {NT24,
    AUTHOR = {Nadeau, Philippe and Tewari, Vasu},
     TITLE = {Forest polynomials and the class of the permutahedral variety},
   JOURNAL = {Adv. Math.},
  FJOURNAL = {Advances in Mathematics},
    VOLUME = {453},
      YEAR = {2024},
     PAGES = {Paper No. 109834, 33},
      ISSN = {0001-8708,1090-2082},
   MRCLASS = {05E05 (05A05)},
  MRNUMBER = {4773088},
MRREVIEWER = {Tanja\ Stojadinovi\'c},
       DOI = {10.1016/j.aim.2024.109834},
       URL = {https://doi.org/10.1016/j.aim.2024.109834},
}

@article {BS22,
    AUTHOR = {Bae, Younghan and Schmitt, Johannes},
     TITLE = {{C}how rings of stacks of prestable curves {I}},
      NOTE = {With an appendix by Bae, Schmitt and Jonathan Skowera},
   JOURNAL = {Forum Math. Sigma},
  FJOURNAL = {Forum of Mathematics. Sigma},
    VOLUME = {10},
      YEAR = {2022},
     PAGES = {Paper No. e28, 47},
      ISSN = {2050-5094},
   MRCLASS = {14H10 (14C15 14C17)},
  MRNUMBER = {4430955},
MRREVIEWER = {Reinier\ Kramer},
       DOI = {10.1017/fms.2022.21},
       URL = {https://doi.org/10.1017/fms.2022.21},
}

@article {ACFW13,
    AUTHOR = {Abramovich, Dan and Cadman, Charles and Fantechi, Barbara and
              Wise, Jonathan},
     TITLE = {{E}xpanded degenerations and pairs},
   JOURNAL = {Comm. Algebra},
  FJOURNAL = {Communications in Algebra},
    VOLUME = {41},
      YEAR = {2013},
    NUMBER = {6},
     PAGES = {2346--2386},
      ISSN = {0092-7872,1532-4125},
   MRCLASS = {14D23 (14H10 14N35)},
  MRNUMBER = {3225278},
MRREVIEWER = {Jo\~ao\ Paulo\ Santos},
       DOI = {10.1080/00927872.2012.658589},
       URL = {https://doi.org/10.1080/00927872.2012.658589},
}

@misc{JM18,
      title={Well-quasi-ordering in lattice path matroids}, 
      author={Meenu Mariya Jose and Dillon Mayhew},
      year={2018},
      eprint={1806.10260},
      archivePrefix={arXiv},
      primaryClass={math.CO},
      note = {arXiv:1806.10260}
}

@article {KMSRA18,
    AUTHOR = {Knauer, Kolja and Mart\'inez-Sandoval, Leonardo and Ram\'irez
              Alfons\'in, Jorge Luis},
     TITLE = {On lattice path matroid polytopes: integer points and
              {E}hrhart polynomial},
   JOURNAL = {Discrete Comput. Geom.},
  FJOURNAL = {Discrete \& Computational Geometry. An International Journal
              of Mathematics and Computer Science},
    VOLUME = {60},
      YEAR = {2018},
    NUMBER = {3},
     PAGES = {698--719},
      ISSN = {0179-5376,1432-0444},
   MRCLASS = {52B40 (05B35 52B20)},
  MRNUMBER = {3849147},
MRREVIEWER = {Winfried\ Hochst\"attler},
       DOI = {10.1007/s00454-018-9965-4},
       URL = {https://doi.org/10.1007/s00454-018-9965-4},
}

@unpublished{spectra,
  author = {Malkiewich, Cary},
  title  = {Spectra and Stable Homotopy Theory},
  year   = {2026},
  note   = {Book manuscript, draft dated May 16, 2026},
  url    = {https://people.math.binghamton.edu/malkiewich/spectra_book_draft.pdf}
}

@incollection {Sta77,
    AUTHOR = {Stanley, Richard P.},
     TITLE = {Some combinatorial aspects of the {S}chubert calculus},
 BOOKTITLE = {Combinatoire et repr\'esentation du groupe sym\'etrique
              ({A}ctes {T}able {R}onde {CNRS}, {U}niv. {L}ouis-{P}asteur
              {S}trasbourg, {S}trasbourg, 1976)},
    SERIES = {Lecture Notes in Math.},
    VOLUME = {Vol. 579},
     PAGES = {217--251},
 PUBLISHER = {Springer, Berlin-New York},
      YEAR = {1977},
   MRCLASS = {05-02 (05A19 14M15 14N10)},
  MRNUMBER = {465880},
MRREVIEWER = {S.\ L.\ Kleiman},
}

@article {Sta90,
    AUTHOR = {Stanton, Dennis},
     TITLE = {Unimodality and {Y}oung's lattice},
   JOURNAL = {J. Combin. Theory Ser. A},
  FJOURNAL = {Journal of Combinatorial Theory. Series A},
    VOLUME = {54},
      YEAR = {1990},
    NUMBER = {1},
     PAGES = {41--53},
      ISSN = {0097-3165,1096-0899},
   MRCLASS = {05A17 (05A30 06A07)},
  MRNUMBER = {1051777},
MRREVIEWER = {E.\ Rodney\ Canfield},
       DOI = {10.1016/0097-3165(90)90004-G},
       URL = {https://doi.org/10.1016/0097-3165(90)90004-G},
}

@incollection {Thi01,
    AUTHOR = {Thibon, Jean-Yves},
     TITLE = {Lectures on noncommutative symmetric functions},
 BOOKTITLE = {Interaction of combinatorics and representation theory},
    SERIES = {MSJ Mem.},
    VOLUME = {11},
     PAGES = {39--94},
 PUBLISHER = {Math. Soc. Japan, Tokyo},
      YEAR = {2001},
      ISBN = {4-931469-14-0},
   MRCLASS = {05E05 (20C08 33D80)},
  MRNUMBER = {1862149},
MRREVIEWER = {Andrzej\ Daszkiewicz},
}

@article {HW17,
    AUTHOR = {Huh, June and Wang, Botong},
     TITLE = {Enumeration of points, lines, planes, etc.},
   JOURNAL = {Acta Math.},
  FJOURNAL = {Acta Mathematica},
    VOLUME = {218},
      YEAR = {2017},
    NUMBER = {2},
     PAGES = {297--317},
      ISSN = {0001-5962,1871-2509},
   MRCLASS = {05A15 (52C10)},
  MRNUMBER = {3733101},
MRREVIEWER = {Michael\ J.\ Falk},
       DOI = {10.4310/ACTA.2017.v218.n2.a2},
       URL = {https://doi.org/10.4310/ACTA.2017.v218.n2.a2},
}

@misc{BGS26,
      title={Torus actions on compactified braid varieties and polytopality of subword complexes}, 
      author={Lara Bossinger and Mikhail Gorsky and José Simental},
      year={2026},
      eprint={2609.12414},
      archivePrefix={arXiv},
      primaryClass={math.AG},
      note = {arXiv:2609.12414}
}

@article {PS24,
    AUTHOR = {Pechenik, Oliver and Satriano, Matthew},
     TITLE = {James reduced product schemes and double quasisymmetric
              functions},
   JOURNAL = {Adv. Math.},
  FJOURNAL = {Advances in Mathematics},
    VOLUME = {449},
      YEAR = {2024},
     PAGES = {Paper No. 109737, 28},
      ISSN = {0001-8708,1090-2082},
   MRCLASS = {05E05 (05E14 14N15 55N91)},
  MRNUMBER = {4749343},
MRREVIEWER = {Laura\ Colmenarejo},
       DOI = {10.1016/j.aim.2024.109737},
       URL = {https://doi.org/10.1016/j.aim.2024.109737},
}

@article {JLS21,
    AUTHOR = {Jahn, Dennis and L\"owe, Robert and Stump, Christian},
     TITLE = {{M}inkowski decompositions for generalized associahedra of
              acyclic type},
   JOURNAL = {Algebr. Comb.},
  FJOURNAL = {Algebraic Combinatorics},
    VOLUME = {4},
      YEAR = {2021},
    NUMBER = {5},
     PAGES = {757--775},
      ISSN = {2589-5486},
   MRCLASS = {13F60},
  MRNUMBER = {4339351},
MRREVIEWER = {Jiarui\ Fei},
       DOI = {10.5802/alco.177},
       URL = {https://doi.org/10.5802/alco.177},
}

@article {BS18,
    AUTHOR = {Brodsky, Sarah B. and Stump, Christian},
     TITLE = {Towards a uniform subword complex description of acyclic
              finite type cluster algebras},
   JOURNAL = {Algebr. Comb.},
  FJOURNAL = {Algebraic Combinatorics},
    VOLUME = {1},
      YEAR = {2018},
    NUMBER = {4},
     PAGES = {545--572},
      ISSN = {2589-5486},
   MRCLASS = {13F60 (05E15 20F55)},
  MRNUMBER = {3875076},
MRREVIEWER = {Li\ Li},
       DOI = {10.5802/alco.25},
       URL = {https://doi.org/10.5802/alco.25},
}

@article {PS15,
    AUTHOR = {Pilaud, Vincent and Stump, Christian},
     TITLE = {Brick polytopes of spherical subword complexes and generalized
              associahedra},
   JOURNAL = {Adv. Math.},
  FJOURNAL = {Advances in Mathematics},
    VOLUME = {276},
      YEAR = {2015},
     PAGES = {1--61},
      ISSN = {0001-8708,1090-2082},
   MRCLASS = {05E45 (05E15 05E30 20F55 52B12)},
  MRNUMBER = {3327085},
MRREVIEWER = {Jian-yi\ Shi},
       DOI = {10.1016/j.aim.2015.02.012},
       URL = {https://doi.org/10.1016/j.aim.2015.02.012},
}

@article {PS12,
    AUTHOR = {Pilaud, Vincent and Santos, Francisco},
     TITLE = {The brick polytope of a sorting network},
   JOURNAL = {European J. Combin.},
  FJOURNAL = {European Journal of Combinatorics},
    VOLUME = {33},
      YEAR = {2012},
    NUMBER = {4},
     PAGES = {632--662},
      ISSN = {0195-6698,1095-9971},
   MRCLASS = {52B12 (05C75 52C30)},
  MRNUMBER = {2864447},
       DOI = {10.1016/j.ejc.2011.12.003},
       URL = {https://doi.org/10.1016/j.ejc.2011.12.003},
}

@article {BCDMTY24,
    AUTHOR = {Bazier-Matte, V\'eronique and Chapelier-Laget, Nathan and
              Douville, Guillaume and Mousavand, Kaveh and Thomas, Hugh and
              Yildirim, Emine},
     TITLE = {A{BHY} associahedra and {N}ewton polytopes of
              {$F$}-polynomials for cluster algebras of simply laced finite
              type},
   JOURNAL = {J. Lond. Math. Soc. (2)},
  FJOURNAL = {Journal of the London Mathematical Society. Second Series},
    VOLUME = {109},
      YEAR = {2024},
    NUMBER = {1},
     PAGES = {Paper No. e12817, 27},
      ISSN = {0024-6107,1469-7750},
   MRCLASS = {13F60 (14M25 16G20 52B05)},
  MRNUMBER = {4680199},
MRREVIEWER = {Olga\ Kravchenko},
       DOI = {10.1112/jlms.12817},
       URL = {https://doi.org/10.1112/jlms.12817},
}

@article {BKT14,
    AUTHOR = {Baumann, Pierre and Kamnitzer, Joel and Tingley, Peter},
     TITLE = {Affine {M}irkovi\'c-{V}ilonen polytopes},
   JOURNAL = {Publ. Math. Inst. Hautes \'Etudes Sci.},
  FJOURNAL = {Publications Math\'ematiques. Institut de Hautes \'Etudes
              Scientifiques},
    VOLUME = {120},
      YEAR = {2014},
     PAGES = {113--205},
      ISSN = {0073-8301,1618-1913},
   MRCLASS = {17B37 (16G20 17B67)},
  MRNUMBER = {3270589},
MRREVIEWER = {Csaba\ Sz\'ant\'o},
       DOI = {10.1007/s10240-013-0057-y},
       URL = {https://doi.org/10.1007/s10240-013-0057-y},
}

@misc{ABDMPT26,
      title={Flows on graphs with cycles, locally gentle algebras, and the Mutoperhedron}, 
      author={Antoine Abram and Jose Bastidas and Benjamin Dequêne and Alejandro H. Morales and GaYee Park and Hugh Thomas},
      year={2026},
      eprint={2601.08150},
      archivePrefix={arXiv},
      primaryClass={math.CO},
      note = {arXiv:2601.08150}
}

@article {FMP26,
    AUTHOR = {Ferroni, Luis and Morales, Alejandro H. and Panova, Greta},
     TITLE = {Skew shapes, {E}hrhart positivity, and beyond},
   JOURNAL = {Proc. Lond. Math. Soc. (3)},
  FJOURNAL = {Proceedings of the London Mathematical Society. Third Series},
    VOLUME = {133},
      YEAR = {2026},
    NUMBER = {2},
     PAGES = {Paper No. e70204},
      ISSN = {0024-6115,1460-244X},
   MRCLASS = {52B20 (05A17 06A07 52B40)},
  MRNUMBER = {5119669},
       DOI = {10.1112/plms.70204},
       URL = {https://doi.org/10.1112/plms.70204},
}

@article{DKLT96,
  author  = {Duchamp, G{\'e}rard and Krob, Daniel and
             Leclerc, Bernard and Thibon, Jean-Yves},
  title   = {Fonctions quasi-sym{\'e}triques, fonctions
             sym{\'e}triques non commutatives et alg{\`e}bres
             de {Hecke} {\`a} {$q=0$}},
  journal = {C. R. Acad. Sci. Paris S{\'e}r. I Math.},
  volume  = {322},
  number  = {2},
  year    = {1996},
  pages   = {107--112},
  url     = {https://hal.science/hal-00018535}
}

@article {MR95,
    AUTHOR = {Malvenuto, Claudia and Reutenauer, Christophe},
     TITLE = {Duality between quasi-symmetric functions and the {S}olomon
              descent algebra},
   JOURNAL = {J. Algebra},
  FJOURNAL = {Journal of Algebra},
    VOLUME = {177},
      YEAR = {1995},
    NUMBER = {3},
     PAGES = {967--982},
      ISSN = {0021-8693,1090-266X},
   MRCLASS = {05E05 (16W30)},
  MRNUMBER = {1358493},
MRREVIEWER = {Jean-Yves\ Thibon},
       DOI = {10.1006/jabr.1995.1336},
       URL = {https://doi.org/10.1006/jabr.1995.1336},
}

@article{Les47,
  author  = {Lesieur, L{\'e}once},
  title   = {Les probl{\`e}mes d'intersection sur une
             vari{\'e}t{\'e} de {Grassmann}},
  journal = {C. R. Acad. Sci. Paris},
  volume  = {225},
  year    = {1947},
  pages   = {916--917}
}

@article{Pie93,
  author  = {Pieri, Mario},
  title   = {Sul problema degli spazi secanti},
  journal = {Rend. Ist. Lombardo (2)},
  volume  = {26},
  year    = {1893},
  pages   = {534--546}
}

@article{Gia02,
  author  = {Giambelli, Giovanni Zeno},
  title   = {Risoluzione del problema degli spazi secanti},
  journal = {Mem. R. Accad. Sci. Torino (2)},
  volume  = {52},
  year    = {1902},
  pages   = {171--211}
}

@Book{Sch01,
 author    = { Issai Schur },
 title     = { {\"{U}}ber eine {K}lasse von {M}atrizen, die sich einer gegebenen {M}atrix zuordnen lassen },
 publisher = { Dieterich },
 year      = { 1901 },
 address   = { Berlin; G{\"{o}}ttingen }
}

@article{Cau15,
  title={Mémoire sur les fonctions qui ne peuvent
  obtenir que deux valeurs égales et de signes
  contraires par suite des transpositions opérées
  entre les variables qu'elles renferment},
  author={Augustin Louis Cauchy},
  journal={Journal de l'École polytechnique},
  volume={10},
  number={17},
  url = {http://sites.mathdoc.fr/cgi-bin/oeitem?id=OE\_CAUCHY_2_1_91_0},
  pages={29--112},
  year={1815},
  note = {Œuvres, ser. 2, vol. 1, pp. 91--169}
}

@article {QR25,
    AUTHOR = {Quail, Jeremy and Rombach, Puck},
     TITLE = {Positroid envelopes and graphic positroids},
   JOURNAL = {Comb. Theory},
  FJOURNAL = {Combinatorial Theory},
    VOLUME = {5},
      YEAR = {2025},
    NUMBER = {3},
     PAGES = {Paper No. 1, 31},
      ISSN = {2766-1334},
   MRCLASS = {05B35},
  MRNUMBER = {4962115},
MRREVIEWER = {Shaopu\ Zhang},
}

@article {S10,
    AUTHOR = {Schweig, Jay},
     TITLE = {On the {$h$}-vector of a lattice path matroid},
   JOURNAL = {Electron. J. Combin.},
  FJOURNAL = {Electronic Journal of Combinatorics},
    VOLUME = {17},
      YEAR = {2010},
    NUMBER = {1},
     PAGES = {Note 3, 6},
      ISSN = {1077-8926},
   MRCLASS = {05B35 (05E45)},
  MRNUMBER = {2578897},
       DOI = {10.37236/452},
       URL = {https://doi.org/10.37236/452},
}

@article {PPVJ23,
    AUTHOR = {Padrol, Arnau and Pilaud, Vincent and Ritter, Julian},
     TITLE = {Shard polytopes},
   JOURNAL = {Int. Math. Res. Not. IMRN},
  FJOURNAL = {International Mathematics Research Notices. IMRN},
      YEAR = {2023},
    NUMBER = {9},
     PAGES = {7686--7796},
}

@incollection {B84,
    AUTHOR = {Brenti, Francesco},
     TITLE = {Log-concave and unimodal sequences in algebra, combinatorics,
              and geometry: an update},
 BOOKTITLE = {Jerusalem combinatorics '93},
    SERIES = {Contemp. Math.},
    VOLUME = {178},
     PAGES = {71--89},
 PUBLISHER = {Amer. Math. Soc., Providence, RI},
      YEAR = {1994},
      ISBN = {0-8218-0294-1},
   MRCLASS = {05A20 (05-02)},
  MRNUMBER = {1310575},
MRREVIEWER = {Christian\ Krattenthaler},
       DOI = {10.1090/conm/178/01893},
       URL = {https://doi.org/10.1090/conm/178/01893},
}

@incollection {S85,
    AUTHOR = {Stanley, Richard P.},
     TITLE = {The number of faces of simplicial polytopes and spheres},
 BOOKTITLE = {Discrete geometry and convexity ({N}ew {Y}ork, 1982)},
    SERIES = {Ann. New York Acad. Sci.},
    VOLUME = {440},
     PAGES = {212--223},
 PUBLISHER = {New York Acad. Sci., New York},
      YEAR = {1985},
      ISBN = {0-89766-275-},
   MRCLASS = {52A25},
  MRNUMBER = {809209},
MRREVIEWER = {P.\ McMullen},
       DOI = {10.1111/j.1749-6632.1985.tb14556.x},
       URL = {https://doi.org/10.1111/j.1749-6632.1985.tb14556.x},
}

@article {FM20,
    AUTHOR = {Foldes, Stephan and Major, L\'aszl\'o},
     TITLE = {Log-concavity of rows of {P}ascal type triangles},
   JOURNAL = {Util. Math.},
  FJOURNAL = {Utilitas Mathematica},
    VOLUME = {116},
      YEAR = {2020},
     PAGES = {203--210},
      ISSN = {0315-3681},
   MRCLASS = {40A05 (05A15 11B83)},
  MRNUMBER = {4243175},
       DOI = {10.1007/s00013-020-01526-4},
       URL = {https://doi.org/10.1007/s00013-020-01526-4},
}

@article {M69,
    AUTHOR = {Menon, K. V.},
     TITLE = {On the convolution of logarithmically concave sequences},
   JOURNAL = {Proc. Amer. Math. Soc.},
  FJOURNAL = {Proceedings of the American Mathematical Society},
    VOLUME = {23},
      YEAR = {1969},
     PAGES = {439--441},
      ISSN = {0002-9939,1088-6826},
   MRCLASS = {40.10},
  MRNUMBER = {246012},
MRREVIEWER = {F.\ P.\ Cass},
       DOI = {10.2307/2037189},
       URL = {https://doi.org/10.2307/2037189},
}

@article {FS25,
    AUTHOR = {Ferroni, Luis and Schr\"oter, Benjamin},
     TITLE = {Face enumeration for split matroid polytopes},
   JOURNAL = {Combin. Probab. Comput.},
  FJOURNAL = {Combinatorics, Probability and Computing},
    VOLUME = {34},
      YEAR = {2025},
    NUMBER = {4},
     PAGES = {528--544},
      ISSN = {0963-5483,1469-2163},
   MRCLASS = {52B05 (05B35 52B40)},
  MRNUMBER = {4929046},
       DOI = {10.1017/s0963548325000021},
       URL = {https://doi.org/10.1017/s0963548325000021},
}

@misc{BGST26,
      title={The Quasisymmetric {G}rassmannian}, 
      author={Nantel Bergeron and Lucas Gagnon and Hunter Spink and Vasu Tewari},
      year={2026},
      eprint={2604.24903},
      archivePrefix={arXiv},
      primaryClass={math.AG},
      note={arXiv:2604.24903}
}

@article {L02,
    AUTHOR = {Li, Jun},
     TITLE = {A degeneration formula of {GW}-invariants},
   JOURNAL = {J. Differential Geom.},
  FJOURNAL = {Journal of Differential Geometry},
    VOLUME = {60},
      YEAR = {2002},
    NUMBER = {2},
     PAGES = {199--293},
      ISSN = {0022-040X,1945-743X},
   MRCLASS = {14N35},
  MRNUMBER = {1938113},
       URL = {http://projecteuclid.org/euclid.jdg/1090351102},
}

@article {L01,
    AUTHOR = {Li, Jun},
     TITLE = {Stable morphisms to singular schemes and relative stable
              morphisms},
   JOURNAL = {J. Differential Geom.},
  FJOURNAL = {Journal of Differential Geometry},
    VOLUME = {57},
      YEAR = {2001},
    NUMBER = {3},
     PAGES = {509--578},
      ISSN = {0022-040X,1945-743X},
   MRCLASS = {14N35},
  MRNUMBER = {1882667},
MRREVIEWER = {Andreas\ Gathmann},
       URL = {http://projecteuclid.org/euclid.jdg/1090348132},
}

@article {O19,
    AUTHOR = {Oesinghaus, Jakob},
     TITLE = {Quasisymmetric functions and the {C}how ring of the stack of
              expanded pairs},
   JOURNAL = {Res. Math. Sci.},
  FJOURNAL = {Research in the Mathematical Sciences},
    VOLUME = {6},
      YEAR = {2019},
    NUMBER = {1},
     PAGES = {Paper No. 5, 18},
      ISSN = {2522-0144,2197-9847},
   MRCLASS = {14C15 (05E05 14A20 14C17 16T05)},
  MRNUMBER = {3911800},
MRREVIEWER = {Aigli\ Papantonopoulou},
       DOI = {10.1007/s40687-018-0168-7},
       URL = {https://doi.org/10.1007/s40687-018-0168-7},
}

@misc{PS26,
      title={Quasisymmetric {S}chubert calculus}, 
      author={Oliver Pechenik and Matthew Satriano},
      year={2026},
      eprint={2205.12415},
      archivePrefix={arXiv},
      note = {arXiv:2205.12415},
      primaryClass={math.AT},
}

@article {C76,
    AUTHOR = {Chase, P. J.},
     TITLE = {Subsequence numbers and logarithmic concavity},
   JOURNAL = {Discrete Math.},
  FJOURNAL = {Discrete Mathematics},
    VOLUME = {16},
      YEAR = {1976},
    NUMBER = {2},
     PAGES = {123--140},
      ISSN = {0012-365X,1872-681X},
   MRCLASS = {05A05 (05A10)},
  MRNUMBER = {540988},
       DOI = {10.1016/0012-365X(76)90140-0},
       URL = {https://doi.org/10.1016/0012-365X(76)90140-0},
}

@misc{V26,
      title={Log-concavity of subsequence counts of words}, 
      author={Vincent Vatter},
      year={2026},
      eprint={2608.22147},
      archivePrefix={arXiv},
      primaryClass={math.CO},
      note = {arXiv:2608.22147} 
}

@article {L11,
    AUTHOR = {Lam, Thomas},
     TITLE = {Affine {S}chubert classes, {S}chur positivity, and
              combinatorial {H}opf algebras},
   JOURNAL = {Bull. Lond. Math. Soc.},
  FJOURNAL = {Bulletin of the London Mathematical Society},
    VOLUME = {43},
      YEAR = {2011},
    NUMBER = {2},
     PAGES = {328--334},
      ISSN = {0024-6093,1469-2120},
   MRCLASS = {14N15 (05E05 16T30 55P35 55R40)},
  MRNUMBER = {2781213},
MRREVIEWER = {Letterio\ Gatto},
       DOI = {10.1112/blms/bdq110},
       URL = {https://doi.org/10.1112/blms/bdq110},
}

@book {MT91,
    AUTHOR = {Mimura, Mamoru and Toda, Hirosi},
     TITLE = {Topology of {L}ie groups. {I}, {II}},
    SERIES = {Translations of Mathematical Monographs},
    VOLUME = {91},
   EDITION = {Japanese},
 PUBLISHER = {American Mathematical Society, Providence, RI},
      YEAR = {1991},
     PAGES = {iv+451},
      ISBN = {0-8218-4541-1},
   MRCLASS = {55-02 (22-02 55P99 57T20 57T25)},
  MRNUMBER = {1122592},
MRREVIEWER = {V.\ P.\ Snaith},
       DOI = {10.1090/mmono/091},
       URL = {https://doi.org/10.1090/mmono/091},
}

@misc{ELPSW26,
      title={{B}ruhat intervals that are large hypercubes}, 
      author={Jordan Ellenberg and Nicolas Libedinsky and David Plaza and José Simental and Geordie Williamson},
      year={2026},
      eprint={2601.01235},
      archivePrefix={arXiv},
      primaryClass={math.CO},
      note = {arXiv:2601.01235}
}

@article {GH25,
    AUTHOR = {Gao, Yibo and Hodges, Reuven},
     TITLE = {Orbit structures and complexity in {S}chubert and {R}ichardson
              varieties},
   JOURNAL = {S\'em. Lothar. Combin.},
  FJOURNAL = {S\'eminaire Lotharingien de Combinatoire},
    VOLUME = {93B},
      YEAR = {2025},
     PAGES = {Art. 105, 12},
      ISSN = {1286-4889},
   MRCLASS = {14N15 (05E14)},
  MRNUMBER = {4973263},
}

@inproceedings {TW15,
    AUTHOR = {Tsukerman, Emmanuel and Williams, Lauren},
     TITLE = {{B}ruhat interval polytopes},
 BOOKTITLE = {Proceedings of {FPSAC} 2015},
    SERIES = {Discrete Math. Theor. Comput. Sci. Proc.},
     PAGES = {13--24},
 PUBLISHER = {Assoc. Discrete Math. Theor. Comput. Sci., Nancy},
      YEAR = {2015},
   MRCLASS = {52B12 (06A07 20F55)},
  MRNUMBER = {3470848},
MRREVIEWER = {Du\v sko\ Joji\'c},
}

@article {S22,
    AUTHOR = {Safonkin, N. A.},
     TITLE = {Semifinite harmonic functions on the zigzag graph},
      NOTE = {Translation of Funktsional. Anal. i Prilozhen. {\bf 56}
              (2022), no. 3, 52--74.},
   JOURNAL = {Funct. Anal. Appl.},
  FJOURNAL = {Functional Analysis and its Applications},
    VOLUME = {56},
      YEAR = {2022},
    NUMBER = {3},
     PAGES = {199--215},
      ISSN = {0016-2663,1573-8485},
   MRCLASS = {05E05 (46L05)},
  MRNUMBER = {4542839},
       DOI = {10.1134/s0016266322030042},
       URL = {https://doi.org/10.1134/s0016266322030042},
}

@article {NT09,
    AUTHOR = {Novelli, Jean-Christophe and Thibon, Jean-Yves},
     TITLE = {Superization and {$(q,t)$}-specialization in combinatorial
              {H}opf algebras},
   JOURNAL = {Electron. J. Combin.},
  FJOURNAL = {Electronic Journal of Combinatorics},
    VOLUME = {16},
      YEAR = {2009},
    NUMBER = {2},
     PAGES = {Research Paper 21, 46},
      ISSN = {1077-8926},
   MRCLASS = {05E05 (05C05 16T30)},
  MRNUMBER = {2576384},
MRREVIEWER = {\^Angela\ Mestre},
       DOI = {10.37236/87},
       URL = {https://doi.org/10.37236/87},
}

@article {KLT97,
    AUTHOR = {Krob, D. and Leclerc, B. and Thibon, J.-Y.},
     TITLE = {Noncommutative symmetric functions. {II}. {T}ransformations of
              alphabets},
   JOURNAL = {Internat. J. Algebra Comput.},
  FJOURNAL = {International Journal of Algebra and Computation},
    VOLUME = {7},
      YEAR = {1997},
    NUMBER = {2},
     PAGES = {181--264},
      ISSN = {0218-1967,1793-6500},
   MRCLASS = {05E05 (15A15)},
  MRNUMBER = {1433196},
MRREVIEWER = {Arun\ Ram},
       DOI = {10.1142/S0218196797000113},
       URL = {https://doi.org/10.1142/S0218196797000113},
}

@article {GKLLRT95,
    AUTHOR = {Gelfand, Israel M. and Krob, Daniel and Lascoux, Alain and
              Leclerc, Bernard and Retakh, Vladimir S. and Thibon,
              Jean-Yves},
     TITLE = {Noncommutative symmetric functions},
   JOURNAL = {Adv. Math.},
  FJOURNAL = {Advances in Mathematics},
    VOLUME = {112},
      YEAR = {1995},
    NUMBER = {2},
     PAGES = {218--348},
      ISSN = {0001-8708,1090-2082},
   MRCLASS = {05E05 (15A15 16W30)},
  MRNUMBER = {1327096},
MRREVIEWER = {Arun\ Ram},
       DOI = {10.1006/aima.1995.1032},
       URL = {https://doi.org/10.1006/aima.1995.1032},
}

@article {HM25,
    AUTHOR = {Hicks, Angela and McCloskey, Robert},
     TITLE = {A combinatorial perspective on the noncommutative symmetric
              functions},
   JOURNAL = {Electron. J. Combin.},
  FJOURNAL = {Electronic Journal of Combinatorics},
    VOLUME = {32},
      YEAR = {2025},
    NUMBER = {1},
     PAGES = {Paper No. 1.53, 40},
      ISSN = {1077-8926},
   MRCLASS = {05E05},
  MRNUMBER = {4883616},
MRREVIEWER = {Mark\ J.\ Wildon},
       DOI = {10.37236/13129},
       URL = {https://doi.org/10.37236/13129},
}

@misc{GR20,
      title={{H}opf Algebras in Combinatorics}, 
      author={Darij Grinberg and Victor Reiner},
      year={2020},
      eprint={1409.8356},
      archivePrefix={arXiv},
      primaryClass={math.CO},
      note = {arXiv:1409.8356}
}

@article {H16,
    AUTHOR = {Huang, Jia},
     TITLE = {A tableau approach to the representation theory of 0-{H}ecke
              algebras},
   JOURNAL = {Ann. Comb.},
  FJOURNAL = {Annals of Combinatorics},
    VOLUME = {20},
      YEAR = {2016},
    NUMBER = {4},
     PAGES = {831--868},
}

@article {W67,
    AUTHOR = {West, Robert W.},
     TITLE = {Weak {$H$}-spaces},
   JOURNAL = {J. Math. Mech.},
  FJOURNAL = {J. Math. Mech.},
    VOLUME = {17},
      YEAR = {1967},
     PAGES = {421--431},
   MRCLASS = {55.40},
  MRNUMBER = {216502},
MRREVIEWER = {Edgar\ H.\ Brown, Jr.},
       DOI = {10.1512/iumj.1968.17.17023},
       URL = {https://doi.org/10.1512/iumj.1968.17.17023},
}

@book {may12,
    AUTHOR = {May, J. P. and Ponto, K.},
     TITLE = {More concise algebraic topology},
    SERIES = {Chicago Lectures in Mathematics},
      NOTE = {Localization, completion, and model categories},
 PUBLISHER = {University of Chicago Press, Chicago, IL},
      YEAR = {2012},
     PAGES = {xxviii+514},
      ISBN = {978-0-226-51178-8; 0-226-51178-2},
   MRCLASS = {55-02 (16T05 18G55 55P60)},
  MRNUMBER = {2884233},
MRREVIEWER = {Ismar\ Voli\'c},
}

@article {Facial_structures,
    AUTHOR = {An, Suhyung and Jung, JiYoon and Kim, Sangwook},
     TITLE = {Facial structures of lattice path matroid polytopes},
   JOURNAL = {Discrete Math.},
  FJOURNAL = {Discrete Mathematics},
    VOLUME = {343},
      YEAR = {2020},
    NUMBER = {1},
     PAGES = {111628, 11},
}

@article{BB,
 ISSN = {0003486X, 19398980},
 author = {A. Bialynicki-Birula},
 journal = {Annals of Mathematics},
 number = {3},
 pages = {480--497},
 publisher = {[Annals of Mathematics, Trustees of Princeton University on Behalf of the Annals of Mathematics, Mathematics Department, Princeton University]},
 title = {Some Theorems on Actions of Algebraic Groups},
 volume = {98},
 year = {1973}
}

@article {fulton94,
    AUTHOR = {Fulton, William and Sturmfels, Bernd},
     TITLE = {Intersection theory on toric varieties},
   JOURNAL = {Topology},
  FJOURNAL = {Topology. An International Journal of Mathematics},
    VOLUME = {36},
      YEAR = {1997},
    NUMBER = {2},
     PAGES = {335--353},
}

@misc{strat_poly,
      title={Vertex Posets, Monotone Path Polytopes, and {C}how Polynomials}, 
      author={Mateusz Michałek and Leonid Monin and Botong Wang},
      year={2026},
      eprint={2604.27515},
      archivePrefix={arXiv},
      primaryClass={math.CO},
      url={https://arxiv.org/abs/2604.27515}, 
      note = {arXiv:2604.27515}
}

@misc{struct_prop,
      title={Structural properties of {B}ia{\l}ynicki-{B}irula decompositions}, 
      author={Teddy Gonzales and Chayim Lowen},
      year={2026},
      eprint={2604.27634},
      archivePrefix={arXiv},
      primaryClass={math.AG}, 
      note = {arXiv:2604.27634}
}

@book{cox2011toric,
  title={Toric Varieties},
  author={Cox, D. A. and Little, J. B. and Schenck, H. K.},
  series={Graduate studies in mathematics},
  year={2011},
  publisher={American Mathematical Society}
}

@book{fult93toric,
 author = {Fulton, William},
 publisher = {Princeton University Press},
 series = {Annals of Mathematics Studies},
 title = {Introduction to Toric Varieties},
 year = {1993},
}

@book{fult_int_theory,
  author    = {Fulton, William},
  title     = {Intersection Theory},
  edition   = {2},
  publisher = {Springer New York, NY},
  year      = {1998},
  doi       = {10.1007/978-1-4612-1700-8},
  isbn      = {978-0-387-98549-7}
}

@article{int_theory_spherical,
  author  = {Fulton, William and MacPherson, Robert and Sottile, Frank and Sturmfels, Bernd},
  title   = {Intersection Theory on Spherical Varieties},
  journal = {Journal of Algebraic Geometry},
  volume  = {4},
  number  = {1},
  pages   = {181--193},
  year    = {1995}
}

@book{Oxley,
    author = {Oxley, James},
    title = {Matroid Theory},
    publisher = {Oxford University Press},
    year = {2011},
}

@article {Lucas75,
    AUTHOR = {Lucas, Dean},
     TITLE = {Weak maps of combinatorial geometries},
   JOURNAL = {Trans. Amer. Math. Soc.},
  FJOURNAL = {Transactions of the American Mathematical Society},
    VOLUME = {206},
      YEAR = {1975},
     PAGES = {247--279},
}

@article{haase21,
    AUTHOR = {Haase, Christian and Paffenholz, Andreas and Piechnik, Lindsay
              C. and Santos, Francisco},
     TITLE = {Existence of unimodular triangulations---positive results},
   JOURNAL = {Mem. Amer. Math. Soc.},
  FJOURNAL = {Memoirs of the American Mathematical Society},
    VOLUME = {270},
      YEAR = {2021},
    NUMBER = {1321},
}

@article {BR08,
    AUTHOR = {Baker, Andrew and Richter, Birgit},
     TITLE = {Quasisymmetric functions from a topological point of view},
   JOURNAL = {Math. Scand.},
  FJOURNAL = {Mathematica Scandinavica},
    VOLUME = {103},
      YEAR = {2008},
    NUMBER = {2},
     PAGES = {208--242},
      ISSN = {0025-5521,1903-1807},
   MRCLASS = {55R40 (05E05 55P35 57T05)},
  MRNUMBER = {2484353},
MRREVIEWER = {Beno\^it\ Fresse},
       DOI = {10.7146/math.scand.a-15078},
       URL = {https://doi.org/10.7146/math.scand.a-15078},
}

@incollection {Edmonds,
    AUTHOR = {Edmonds, Jack},
     TITLE = {Submodular functions, matroids, and certain polyhedra},
 BOOKTITLE = {Combinatorial {S}tructures and their {A}pplications ({P}roc. {C}algary {I}nternat. {C}onf., {C}algary, {A}lta., 1969)},
     PAGES = {69--87},
 PUBLISHER = {Gordon and Breach, New York-London-Paris},
      YEAR = {1970},
}

@article {GGMS,
    AUTHOR = {Gel\cprime fand, I. M. and Goresky, R. M. and MacPherson, R.
              D. and Serganova, V. V.},
     TITLE = {Combinatorial geometries, convex polyhedra, and {S}chubert
              cells},
   JOURNAL = {Adv. in Math.},
  FJOURNAL = {Advances in Mathematics},
    VOLUME = {63},
      YEAR = {1987},
    NUMBER = {3},
     PAGES = {301--316},
      ISSN = {0001-8708},
   MRCLASS = {14M15 (05B35 22E45 22E70 32C38 32C45 32M10)},
  MRNUMBER = {877789},
MRREVIEWER = {Hiroaki\ Terao},
       DOI = {10.1016/0001-8708(87)90059-4},
       URL = {https://doi.org/10.1016/0001-8708(87)90059-4},
}

@article {GS87,
    AUTHOR = {Gel\cprime fand, I. M. and Serganova, V. V.},
     TITLE = {Combinatorial geometries and the strata of a torus on
              homogeneous compact manifolds},
   JOURNAL = {Uspekhi Mat. Nauk},
  FJOURNAL = {Akademiya Nauk SSSR i Moskovskoe Matematicheskoe Obshchestvo.
              Uspekhi Matematicheskikh Nauk},
    VOLUME = {42},
      YEAR = {1987},
    NUMBER = {2(254)},
     PAGES = {107--134, 287},
}

@book{MichalekSturmfels2021,
  author    = {Micha{\l}ek, Mateusz and Sturmfels, Bernd},
  title     = {Invitation to Nonlinear Algebra},
  series    = {Graduate Studies in Mathematics},
  volume    = {211},
  publisher = {American Mathematical Society},
  address   = {Providence, Rhode Island},
  year      = {2021},
  isbn      = {978-1-4704-5367-1}
}

@book{SGA1,
  shorthand = {SGA1},
  author    = {Grothendieck, Alexander and Raynaud, Mich{\`e}le},
  title     = {Rev{\^e}tements {\'e}tales et groupe fondamental ({SGA} 1)},
  series    = {Lecture Notes in Mathematics},
  volume    = {224},
  publisher = {Springer-Verlag},
  address   = {Berlin},
  year      = {1971},
}

@incollection {hironaka,
    AUTHOR = {Hironaka, Heisuke},
     TITLE = {Triangulations of algebraic sets},
 BOOKTITLE = {Algebraic geometry ({P}roc. {S}ympos. {P}ure {M}ath., {V}ol.
              29, {H}umboldt {S}tate {U}niv., {A}rcata, {C}alif., 1974)},
    SERIES = {Proc. Sympos. Pure Math.},
    VOLUME = {Vol. 29},
     PAGES = {165--185},
 PUBLISHER = {Amer. Math. Soc., Providence, RI},
      YEAR = {1975},
}

@article {Lojasiewicz,
    AUTHOR = {Lojasiewicz, S.},
     TITLE = {Triangulation of semi-analytic sets},
   JOURNAL = {Ann. Scuola Norm. Sup. Pisa Cl. Sci. (3)},
  FJOURNAL = {Annali della Scuola Normale Superiore di Pisa. Classe di Scienze. Serie III},
    VOLUME = {18},
      YEAR = {1964},
     PAGES = {449--474},
}

@book{RAG,
    AUTHOR = {Bochnak, Jacek and Coste, Michel and Roy, Marie-Fran{\c{c}}oise},
     TITLE = {Real algebraic geometry},
    SERIES = {Ergebnisse der Mathematik und ihrer Grenzgebiete (3) [Results
              in Mathematics and Related Areas (3)]},
    VOLUME = {36},
      NOTE = {Translated from the 1987 French original,
              Revised by the authors},
 PUBLISHER = {Springer-Verlag, Berlin},
      YEAR = {1998},
}

@book{hatcher,
  author    = {Allen Hatcher},
  title     = {Algebraic Topology},
  publisher = {Cambridge University Press},
  year      = {2002}
}

@misc{postnikov06,
      title={Total positivity, {G}rassmannians, and networks}, 
      author={Alexander Postnikov},
      year={2006},
      eprint={math/0609764},
      archivePrefix={arXiv},
      primaryClass={math.CO}, 
      note = {arXiv:math/0609764}
}

@book {TTD,
    AUTHOR = {tom Dieck, Tammo},
     TITLE = {Algebraic topology},
    SERIES = {EMS Textbooks in Mathematics},
 PUBLISHER = {European Mathematical Society (EMS), Z\"urich},
      YEAR = {2008},
}

@article{Knutson_Lam_Speyer_2013, title={Positroid varieties: juggling and geometry}, volume={149}, DOI={10.1112/S0010437X13007240}, number={10}, journal={Compositio Mathematica}, author={Knutson, Allen and Lam, Thomas and Speyer, David E.}, year={2013}, pages={1710–1752}}

@article {oh,
    AUTHOR = {Oh, Suho},
     TITLE = {Positroids and {S}chubert matroids},
   JOURNAL = {J. Combin. Theory Ser. A},
  FJOURNAL = {Journal of Combinatorial Theory. Series A},
    VOLUME = {118},
      YEAR = {2011},
    NUMBER = {8},
     PAGES = {2426--2435},
      ISSN = {0097-3165,1096-0899},
   MRCLASS = {05B35 (52B40)},
  MRNUMBER = {2834184},
MRREVIEWER = {Anna\ de Mier},
       DOI = {10.1016/j.jcta.2011.06.006},
       URL = {https://doi.org/10.1016/j.jcta.2011.06.006},
}

@misc{birkhoff,
      title={The topology of {B}irkhoff varieties}, 
      author={Luke Gutzwiller and Stephen A. Mitchell},
      year={2009},
      eprint={0809.1091},
      archivePrefix={arXiv},
      primaryClass={math.AT},
      note = {arXiv:0809.1091}
}

@book {may99,
    AUTHOR = {May, J. P.},
     TITLE = {A concise course in algebraic topology},
    SERIES = {Chicago Lectures in Mathematics},
 PUBLISHER = {University of Chicago Press, Chicago, IL},
      YEAR = {1999},
}

@article{BB_Example,
  title={Some properties of the decompositions of algebraic varieties determined by actions of a torus},
  author={A. {B}ia{\l}ynicki-{B}irula},
  year={1976},
  journal = {Bull. Acad. Polon. Sci. S{\'e}r. Sci. Math. Astronom. Phys.},
  PAGES = {667--674},
  volume = {24},
  no = {9},
}

@InProceedings{chow_basis,
author={F. Rossell{\'o}-Llompart and S. Xamb{\'o}-Descamps},
title={Computing {C}how groups},
booktitle={Algebraic Geometry {S}undance 1986},
year={1988},
publisher={Springer},
pages={220--234},
}

@book{Fulton1997,
  author    = {Fulton, William},
  title     = {{Y}oung Tableaux: With Applications to Representation Theory and Geometry},
  series    = {London Mathematical Society Student Texts},
  volume    = {35},
  publisher = {Cambridge University Press},
  address   = {Cambridge},
  year      = {1997},
  isbn      = {978-0-521-56724-4}
}

@article{LeeMasudaPark2021,
  author  = {Lee, Eunjeong and Masuda, Mikiya and Park, Seonjeong},
  title   = {Toric {B}ruhat interval polytopes},
  journal = {Journal of Combinatorial Theory, Series A},
  volume  = {179},
  pages   = {105387},
  year    = {2021},
  doi     = {10.1016/j.jcta.2020.105387}
}

@article{LeeMasudaPark2023,
  author  = {Lee, Eunjeong and Masuda, Mikiya and Park, Seonjeong},
  title   = {Toric {R}ichardson varieties of {C}atalan type and
             {W}edderburn--{E}therington numbers},
  journal = {European Journal of Combinatorics},
  volume  = {108},
  pages   = {103617},
  year    = {2023},
  doi     = {10.1016/j.ejc.2022.103617}
}

@article{CanSaha2025,
  author  = {Can, Mahir Bilen and Saha, Pinakinath},
  title   = {Toric {R}ichardson varieties},
  journal = {Communications in Algebra},
  volume  = {53},
  number  = {5},
  pages   = {1770--1790},
  year    = {2025},
  doi     = {10.1080/00927872.2024.2422028}
}

@misc{GKSB26,
  author        = {Gorsky, Eugene and Kim, Soyeon and
                   Sherman-Bennett, Melissa},
  title         = {Unexpected toric {R}ichardson varieties},
  year          = {2026},
  eprint        = {2603.29260},
  archivePrefix = {arXiv},
  primaryClass  = {math.AG},
  note = {arXiv:2603.29260}
}

@incollection {Gessel1984,
    AUTHOR = {Gessel, Ira M.},
     TITLE = {Multipartite {$P$}-partitions and inner products of skew
              {S}chur functions},
 BOOKTITLE = {Combinatorics and algebra ({B}oulder, {C}olo., 1983)},
    SERIES = {Contemp. Math.},
    VOLUME = {34},
     PAGES = {289--317},
 PUBLISHER = {Amer. Math. Soc., Providence, RI},
      YEAR = {1984},
      ISBN = {0-8218-5029-6},
   MRCLASS = {05A17 (20C30)},
  MRNUMBER = {777705},
MRREVIEWER = {J.\ D\'esarm\'enien},
       DOI = {10.1090/conm/034/777705},
       URL = {https://doi.org/10.1090/conm/034/777705},
}

@incollection {brion_aut,
    AUTHOR = {Brion, Michel},
     TITLE = {Automorphism groups of almost homogeneous varieties},
 BOOKTITLE = {Facets of algebraic geometry. {V}ol. {I}},
    SERIES = {London Math. Soc. Lecture Note Ser.},
    VOLUME = {472},
     PAGES = {54--76},
 PUBLISHER = {Cambridge Univ. Press, Cambridge},
      YEAR = {2022},
}

@article {TOTARO_chow,
    AUTHOR = {Totaro, Burt},
     TITLE = {{C}how groups, {C}how cohomology, and linear varieties},
   JOURNAL = {Forum Math. Sigma},
  FJOURNAL = {Forum of Mathematics. Sigma},
    VOLUME = {2},
      YEAR = {2014},
     PAGES = {Paper No. e17, 25},
      ISSN = {2050-5094},
   MRCLASS = {14C15 (14F42 14M20 14M25)},
  MRNUMBER = {3264256},
MRREVIEWER = {Jinhyun\ Park},
       DOI = {10.1017/fms.2014.15},
       URL = {https://doi.org/10.1017/fms.2014.15},
}

@article{purity_paper,
  author  = {Björner, Anders and Ekedahl, Torsten},
  title   = {On the shape of {B}ruhat intervals},
  journal = {Annals of Mathematics},
  volume  = {170},
  number  = {2},
  pages   = {799--817},
  year    = {2009},
  doi     = {10.4007/annals.2009.170.799},
  url     = {https://doi.org/10.4007/annals.2009.170.799}
}

@book {munkres84,
    AUTHOR = {Munkres, James R.},
     TITLE = {Elements of algebraic topology},
 PUBLISHER = {Addison-Wesley Publishing Company, Menlo Park, CA},
      YEAR = {1984},
     PAGES = {ix+454},
      ISBN = {0-201-04586-9},
   MRCLASS = {55-01},
  MRNUMBER = {755006},
MRREVIEWER = {Christopher\ W.\ Stark},
}

@book{vakil,
    AUTHOR = {Vakil, Ravi},
     TITLE = {The rising sea---foundations of algebraic geometry},
 PUBLISHER = {Princeton University Press, Princeton, NJ},
      YEAR = {2025},
     PAGES = {xxiii+662},
      ISBN = {978-0-691-26866-8; 978-0-691-26867-5; 978-0-691-26868-2},
   MRCLASS = {14-01},
  MRNUMBER = {4942570},
}

@incollection {dual_W,
    AUTHOR = {May, J. P.},
     TITLE = {The dual {W}hitehead theorems},
 BOOKTITLE = {Topological topics},
    SERIES = {London Math. Soc. Lecture Note Ser.},
    VOLUME = {86},
     PAGES = {46--54},
 PUBLISHER = {Cambridge Univ. Press, Cambridge},
      YEAR = {1983},
      ISBN = {0-521-27581-4},
   MRCLASS = {55P10},
  MRNUMBER = {827247},
       DOI = {10.1017/CBO9780511600760.004},
       URL = {https://doi.org/10.1017/CBO9780511600760.004},
}

@article {HSTH,
    AUTHOR = {Hirai, Takeshi and Shimomura, Hiroaki and Tatsuuma, Nobuhiko
              and Hirai, Etsuko},
     TITLE = {Inductive limits of topologies, their direct products, and
              problems related to algebraic structures},
   JOURNAL = {J. Math. Kyoto Univ.},
  FJOURNAL = {Journal of Mathematics of Kyoto University},
    VOLUME = {41},
      YEAR = {2001},
    NUMBER = {3},
     PAGES = {475--505},
      ISSN = {0023-608X},
   MRCLASS = {54H11 (54B15)},
  MRNUMBER = {1878717},
MRREVIEWER = {E.\ Colebunders},
       DOI = {10.1215/kjm/1250517614},
       URL = {https://doi.org/10.1215/kjm/1250517614},
}
